\documentclass[11pt,a4paper]{amsart}
\usepackage{marginnote}
\usepackage[pdftex]{graphicx}
\usepackage[pdftex]{xcolor}
\usepackage{ascmac}
\usepackage{enumitem}  
\usepackage{graphicx}
\usepackage{hyperref}
\usepackage{mathtools,  stmaryrd}
\usepackage{xparse} \DeclarePairedDelimiterX{\Iintv}[1]{\llbracket}{\rrbracket}{\iintvargs{#1}}
\NewDocumentCommand{\iintvargs}{>{\SplitArgument{1}{,}}m}
{\iintvargsaux#1} %
\NewDocumentCommand{\iintvargsaux}{mm} {#1\mkern1.5mu,\mkern1.5mu#2}
\newcommand{\sncomment}[1]{\marginnote{\color{magenta}\setlength{\fboxrule}{1.5pt}%
    \tiny\sffamily\fbox{\parbox{1.3cm}{\raggedright{N: #1}}}}}

\usepackage[normalem]{ulem}

\makeatletter
\newtheorem*{rep@theorem}{\rep@title}
\newcommand{\newreptheorem}[2]{%
\newenvironment{rep#1}[1]{%
 \def\rep@title{#2 \ref{##1}}%
 \begin{rep@theorem}}%
 {\end{rep@theorem}}}
\makeatother

\definecolor{RedOrange}{cmyk} {0, 0.77, 0.87, 0}
\definecolor{RoyalPurple}{cmyk} {0.84, 0.53, 0, 0}
\definecolor{YellowGreen}{cmyk} {0.44, 0, 0.74, 0}
\definecolor{Fuchsia}{cmyk} {0.47, 0.91, 0, 0.08}
\definecolor{Blue}{cmyk} {0.84, 0.53, 0, 0}
\definecolor{BlueViolet}{cmyk} {0.84, 0.53, 0, 0}
\definecolor{Black}{cmyk} {0.75, 0.68, 0.67, 0.9}
\usepackage{xcolor}
\usepackage{verbatim}
\usepackage{amsmath}
\usepackage{amssymb, mathrsfs}
\usepackage{amsbsy}

\usepackage{amscd}
\usepackage{amsthm}
\usepackage[english]{babel}
\usepackage{todonotes}
\usepackage[T1]{fontenc}

\usepackage{color}

\newcommand{\R}{\mathbb{R}}
\newcommand{\SH}{\mathfrak{s}}

\newcommand{\Q}{\mathbb{Q}}

\newcommand{\G}{\mathcal{G}}
\newcommand{\N}{\mathbb{N}}
\newcommand{\e}{\varepsilon}
\newcommand{\ep}{\varepsilon}

\newcommand{\E}{\mathbb{E}}
\newcommand{\Z}{\mathbb{Z}}
\newcommand{\sZ}{\mathbb{Z}}

\renewcommand{\P}{\mathbb{P}}

\newcommand{\kD}{\mathcal{D}}

\newcommand{\kA}{\mathcal{A}}

\newcommand{\kC}{\mathcal{C}}
\newcommand{\cD}{\mathcal{D}}
\newcommand{\sC}{\mathscr{C}}

\newcommand{\kF}{\mathcal{F}}
\newcommand{\kG}{\mathcal{G}}

\newcommand{\kE}{\mathcal{E}}
\newcommand{\cE}{\mathcal{E}}

\newcommand{\rH}{\mathrm{H}}
\newcommand{\rL}{\mathrm{L}}

\newcommand{\rN}{\mathrm{N}}

\newcommand{\fp}{\mathfrak{p}}

\newcommand{\rmB}{\mathrm{B}}

\newcommand{\lin}{\left[\kern-0.15em\left[}
\newcommand{\rin} {\right]\kern-0.15em\right]}
\newcommand{\linf}{[\kern-0.15em [}
\newcommand{\rinf} {]\kern-0.15em ]}
\newcommand{\ilin}{\left]\kern-0.15em\left]}
\newcommand{\irin} {\right[\kern-0.15em\right[}

\def\ben#1{\begin{equation}#1\end{equation}}

\def\al#1{\begin{align*}#1\end{align*}}
\def\aln#1{\begin{align}#1\end{align}}

\usepackage{constants}

\newconstantfamily{c}{symbol=c}
\newconstantfamily{a}{symbol=\alpha}
\newcommand{\secno}[1]{\thesection.\arabic{#1}}
\newconstantfamily{kE}{
symbol=\mathcal{E},
format=\secno,
reset={section}
}
\renewcommand{\tilde}{\widetilde}

\newtheorem{lem}{Lemma}[section]

\newtheorem{claim}[lem]{Claim}
\newtheorem{prop}[lem]{Proposition}
\newtheorem{thm}[lem]{Theorem}

\newtheorem{cor}[lem]{Corollary}

\newtheorem {defin}[lem] {Definition}
\newtheorem {rem}[lem] {Remark}

\newcounter{assu}
\usepackage{color}
\definecolor{lilas}{RGB}{182, 102, 210}
\newcommand{\SN}{\color{red}}

\numberwithin{equation}{section}

\def\ben#1{\begin{equation}#1\end{equation}}

\DeclareMathOperator{\Diam}{Diam}

\title[On upper tail large deviation rate function for chemical distance]
{On upper tail large deviation rate function  for chemical distance in supercritical percolation}
\date{\today}
\author{Barbara Dembin} 
\address[Barbara Dembin]
{ETH, Zurich, Switzerland}
\email{barbara.dembin(at)math.ethz.ch}

\author{Shuta Nakajima} 
\address[Shuta Nakajima]
{Meiji University, Kanagawa, Japan}
\email{njima(at)meiji.ac.jp}

\begin{document}

\maketitle

\begin{abstract}
We consider the supercritical bond percolation on $\Z^d$ and study the graph distance on the percolation graph called the chemical distance. It is well-known that there exists a deterministic constant $\mu(x)$ such that the chemical distance $\cD(0,nx)$ %(that can also be interpreted as the time to go from $0$ to $nx$)
between two connected points $0$ and $nx$ grows like $n\mu(x)$.  Garet and Marchand \cite{GaretMarchand} proved that the probability of the upper tail large deviation event $\left\{n\mu(x)(1+\ep)<\cD(0,nx)<\infty\right\} $ decays exponentially with respect to $n$. In this paper, we prove the existence of the rate function for upper tail large deviation when $d\ge 3$ and $\ep>0$ is small enough. Moreover, we show that for any $\ep>0$, the upper tail large deviation event is created by  space-time cut-points (points that all paths from $0$ to $nx$ must cross after a given time) that force the geodesics   to consume more time  by going in a non-optimal direction or by  wiggling considerably. This enables us to express the rate function in regards to space-time cut-points.
%hat forces the geodesics to go in a non-optimal direction or to wiggle before reaching the cut-point, where the geodesics lose extra time. This enables us to express the rate function in terms of the rate function for a space-time cut-point.
\end{abstract}
\section{Introduction and main results}

The model of {\bf Bernoulli bond percolation} was introduced by Broadbent and Hammersley in 1957 to model the circulation of water in a porous medium \cite{MR91567}. Since then, various models of percolation have been developed, such as level-sets of random fields, Boolean and Voronoi percolation, random interlacements, etc. 
 A common feature among all of these models is that they involve a random subset of the underlying space, which is characterized by a parameter. This random subset grows in size as the parameter increases. The primary focus is on studying the connectivity of the random subset, specifically identifying the parameter values at which the random subset has an unbounded connected component. Of particular interest is the existence of a phase transition that occurs in the parameter when transitioning from a regime where there is no infinite connected component to a regime where one exists. 
 
The study of these problems has uncovered numerous properties and techniques, giving rise to the so-called {\bf percolation theory}. This theory holds significant importance in probability theory and statistical physics, especially regarding its relationship with other statistical physics models like lattice spin models, interacting particle systems, and random walks in random environments. These connections enable the successful application of powerful results and arguments from percolation theory to other fields, even in the absence of a direct link. Refer to \cite{grimmett:percolation} for comprehensive background information and known results on percolation theory.

 In the model of  Bernoulli bond percolation, each edge of $(\Z^d,\E^d)$ is independently removed with probability $1-p$. When $p$ increases the set of remaining edges increases.
 In particular,  Aizenman--Kesten--Newman \cite{MR901151} proved the existence of a phase transition  in the sense that there exists a special parameter $p_c(d)\in(0,1)$ such that below $p_c(d)$, there is no infinite connected component; above $p_c(d)$, there exists a unique one. In Bernoulli percolation, the subcritical regime $p<p_c(d)$ and the supercritical regime $p>p_c(d)$ are both well-studied and a lot of important results have been obtained rigorously. %In particular, there is an exponential decay of the probability of long connections in the subcritical regime. In the supercritical regime, there is exponential decay of the probability of long connections outside the infinite connected component. 
  On the other hand, the study of {\bf critical percolation} (i.e. the case $p=p_c(d)$) have appeared considerably more complicated. One of the major challenges in percolation is to exclude the existence of an infinite connected component in critical percolation. In $d=2$, Kesten \cite{MR575895} proved that $p_c(2)=1/2$ and there is no infinite connected component at $p=p_c(2)$. Additionally, planar critical percolation was hypothesized to be associated with Schramm-Loewner evolution with $\kappa=6$. As a result, this area has been an active research field for the past three decades. For $d\geq 11$, the absence of an infinite cluster in critical percolation has been conclusively proven through the use of lace expansion and mean-field approximation techniques \cite{hara1990mean,10.1214/17-EJP56}. Nonetheless, for $d\in\{3, ..., 10\}$, this problem remains unsolved and is regarded as one of the most significant conjectures in probability theory. 
  
     From the viewpoint of graph theory, it is also important to study the graph distance in a graph since the graph distance gives even more indepth information on geometric structures of the graph compared to connectivity alone. In a percolation graph, the graph distance is referred to as the chemical distance, a name coined by a physicist S. Alexander. A central theme in this field is the asymptotic behavior of the chemical distance between two distant endpoints. This question has been explored for various models, including those with long-range interactions \cite{AntalPisztora,Biskup,Drewitz,Ding,Baumler}. % The study of chemical distance appears considerably more involved than the study of connections. For instance, 
    The extensive study of the chemical distance started in the 1980s. Physicists were particularly interested in the growth exponent for the chemical distance in critical percolation, known as the {\bf chemical distance exponent}.  Despite numerous efforts and advancements in this area, there is still no consensus on the exact values of the chemical distance exponents even for the planar critical percolation. Recently,  Damron--Hanson--Sosoe \cite{DamronHansonSosoe} established an upper bound, showing that the chemical distance exponent is strictly smaller that the growth exponent for the lowest left/right crossing. This latter exponent is related to the three-arm exponent in percolation and its value is  expected to be universal over the 2-dimensional lattices, which was, in particular, rigorously obtained for the triangular lattice. This suggests that the chemical distance exponent is not characterized by a well-known quantity in connectivity. For more on the background and known results on chemical distance, we refer the
reader to \cite{https://doi.org/10.48550/arxiv.1602.00775}.

In Bernoulli percolation on $\Z^d$ lattice, the properties of the chemical distance in supercritical percolation are also of great interest both in mathematics and physics. Garet and Marchand proved a kind of law of large numbers for the chemical distance \cite{GaretMarchandexistencemu}. So, the next step is to study the fluctuations and large deviations. However, there have been few results on the fluctuations of the chemical distance. Recently, Dembin \cite{dembinvar} proved super-concentration for the chemical distance, that is, the variance of the chemical distance incleases at most sublinearly in the distance between the endpoints. However, as far as we know, there are no predictions in mathematical literature regarding the expected behavior of the variance of chemical distance. For the lower bound, the problem whether the variance diverges or not is important but still open for any dimension larger than $1$.

   Moving on to the large deviations, Antal--Pisztora and subsequently Garet--Marchand determined the correct speed of the upper tail large deviations \cite{AntalPisztora,GaretMarchand}. However, previous studies did not establish  the existence of the rate function and a precise description for upper tail large deviations. The aim of this paper is to obtain the rate function by identifying the correct scenario responsible for  upper tail large deviations. We believe that the methods and concepts introduced in this paper  provide new insights on upper tail large deviations for other percolation models. 

\subsection{Chemical distance in Bernoulli percolation}\label{section: Model}
 The model of Bernoulli percolation is formally defined as follows. 
 Let $\E^d$ be the set of all pairs of nearest neighbours in $\sZ^d$. We consider i.i.d. Bernoulli random variables $(B_e)_{e\in\E^d}$ of parameter  $p\in [0,1]$. If $B_e=1$, then the edge $e$ is called {\bf open}; otherwise, the edge is called {\bf closed}. Let $\G_p$ be the graph of the open edges:
\[\G_p:=(\Z^d,\{e\in \E^d: B_e=1\})\,.\]
The graph $\G_p$ is called the {\bf percolation graph}.  A path is said to be {\bf open} if the path consists only of open edges. We write $x\leftrightarrow y$ if $x$ and $y$ are connected in $\kG_p$. This model exhibits a phase transition. Indeed, when $d\geq 2$, there exists a critical parameter $p_c(d)\in(0,1)$ such that for $p>p_c(d)$ ({\bf supercritical regime}), there {almost surely} exists a unique
infinite open cluster $\sC_\infty$ in $\G_p$. In contrast, for $p<p_c(d)$ ({\bf subcritical regime}), there are no infinite open clusters. We refer to  \cite{grimmett:percolation} for general backgrounds and known results on Bernoulli percolation. \textbf{Throughout the paper, we always assume $p>p_c(d)$}.  We denote by $\cD^{\G_p}$ the graph distance on the graph $\G_p$, i.e. for $x,y\in \Z^d$,
\begin{equation}
\cD^{\G_p}(x,y):=\inf\Big\{\,|r|: \text{$r$ is a path from $x$ to $y$ in $\G_p$}\,\Big\}\,,
\end{equation}
where $|r|$ denotes the number of edges in the path $r$ and we use the convention  $\inf\emptyset=+\infty$. In particular, if $x$ and $y$ are not connected in $\kG_p$, then  $\kD^{\G_p}(x,y)=\infty$. For later purposes, we extend the chemical distance to a function of real vectors by setting $\kD^{\G_p}(x,y)=\kD^{\G_p}(\lfloor x\rfloor,\lfloor y\rfloor )$ for $x,y\in \R^d$, where $\lfloor\cdot \rfloor $ stands for the floor function. This $\kD^{\G_p}$ is the so-called {\bf chemical distance}.  We interpret the quantity  $\cD^{\G_p}(x,y)$ as the time to go from $x$ to $y$ (see Section \ref{sec:link}).  Any path achieving the infimum in $\cD^{\G_p}(x,y)$ is called a {\bf geodesic}.  Note that any geodesic is  self-avoiding.\\

\paragraph{\bf Time constant}
 %The existence of the time constant for the chemical distance was first obtained in a more general context of stationary integrable ergodic fields by
 Garet and Marchand   \cite{GaretMarchandexistencemu} obtained an asymptotic behavior of $\kD^{\kG_p}(0,nx)$ as $n\to\infty$: % They prove the existence of an asymptotic speed: 
  {for any $p>p_c(d)$,} there exists a deterministic {norm} $\mu:\R^d\rightarrow [0,+\infty)$ such that
\begin{align}\label{time constant}
\forall x\in\R^d\qquad\lim_{\substack{n\rightarrow \infty\\0\leftrightarrow \lfloor nx\rfloor}} \frac{\cD^{\G_p}(0,nx)}{n}=\mu(x)\qquad\text{a.s.}
\end{align}
%{\color{red} where $\lfloor x\rfloor $ is the closest point of $\Z^d$ in the Euclidean distance from $x$ with a deterministic rule breaking ties.} \bdcomment {it was already defined above as floor function so no need of that }
%\sncomment{We need a continuity of $\mu$ somewhere.}
The function $\mu$ is the so-called {\bf time constant}.  We remark that they obtain \eqref{time constant} in a more general context of stationary integrable ergodic fields. It is well-known that $\mu$ is a norm, in particular convex and  continuous.\\

\paragraph{\bf Upper tail large deviations}
 Our main goal is to study the upper tail large deviation event for $x\in\R^d$:  
 \aln{\label{ldp event}
 \{\mu(x)(1+\xi)n<\cD^{\G_p}(0,nx)<\infty\},\quad\xi>0.
 }
 We here consider the case $x=\mathbf{e}_1=(1,0,\dots,0)$.
 To better understand the decay rate of its probability, let us give an example configuration. %where all the geodesics from $0$ to $n\mathbf{e}_1$ are forced to go to $-\xi n\mathbf{e}_1$ and the event above holds.
{Picture} the configuration where all the horizontal edges along the segment from  $-\xi n\mathbf{e}_1$ to $0$ are open, and all edges sharing exactly one endpoint with that segment are closed  except for the horizontal edge $\langle -( \xi n+1)\mathbf{e}_1, -\xi n\mathbf{e}_1\rangle$ (see Figure \ref{fig1}). Moreover, we further assume that $-\xi n\mathbf{e}_1$ is connected to $n\mathbf{e}_1$. {With} this configuration, the geodesic is forced to go to $-\xi n\mathbf{e}_1$, hence the upper tail large deviation event  occurs with high probability. Note that the probability of such configurations decays exponentially with respect to  $n$. 
 \begin{figure}[!ht]
\def\svgwidth{0.3\textwidth}
 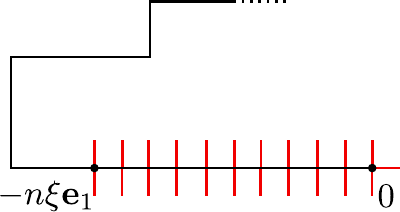
 \caption[fig1]{\label{fig1}Construction of a cut-point at $-\xi n \mathbf{e}_1$. The red edges are closed; the black edges are open.}
\end{figure}

 Garet and Marchand proved in \cite{GaretMarchand} that the probability of the event \eqref{ldp event} decays exponentially with respect to $n$. We aim to prove the existence of the so-called {\bf rate function}  $J_x$ such that  \[\P(\mu(x)(1+\xi)n<\cD^{\G_p}(0,nx)<\infty) =e^{-nJ_x(\xi)+o(n)},\]
and give an explicit description of the rate function. In particular, we prove that on the upper tail large deviation event, typical scenario{s are} similar to the example we gave above (a cut-point where all the geodesics are forced to pass). For the study of cut-point, we should not only record its position but also the time when the geodesics pass through the cut-point (that is the graph distance from $0$ to this cut-point, see \eqref{def:st cutpoint} for a formal definition). We call this a {\bf space-time cut-point}. Space-time cut-points play a central role in studying upper tail large deviation for chemical distance as we will see later.  In the example above, we created a cut-point in a non-optimal location so that  the geodesics  make a detour and {even} go in the opposite direction $-\mathbf{e}_1$. However, there may exist other scenarios where cut-points are located in optimal directions, but the upper tail large deviation event still occurs. % but the chemical distance to the cut-point is very large compared to $\ell_1$ distance.
 For instance, picture a space-time cut-point located around $0$ whose chemical distance to $0$ is larger than $\xi n $ (see Figure \ref{fig2}). Hence, we need to compare the probabilities of all the space-time cut-points leading to the upper tail large deviations and determine the best scenario among them.
\begin{figure}[!ht]
\def\svgwidth{0.4\textwidth}
 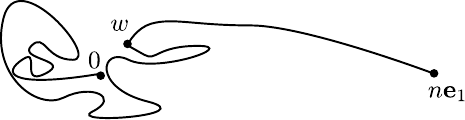
 \caption[fig2]{\label{fig2}Construction of a cut-point $w_n$ close to $0$. The path corresponds to a geodesic from $0$ to $n\mathbf{e}_1$.}
\end{figure}

\subsection{Main results} 
%To prove the existence of the rate function, we introduce the notion of space-time cut-points. We prove that we can relate the upper tail large deviation with the existence of a cut-point along the geodesicsB_ at the beginning or at the end of the path. 
We assume $p>p_c(d)$. Let us give a formal definition of space-time cut-point. For $t\ge 1$ and $z\in\R^ d$, define
\[\rmB_t(z):=\{x\in\Z^d: \kD^{\kG_p}(z,x)\le t\}.\]
For short, we write $\rmB_t$ instead of $\rmB_t(0)$. Let $s>0$ and $x\in\R^ d$. We denote $\Lambda_s(x):=x+[-s,s]^d$. {We define $\alpha_d:=1-\frac{1}{6d}<1$. The value of $\alpha_d$ is due to technical reasons, and its precise value is not important.}
 We define $\kA_{s,x}(n)$ the event that there is a cut-point located around $n x$ whose chemical distance from $0$ is larger than or equal to $ sn $:
\ben{\label{def:st cutpoint}
\kA_{s,x}(n):=\left\{\exists s_n\ge sn\quad \exists x_n\in \Lambda_{n^{\alpha_d}}(nx); \,{\rm B}_{s_n}\setminus {\rm B}_{s_n-1}=\{x_n\}  \right\}.
}
The most technical and innovative part of this paper is the following.
\begin{thm}[Rate function for space-time cut-point]\label{prop:existencelimit} 
 There exists a function $I:[0,\infty)\times \R^d\mapsto [0,+\infty)$ such that for all $(s,x)\in[0,+\infty)\times\R^ d$,
$$\lim_{n\rightarrow\infty}-\frac{1}{n}\log{\P(\kA_{s,x}(n))}=I(s,x),$$
where the convergence is locally uniform on $[0,+\infty)\times\R^ d$.
% There exists a function $I:[0,\infty)\times \R^d\mapsto [0,+\infty)$ such that for any compact set $K\subset [0,\infty)\times \R^d$, 
% $$\lim_{n\rightarrow\infty}\sup_{(s,x)\in K}\left|\frac{1}{n}\log{\P(\kA_{s,x}(n))}+I(s,x)\right|=0.$$
For any  $s>0$ and $x\in\R^d$, $I(s,x)>0$. 
Moreover, the function $I$ is convex and
 homogeneous (i.e. $I(\lambda s,\lambda x)=\lambda I(s,x)$ for $\lambda\geq 0$).
% \begin{equation}\label{eq:propI}
%     \forall s,s'\ge 0\quad \forall x,x'\in\R^d\qquad I(s+s',x+x')\le I(s,x)+I(s',x').
% \end{equation}
\end{thm}
\iffalse
 To prove the theorem, we construct a configuration in $\kA_{s,x}(n+m)$ from two configurations in $\kA_{s,x}(n)$ and $ nx +\kA_{s,x}(m)$,  where $nx +\kA_{s,x}(m)$ is the event of $\kA_{s,x}(m)$ shifted by $nx$ (see \eqref{shifted kA} for a formal definition). This step may seem easy at first look, though we argue in Section \ref{sec:idea} that some configurations make this construction very challenging. In particular, it requires involved technical tools.
 \fi
We define the rate function $J_x$ for $x\in \Z^d$ as
\ben{\label{def J}
J_x(\xi):=\inf\left\{I(s,y):~y\in \R^d,\,s\geq 0 ;\,s+\mu(y-x)\geq (1+\xi)\mu(x)\right\}.
}
\begin{prop}\label{prop:Jcont}
 For any $\xi>0$, $J_x(\xi)>0$ and the function $\xi\mapsto J_x(\xi)$ is  non-decreasing and continuous. 
\end{prop}
 The following is our main result.
\begin{thm}\label{thm:main} For any $x\in \R^d\setminus\{0\}$, there exists $\xi_0=\xi_0(x)>0$ such that for any $\xi\in(0, \xi_0)$, 
\[\lim_{n\rightarrow\infty}\frac{1}{n}\log \P(\mu(x)(1+\xi)n<\cD^{\G_p}(0,nx)<\infty)=- J_x(\xi).\]
%Moreover, the function $J_x$ is continuous. 
\end{thm}
{An explicit expression of $\xi_0(x)$ is given in \eqref{eq:defxi0}.}
\begin{rem}
We shall discuss the expression as defined in \eqref{def J}. Let us examine the point $(s, y) \in [0, +\infty) \times \mathbb{R}^d$ such that
\[s + \mu(y - x) \geq (1 + \xi) \mu(x).\]
Upon the occurrence of event $\kA_{s, y}(n)$, a geodesic connecting the origin and $n\mathbf{e}_1$ arrives later than time $sn$ at a cut-point in proximity to $ny$. According to \eqref{time constant}, the chemical distance from this cut-point to $nx$ is typically about $n\mu(y - x)$. Consequently, it is not possible for the geodesic to arrive at $nx$ earlier than $(1 + \xi) \mu(x)n$.  Consequently, $J_x$ can be interpreted as the minimum cost across all cut-point situations that result in the occurrence of a large deviation event in the upper tail.
\end{rem}
\iffalse 
This represents one possible scenario to create a large deviation event with a unique space-time cut-point around the origin. One may also consider another scenario, where there are two space-time cut-points around both endpoints. Though, thanks to the property of the rate function $I$, we can prove that two cut-points is less likely to occur than a unique cut-point when $\xi\le \xi_0$.
\fi 
\begin{rem}
 In this remark, we provide a rationale for our exclusive focus on the regime $\xi < \xi_0$ in the main result. We will demonstrate that the large deviation event arises from the presence of two cut-points, each connected to one of the endpoints, as detailed in Proposition~\ref{prop:creationcutpoint}. When $\xi \le \xi_0$, the balls associated with these cut-points are in close enough proximity to their corresponding endpoints, ensuring that they do not intersect. Consequently, the scenario can be simplified to a single cut-point. However, when $\xi$ is large, the two cut-points may leverage shared closed edges to increase the probability relative to a single cut-point. Hence, this larger $\xi$ regime might necessitate further studies of more complicated interacting cut-points.
\end{rem}
%\begin{rem}
%To lighten the writing, we focus only on the direction $\mathbf{e}_1$. However, all of our results can be  extended to any direction without difficulty. In particular, the symmetries in regards to the direction $\mathbf{e}_1$ are never used in our proofs.
%\end{rem}
\subsection{2D v.s. 3D and higher}
We expect that the upper tail large deviations behave differently in $d=2$ and  $d\ge 3$. For $d\ge 3$,  only space-time cut-points can create an upper tail large deviation event. Whereas, for $d=2$, we need to consider other scenarios unrelated to space-time cut-points, which create the event. One possible scenario is to add a close wall in the middle that forces the geodesic to deviate from the horizontal line (see figure \ref{fig3}). In  $d=2$, the cost of adding the wall is of order $\exp(cn)$, so it is of the same order as the speed of large deviations. However, in dimensions $d\ge 3$, the cost becomes of order  $\exp(cn^{d-1})$. 
According to this picture, we have the following theorem, which will  be proved in Section~\ref{proof of thm:3}. 

\begin{figure}[b]
\def\svgwidth{0.3\textwidth}
 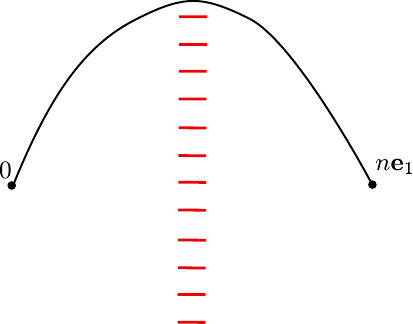
 \caption[fig3]{\label{fig3}Construction of a closed wall. The edges in red are closed and force the geodesic to deviate a lot from the straight line. }
\end{figure} 
\begin{thm}\label{thm:3}
For any $\e>0$ and $\xi>0$, there exists $c>0$ such that for $n$ large enough,%\sncomment{Define $\kD(A,B)$}
\al{
\P(\cD^{\G_p}( \Lambda_{\e n}(0),\Lambda_{\ep n}(n\mathbf{e}_1))>(1+\xi)\mu(\mathbf{e}_1)n) \leq e^{-c n^{d-1}},
}
 where $\cD^{\G_p}(A,B):=\inf_{x\in A,y\in B}\cD^{\G_p}(x,y)$ for $A,B\subset\Z^d$.
% where $\Lambda_s(x):=x+\Lambda_s$.
\end{thm}
The aforementioned theorem emphasizes the distinction between dimension $2$ and dimensions $d\ge3$. In dimensions $d\ge3$, the upper tail large deviation event for the chemical distance between two macroscopic boxes decays significantly more rapidly than the upper tail large deviation event for the chemical distance between two points. This implies that if the balls $\rmB_{\ep n}(0)$ and $\rmB_{\ep n}(n\mathbf{e}_1)$ are both substantially large (as would be the case in a typical supercritical percolation), the probability of the upper tail large deviation event decays much faster than $e^{-cn}$. Consequently, during the upper tail large deviation event, the balls at the endpoints do not exhibit typical behavior and are noticeably small. In this context, we describe the upper tail large deviation as {\bf local}, since it only pertains to edges in close proximity to the endpoints. However, in $d=2$, the balls may maintain their usual volumetric properties, and both global scenarios (e.g. closed walls mentioned earlier) and local scenarios (e.g. space-time cut-points) may coexist, necessitating the identification of the most advantageous scenario among them.
\subsection{Link between First-passage percolation and chemical distance}\label{sec:link}
The model of {\bf First-passage percolation (FPP)} was first introduced by Hammersley and Welsh \cite{HammersleyWelsh} as a model for the spread of a fluid in a porous medium. In the model, we assign a non-negative random variable $\tau_e$ to each edge $e$ such that the family $(\tau_e)_{e\in\E^d}$ is independent and identically distributed with a distribution $F$. The random variable $\tau_e$ may be interpreted as the time needed for the fluid to cross the edge $e$. For any pair of vertices $x,y\in\sZ^d$, the random variable $T(x,y)$, called the {\bf first passage time}, is the shortest time to go from $x$ to $y$. We are interested in the asymptotic behavior of the quantity $T(0,x)$ when $\|x\|$ goes to infinity. Under some integrability conditions on $F$, it is proved that \aln{\label{Kingman thm}
\lim_{n\rightarrow\infty}\frac{1}{n}T(0,nx)=
\inf_{n\in\N}\frac{1}{n}\E T(0,nx):= \mu_F(x)\qquad\text{a.s.},
}
 where $\mu_F$ is  a semi-norm associated to the distribution $F$ called the {\bf time constant for FPP}. Indeed, Cox and Durrett \cite{CoxDurrett} proved  \eqref{Kingman thm}  under necessary
and sufficient integrability conditions on the distribution $F$ in dimension $2$; Kesten \cite{Kesten:StFlour} extended the result to dimensions $d\geq 2$, and proved that $\mu_F$ is exactly a norm if and only if $\P(\tau_e=0)<p_c(d)$. See \cite{auffinger201750} for more detailed backgrounds on FPP.

We here mention an important correspondence between the chemical distance in percolation and FPP. Consider the bond percolation $\kG_p$ introduced in Section~\ref{section: Model}. Let us couple the bond percolation with the time configuration $(\tau_e)_{e\in \E^d}$ in the following way:
\aln{\label{Chemical-FPP}
\tau_e:=
\begin{cases}
1&\text{ if $e$ is open,}\\
\infty&\text{ if $e$ is closed.} 
\end{cases}
}
 In this setting, the infinite cluster made of the edges with passage time $1$ corresponds to the infinite cluster $\sC_\infty$. 
 For $x,y\in\Z^d$, we have
 \[T(x,y):=\inf_{\gamma:x\leftrightarrow  y}\sum_{e\in \gamma}\tau_e=\inf_{\gamma:x\leftrightarrow y\subset \G_p} |\gamma|=\cD^ {\G_p}(x,y),\]
 where the first infimum is taken over all paths from $x$ to $y$ in $\Z^d$; the second infimum is taken over all open paths from $x$ to $y$. Hence, the first passage time for the distribution $G_p=p\delta_1+(1-p)\delta_\infty$ is the same as the chemical distance in percolation. 

\subsection{Related work}
In this section, we discuss related work on the large deviations for chemical distance.
\subsubsection{Large deviations for the chemical distance in supercritical percolation}
The first upper tail large deviation bound on the chemical distance for supercritical percolation was obtained by Antal and Pisztora \cite{AntalPisztora}. They obtained the correct speed in the regime of large $\xi$: for every $p>p_c(d)$, there exists $C\geq 1$ such that 
\begin{equation}\label{eq:antalpisztora}
    \limsup_{\|x\|_1\rightarrow \infty}\frac{1}{\|x\|_1}\log\P\left(  C \|x\|_1\le \cD^{\kG_p}(0,x)<\infty\right)<0\,. 
\end{equation}
\begin{rem}
Antal and Pisztora actually proved a stronger claim \cite[(4.9)]{AntalPisztora}:  there exists $C\ge 1$ and $c>0$ such that 
\begin{equation}\label{eq:antalpisztora2}
\forall x\in\Z^d\quad\forall l\ge C\|x\|_1\quad\P\left(  l\le \cD^{\kG_p}(0,x)<\infty\right)\le \exp(-cl).
\end{equation}
We use this claim several times in this paper.
\end{rem}
The result was later extended to the entire regime of $\xi$ by Garet and Marchand
in \cite{GaretMarchand}: for every $p>p_c(d)$ and $\xi>0$,
\begin{equation}\label{eq:garetmarchand}
   \limsup_{\|x\|_1\rightarrow \infty}\frac{1}{\|x\|_1}\log\P\left((1+\xi) \mu(x)\le \cD^{\kG_p}(0,x)<\infty \right)<0\,.
\end{equation}
In the same paper, they also proved the existence of the rate function for  lower tail large deviation:
 for every $p>p_c(d)$, $\xi>0$, and $x\in\R^d\setminus\{0\}$, the following limit exists and is negative or $-\infty$:
\ben{\label{lower ldp}
\lim_{n\to\infty}\frac{1}{n}\log{\P(\cD^{\G_p}(0,n x)<\mu(nx)(1-\xi))}.
}
\subsubsection{ Upper tail large deviations in FPP with compactly supported distributions}
Basu, Ganguly, and Sly \cite{MR4275334} established the existence of a rate function for upper tail large deviation when the distribution $F$ is both compactly supported and possesses a continuous density. It is important to note the contrast between the behavior of the upper tail large deviation for the chemical distance and that for the FPP with a compactly supported distribution. In the first scenario, the speed of large deviation exhibits a linear order. Conversely, in the second scenario, increasing the first passage time necessitates the involvement of the entire environment, resulting in a volumetric speed, {i.e. of order $n^d$}. The foundation of the proof in \cite{MR4275334} lies in the presence of an underlying limiting metric accountable for the upper tail large deviation. Although the authors did not formalize this limiting metric, it enabled them to derive an involved subadditivity for the upper tail large deviation event. Dembin and Théret \cite{Dembintheretldmaxflow} later formalized the concept of the limiting environment for maximal flow in FPP in order to demonstrate the upper tail large deviation principle.
\iffalse Basu, Ganguly, and Sly \cite{MR4275334} proved the existence of the rate function for upper tail large deviation in the case where the distribution $F$ {is  compactly supported and has} a continuous density. Note that the behavior of the upper tail large deviation for the chemical distance is considerably different from that for the FPP with a compactly supported distribution. In the former case, the speed of large deviation is of linear order.  Whereas in the latter case, to increase the first passage time, the entire environment must be involved, hence the speed is volumic, {i.e. of order $n^d$}. The idea behind the proof in \cite{MR4275334} is the existence of  an underlying limiting metric responsible for the upper tail large deviation. Even though the authors did not formalize  the limiting metric, it helps them to recover a very involved subadditivity for the upper tail large deviation event. The idea of the limiting environment was later formalized by Dembin and Théret \cite{Dembintheretldmaxflow} for a different model to prove the  upper tail large deviation principle for maximal flow in FPP. \fi

\subsubsection{ Upper tail large deviations in FPP with distributions under tail estimates}
 Our approach was initially inspired by the arguments in \cite{cosco2021variational}. We briefly explain the main results and the sketch of the proof in \cite{cosco2021variational} to discuss obstacles when adapting the strategies there to chemical distance. Cosco and Nakajima \cite{cosco2021variational} considered the FPP on $\Z^d$ with distributions satisfying $$\P(\tau_e\geq t)\asymp e^{-\alpha t^r}\text{ as $t\to\infty$ with some constants } r\in (0,d],\,\alpha>0.$$
 They derived a specific rate function for upper tail large deviations, known as the {\bf discrete p-capacity}. In order to demonstrate this, they considered numerous slabs and found a good slab among them where the first passage time between any pair of points $x,y$ inside the slab is approximately equal to the time constant $\mu_F(y-x)$.  This suggests that on the upper tail large deviation event, the first passage time to join one of the endpoints and this good slab is abnormally large. This enables us to replace the large deviation event $\{T(0,n\mathbf{e}_1)>\mu_F(\mathbf{e}_1)(1+\xi)n\}$ with a local event where there are many high weights around the endpoints.  To calculate the rate function, they explicitly estimated the probability of the local event, invoking the Laplace principle.

Nonetheless, these arguments are inapplicable to the current model. Contrary to FPP, there are closed edges that a geodesic is unable to traverse. In FPP, when a geodesic reaches a suitable slab, it is understood that the geodesic moves at a typical speed. {In this study, however,} arriving at an appropriate slab does not {necessarily} guarantee that a geodesic will maintain a typical speed within the slab. In fact, even when the geodesic crosses a slab that possesses beneficial connectivity within its infinite cluster, if the geodesic avoids this infinite cluster, the slab's advantageous connectivity cannot be utilized.

Furthermore, the high weights around the endpoints in FPP are substituted by the presence of space-time cut-points in our situation. Specifically, we cannot employ the Laplace principle, so we resort to the rate function for space-time cut-point instead.
\subsection{Sketch of the proofs}\label{sec:idea}
\subsubsection{Sketch of the proof of Theorem \ref{thm:3}}
     The proof is inspired by the arguments in \cite{cosco2021variational}. Let $N$ be a fixed large integer. We consider  disjoint $\Z^ 2$-slabs of order $n^{d-2}$ and of thickness $N$ that intersect $\Lambda_{\e n}(0),\Lambda_{\ep n}(n\mathbf{e}_1)$. Thanks to a renormalization argument, we can prove that for each slab the probability that the time in the slab between the two boxes $\Lambda_{\e n}(0),\Lambda_{\ep n}(n\mathbf{e}_1)$ is larger than $(1+\xi)\mu(\mathbf{e}_1)$ decays exponentially fast in $n$. Thus, if $\cD^{\G_p}(\Lambda_{\e n}(0),\Lambda_{\ep n}(n\mathbf{e}_1))>(1+\xi)\mu(\mathbf{e}_1)n$, then the time between  $\Lambda_{\e n}(0),\Lambda_{\ep n}(n\mathbf{e}_1)$ on each slab is abnormally large. Since we consider disjoint slabs of order $n^{d-2}$, it follows that the decay of this event is of order $\exp(-cn^{d-1})$.
\subsubsection{ Sketch of the proof of Theorem \ref{prop:existencelimit}} Consider two configurations in $\kA_{s,x}(n)$ and $nx+ \kA_{s,x}(m)$, where
\aln{\label{shifted kA}
y+ \kA_{s,x}(m):=\left\{\exists t\geq s m\quad \exists x_m\in \Lambda_{m^{\alpha_d}}(y+mx): \,{\rm B}_{t}(y)\setminus {\rm B}_{s-1}(y)=\{x_m\}  \right\}.
}
Let $s_n\ge sn ,s_m\ge s m,$ and $w_n,z_m\in \Z^ d$ such that 
$\rmB_{s_n}\setminus \rmB_{s_n-1}=\{w_n\}$ and $\rmB_{s_m}(nx)\setminus \rmB_{s_m-1}(nx)=\{z_m\}$. The two balls come from different time configurations.
To build a configuration in $\kA_{s,x}(n+m)$, we want to join $w_n$ to $x n$ without changing the metric structure so that  $z_m$ is a still cut-point and the chemical distance between $0$ and $z_m$ is larger than $s(n+m)$. A first problem is when the two balls intersect. We can circumvent this problem by finding an appropriate translation of $\rmB_{s_m}(nx)$ not  far from its original position. Such a translation can be found using the ball size control.

The main problem comes from the situations where $w_n$ is not connected to infinity in $\Z^ d\setminus  \rmB_{s_n}$, or where it is connected to infinity but a path to join $w_n$ and $nx$ wiggles a lot. These scenarios may seem pathological, but we believe there are no easy ways to discard these scenarios. Indeed, when $x=0$, we can consider the events where the ball $\rmB_{s_n}$ fills most of the box $\Lambda_{n^{1/d}}(0)$ (e.g. a space-filling spiral) such that $0$ is not connected to $\infty$ in $\Z^ d\setminus \rmB_{s_n}$ or $0$ is connected to $\partial \Lambda_{n^{1/d}}(0)$ but with a very large graph distance (of order $n$). In the first scenario, it is not possible to open edges to create a connection between the cut-point $w_n$ of the first ball to the origin $nx$ of the second ball without modifying the second ball structure. In the second scenario, the cost to join $w_n$ and $nx$ is at least of order exponential in $n$ (in particular, it does not vanish in the limit). 

To solve this issue, we modify the ball structure to create a free line, that is a line intersecting $\rmB_{s_n}$ only at $w_n$ (see Lemma~\ref{lem:drilling}).  The main difficulty of this step is to only slightly modify the ball structure so that the new space-time cut-point is still close to the original one. In particular, the metric structure of the ball has to be preserved as much as possible to ensure that the event $\kA_{s,x}(n)$ still occurs for the modified environment.
%where we call this procedure ``drilling'' in this paper. 
The technical part in this procedure is that the cost of modifying the structure of the ball has to vanish in the limit.
Creating the free lines at $w_n$ and $nx$  gives more freedom to build a connection between these two vertices.
 \subsubsection{Sketch of the proof of Theorem \ref{thm:3}} Consider a couple $(s,x)\in\R_+\times \R^ d$ such that 
 \begin{equation}\label{eq:star}
     s+\mu(x-\mathbf{e}_1)> (1+\xi)\mu(\mathbf{e}_1).
 \end{equation}
    On the event $\kA_{s,x}(n)$, any geodesic between $0$ and $n\mathbf{e}_1$ arrives at a cut-point located close to $x n$ at time larger than $sn$. Moreover, the time needed to go from $nx$ to $n\mathbf{e}_1$ is typically at least $n\mu(x-\mathbf{e}_1)>((1+\xi)\mu(\mathbf{e}_1)-s)n$ thanks to \eqref{time constant} and \eqref{eq:star}. Hence, the event $\{\kD^{\kG_p}(0,n\mathbf{e}_1)>\mu(\mathbf{e}_1)(1+\xi)n\}$ is very likely to occur if the event $\kA_{s,x}(n)$ occurs. Using Theorem \ref{prop:existencelimit}, it follows that 
\[-\liminf_{n\rightarrow\infty}\frac{1}{n}\log \P(\mu(\mathbf{e}_1)(1+\xi)n<\cD^{\G_p}(0,n\mathbf{e}_1)<\infty)\ge I(s,x).\]
Taking the infimum over $(s,x)$ satisfying \eqref{eq:star} on the right hand side, we get 
\begin{equation}\label{eq1}
\begin{split}
   -\liminf_{n\rightarrow\infty}\frac{1}{n}\log\P&(\mu(\mathbf{e}_1)(1+\xi)n<\cD^{\G_p}(0,n\mathbf{e}_1)<\infty) \ge J(\xi).
   \end{split}
\end{equation}

To get the converse inequality, we have to prove that if the upper tail large deviation event occurs, then there exists a space-time cut-point. The idea goes as follows. Consider first the balls $\rmB_s(0)$, $\rmB_t(n\mathbf{e}_1)$ at both endpoints at the moments $s$ and $t$ when their respective size reaches  $n^{7/4}$ (See \eqref{def of s and t} for formal definitions).
 Since the balls are large enough, it is very likely to find $x\in\rmB_s(0)$ and $y\in\rmB_t(n\mathbf{e}_1)$ such that $\cD^{\G_p}(x,y)\le \mu(x-y)+o(n)$. Moreover, since the balls are of size negligible compared to $n^2$, we can make $x$ and $y$ cut-points (see Lemma~\ref{lem:createcutpoin}).
Besides, since we are on the upper tail large deviation event, these cut-points must satisfy
\[\cD^{\G_p}(0,x)+\mu(x-y)+\cD^{\G_p}(y,n\mathbf{e}_1)+o(n)\ge \mu(\mathbf{e}_1)(1+\xi)n\,.\]
Hence, one expects from the Laplace principle and the continuity of $I$ (Theorem~\ref{prop:existencelimit}) that
\begin{equation}
\begin{split}
   \limsup_{n\rightarrow\infty}\frac{1}{n}\log\P&(\mu(\mathbf{e}_1)(1+\xi)n<\cD^{\G_p}(0,n\mathbf{e}_1)<\infty) \notag\\
   &\le \limsup_{n\rightarrow\infty}\frac{1}{n}\log\P\left(\bigcup_{\substack{s,t\ge0,\,x,y\in\R^d:\\s+t+\mu(x-y-\mathbf{e}_1)\ge (1+\xi)\mu(\mathbf{e}_1)}}\kA_{s,x}(n)\cap(n\mathbf{e}_1+\kA_{t,y}(n))\right)\notag\\
   &\leq \sup_{\substack{s,t\ge0,\,x,y\in\R^d:\\s+t+\mu(x-y-\mathbf{e}_1)\ge (1+\xi)\mu(\mathbf{e}_1)}} \limsup_{n\rightarrow\infty}\frac{1}{n}\log\P\left(\kA_{s,x}(n)\cap(n\mathbf{e}_1+\kA_{t,y}(n))\right).\label{upper idea}
   \end{split}
\end{equation}
When $\xi\le \xi_0$, the two space-time cut-points at $0$ and $n\mathbf{e}_1$ occur disjointly (see the proof of Propostion \ref{prop:ULD}). 
Using the property of $I$ (convexity and homogeneity), we  argue that the cost of having a unique cut-point at one endpoint $\kA_{s+t,x-y}(n)$ is smaller than the cost of having two cut-points one at each endpoint of $\kA_{s,x}(n)\cap(n\mathbf{e}_1+\kA_{t,y}(n))$. Therefore, \eqref{upper idea} is further bounded from above by 
\begin{equation}\label{eq2}
\begin{split}
  \limsup_{n\rightarrow\infty}\frac{1}{n}\log\P\left(\bigcup_{\substack{s\ge0,\,x\in\R^d:\\s+\mu(x-\mathbf{e}_1)\ge (1+\xi)\mu(\mathbf{e}_1)}}\kA_{s,x}(n)\right) \le -J(\xi).
   \end{split}
\end{equation}
The result follows combining \eqref{eq1} and \eqref{eq2}.

\subsection{Notations and terminologies} 
In this section, we collect useful notations, terminologies, and claims.
 For $m\ge 1$, set $[m]:=\{1,\cdots, m\}$. 
 
 \noindent \textit{Distances.} We denote by $\|\cdot\|_1,\|\cdot\|_\infty,\|\cdot\|_2$ respectively the $\ell_1$, $\ell_\infty$ and $\ell_2$ norms.  For $A,B\subset \R^d$ and $i\in\{1,2,\infty\}$, we denote
\[\mathrm{d}_i(A,B):=\inf_{x\in A, y\in B}\|x-y\|_i\,.\]
 \noindent \textit{Diameter.} 
 For a set $S \subset \Z^ d$ and $i\in[d]$, we define 
\[\Diam_i(S):=\max_{x,y\in S}|x_i-y_i|\qquad \text{and}\qquad \Diam(S)=\max_{1\le i\le d} \Diam_i(S).\]
 
 \noindent \textit{$\Z^d$-path.} 
  We say that a sequence $\gamma=(v_0,\dots,v_n)\in\Z^d$ is a $\Z^d$-path if for all $i\in [n]$, $\|v_i-v_{i-1}\|_1=1$. Note that a $\Z^d$-path may also be seen as a set of edges $\{\langle v_i,v_{i+1}\rangle, \,0\le i \le n-1\}$. In what follows, a path will implicitly mean a $\Z^ d$-path.
  
\noindent \textit{Cluster.} Given $x,y\in \Z^d$ and $A\subset \Z^d$, if there exists a $\Z^d$-path $\gamma\subset A$ from $x$ to $y$ in $\G_p$, then we write $x\overset{A}{\leftrightarrow} y.$ When $A=\Z^d$, we simply write $x\leftrightarrow y$.  Given $x\in \Z^d$, we define the open cluster containing $x$ as
\aln{
\mathcal{C}(x):=\{y\in \Z^d:~x\leftrightarrow y\}.
}

\noindent \textit{Paths concatenation.}
 Given two paths $\gamma^1,\gamma^2$ such that the end point of $\gamma^1$ is the same as the starting point of $\gamma^2$, we denote by $\gamma^1\oplus \gamma^2$ the concatenation of $\gamma^1$ and $\gamma^2$. Inductively, we can define the concatenation of $n$ paths $\gamma^1\oplus\cdots \oplus \gamma^n$.

 % Given $A\subset \Z^d$ and $B\subset \E^d$, we always regard a procedure in set theory as for vertices, i.e.
%$$A\cup B=A\cup\{x\in \Z^d:~\exists y;~\langle x,y\rangle \in B\},\, A\cap B=A\cap\{x\in \Z^d:~\exists y;~\langle x,y\rangle \in B\}.$$
%where do we use that ? {\SN For example, $\rH^*\cap \gamma_{0,w}$ in page 20} I think if you see $\gamma $ as a set of vertices you don't need that. I said later that $\gamma$ may be seen depending on the context as a set of vertices or edges{\SN Just in case, can we leave it for a while? I will check this carefully after putting this arxiv}}

\noindent \textit{Interior and exterior boundary.}
For a set $\Gamma\subset \Z^d$, we denote by $\mathrm{Int}(\Gamma)$ the set of points in $\Z^d$ enclosed by $\Gamma$:
\[\mathrm{Int}(\Gamma):=\{x\in\Z^d\setminus \Gamma : \text{ $x$ is not connected to $\infty$ in $\Z^d\setminus \Gamma$}\}.\] We define the outer boundary of $\Gamma$ as
\[\partial ^{\rm ext}\Gamma:=\{x\in \Z^d\setminus \Gamma: x\text{ is connected to infinity in $\Z^d\setminus \Gamma$ and has a neighbour in $\Gamma$}\}.\]
 There exists a constant $\kappa_d>1$ such that for every connected set $\Gamma\subset\Z^d$, we have the following discrete isoperimetric inequality (see, for instance,  \cite[Theorem 1]{Coulhon1993})
\begin{equation}\label{eq:isoperimetry}
    |\Gamma|\le \kappa_d |\partial ^{\rm ext}\Gamma|^{d/(d-1)}.
\end{equation}

\noindent \textit{$*$-connected and lattice animals.}
  We say that $x$ and $y$ are $*$-neighbor if $\|x-y\|_\infty=1$. We say that a sequence $\gamma=(v_0,\dots,v_n)\subset \Z^d$ is a $*$-path from $x$ to $y$ in $\Gamma\subset \Z^d$ if $v_0=x\in \Gamma$, $v_n=y$ and for all $i\in[n]$, $v_i\in \Gamma$ and $\|v_i-v_{i-1}\|_\infty=1$. We say that $x$ and $y$ are $*$-connected in $\Gamma$ if such a path exists. We say that $\Gamma$ is $*$-connected if for any $x,y\in \Gamma$, $x$ and $y$ are $*$-connected in $\Gamma$.
 We denote by $\textbf{Animal}_{\mathbf{i}}^k$ the set of $*$-connected components of $\Z^ d$ of size $k$ containing $\mathbf{i}\in \Z^d$.
We have  (see for instance Grimmett \cite[p85]{grimmett:percolation})
\begin{equation}\label{eq:animals}
    |\textbf{Animal}_{\mathbf{i}}^k|\le 7^{dk}.
\end{equation}

\noindent \textit{Lines and hyperplanes.} Given $w\in \Z^d$, $n\ge 1$ and $i\in[d]$, we define 
\al{
&\mathrm {L}_{i}(w):=\{w+k\mathbf{e}_i:~k\in \Z\},\\
&\mathrm {L}_{i}^n(w):=\{w+k\mathbf{e}_i:~|k|\le n\},\\
&\mathrm{H}_{i}(w):=\{w+ \textstyle \sum_{j\neq i} k_j \mathbf{e}_j:~k_j\in \Z \}.
}
Let $P_i$ be the projection on the hyperplane $\mathrm{H}_{i}(0)$.
\subsection{Organisation of the paper}
In Section \ref{sec:2}, we introduce a renormalization procedure and prove Theorem \ref{thm:3}. In Section \ref{sec:3}, we prove the existence of the rate function for space-time cut-point (Theorem \ref{prop:existencelimit}). Finally, in Section \ref{sec:5}, we prove that on the upper tail large deviation event, there exist space-time cut-points and deduce the existence of the rate function for upper tail large deviation (Theorem \ref{thm:main}). As before, we always assume $p>p_c(d)$ in the rest of the paper.\\

\textbf {Acknowledgement.} The authors
are grateful to FIM, for the hospitality during the second author's stay at ETH where the present work was completed. The research of the first author has received funding from the European Research Council (ERC) under the European Union’s Horizon 2020 research and innovation program (grant agreement No
851565). The second author is supported by JSPS KAKENHI 22K20344.

\section{Renormalization}\label{sec:2}
For simplicity, we write $\Lambda_s:=\Lambda_s(0)=[-s,s]^d$.
\subsection{Preliminary on percolation}

The following estimate controls the probability that the density of the infinite cluster is atypically small.
\begin{thm}{\cite[Theorem 2]{Pisztora1996}}\label{thm:density}
Let $\ep>0$. There exists $c>0$ such that for $n$ large enough
\[\P\left( \frac {|\sC_\infty \cap \Lambda_{n}(w)|}{|\Lambda_{n}(w)|}\le \theta(p)-\ep\right)\le e^{-cn^{d-1}},\]
where $\theta(p):=\P(0\in\sC_\infty)$ corresponds to the density of the infinite cluster.
\end{thm}
The following theorem may be seen as a corollary of the previous theorem. It controls the probability that the infinite cluster does not intersect a given box.
\begin{thm}[Holes]\label{thm:holes}
There exists $c>0$ such that for $n\in\N$ large enough,
\[\P(\sC_\infty\cap \Lambda_n=\emptyset)\le e^{-cn^{d-1}}.\]
\end{thm}

The following theorem enables us to control the probability of the existence of two large disjoint clusters (e.g. see   \cite[Lemma 7.104]{grimmett:percolation} for a reference).
\begin{thm}[Distinct clusters]\label{thm:distclus} For any $\ep>0$, there exists $c>0$ such that for $n\in\N$ large enough,
\[\P(\exists \text{two disjoint open clusters of diameter at least $\ep n$ in $\Lambda_n$})\le e^{-cn}.\]
\end{thm}
\subsection {Macroscopic lattice}
 % We do a renormalization of size $N\in\N$.
We define the $N$-box and the enlarged box for $\mathbf{k}\in\Z^d$: $$\underline\Lambda_N(\mathbf{k}):=[-N,N)^ d\cap \sZ^d+2\mathbf{k}N,\quad \underline\Lambda_N'(\mathbf{k}):=\bigcup_{\|\mathbf{i}-\mathbf{k}\|_\infty \le 1}\underline\Lambda_N(\mathbf{i}).$$ 
Then, we have the decomposition of $\Z^d$ as $\sZ^d=\bigsqcup_{\mathbf{i}\in\sZ^d}\underline\Lambda_N(\mathbf{i})$ where $\sqcup$ denotes the disjoint union. 
Thus, for any $x\in\Z^d$, there exists a unique $\mathbf{i}\in\Z^d$ such that $x\in\underline{\Lambda}_N(\mathbf{i})$. With abuse of notation, we use the notation $\underline{\Lambda}_N(x)$ for the box  $\underline{\Lambda}_N(\mathbf{i})$ containing $x$. The sites corresponding to the boxes are the so-called {\bf macroscopic lattice} of sidelength $N$; whereas the standard vertices in $\Z^d$ correspond to the {\bf microcopic lattice}.

\begin{defin}\label{defgoodbox}Let $\ep\in(0,{\frac 1 2})$ and $N\in\N$. We say that a site $\mathbf{i}\in\Z^d$ is {\bf $\ep$-good} if the following hold:
\begin{enumerate}
\item \label{enum1:1} There exists a unique open cluster, denoted by $\sC(\mathbf{i})$, in $\underline\Lambda'_{N}(\mathbf{i})$ with diameter at least $\frac{N}2$;
\item \label{enum1:2} This cluster $\sC(\mathbf{i})$ intersects with  every subbox in $\underline\Lambda'_{N}(\mathbf{i})$ of sidelength $\ep N$;
%\item \label{enum1:3}For all $x,y\in  {\underline\Lambda'_{N}(\mathbf{i})}\cap \sC$ with $|x-y|\geq N/2$, we have \[\cD^{\kG_p}(x,y)\leq (1+\ep)\mu(x-y).\]
\item \label{enum1:3}For all $x,y\in \underline\Lambda'_{N}(\mathbf{i})\cap \sC(\mathbf{i})$, we have \[\cD^{\kG_p}(x,y)\leq \mu(x-y)+\ep N.\]
\end{enumerate}
{Otherwise, if at least one of the three conditions does not hold, we say that the site $\mathbf{i}$ is {\bf $\ep$-bad}.}
\end{defin}
\begin{rem}
Thanks to our definition of $\ep$-good, for any $*$-neighbours $\mathbf{i}$ and $\mathbf{j}$ that are $\ep$-good, the clusters $\sC(\mathbf i)$ and $\sC(\mathbf j)$ are connected in $\underline\Lambda'_{N}(\mathbf{j}) $. Indeed, thanks to Property \eqref{enum1:2}, $\sC(\mathbf i)\cap \underline\Lambda'_{N}(\mathbf{j}) $ has a connected component of diameter at least $N/2$ and the claim follows from Property \eqref{enum1:1} for the site $\mathbf{j}$.
\end{rem}
When there is no confusion, we say good instead of $\ep$-good.  The states of the boxes have a short-range dependence. {To see this,} set 
\begin{equation}\label{def:rho}
    \rho:=\lfloor 10d\mu(\mathbf{e}_1)\rfloor.
\end{equation}
Thanks to the definition of $\ep$-good, for $\mathbf{i,j}\in\Z^ d$ with $\|\mathbf{i}-\mathbf{j}\|_\infty\ge \rho $, the states of the sites $\mathbf{i}$ and $\mathbf{j}$ are independent.
Finally, set
\[\underline \Lambda_N^ \rho(\mathbf{i}):=\bigcup_{\mathbf{j}:\|\mathbf{i}-\mathbf{j}\|_\infty\leq  \rho}\underline \Lambda_N(\mathbf{j}).\]
\begin{prop}\label{prop:path}Let $\ep>0$. For $p>p_c(d)$, there exists $c>0$ such that for $N\in\N$ large enough,
\[\P(\underline\Lambda_N({\bf 0})\text{ is $\ep$-bad})\leq e^{-cN}\,.\]
\end{prop}
\begin{proof}
Thanks to Theorems \ref{thm:holes}, $\P(\sC_\infty\cap\underline\Lambda_N(\mathbf{0})=\emptyset)\leq e^{-cN}$. Hence, we suppose $\sC_\infty\cap\underline\Lambda_N(\mathbf{0})\neq \emptyset.$ Then, we can take a connected component $\sC(\mathbf{0})$ in $\sC_\infty\cap\underline\Lambda'_N(\mathbf{0})$ with diameter at least $N/2$.  Thanks to Theorem \ref{thm:distclus}, this is a unique connected component of diameter at least $N/2$ with probability at least $1-e^ {-cN/2}$. \iffalse Thanks to \ref{thm:distclus},
\begin{equation*}
\begin{split}
&\P(\sC(\mathbf{0})\cap\sC_\infty=\emptyset;~\sC_\infty\cap\underline\Lambda_N(\mathbf{0})\neq \emptyset)\\
&\le \P(\sC_\infty\cap\underline\Lambda_N(\mathbf{0})=\emptyset)+\P(\exists\text{two disjoint open clusters of diameter at least $\ep N$ in  $\underline\Lambda_N'(\mathbf{0})$})\\&\le \exp(-cN).
\end{split}
\end{equation*}
Since $\sC(\mathbf{0})\cap\sC_\infty\neq \emptyset$ implies $\sC(\mathbf{0})\subset \sC_\infty$, \fi
The result follows from the union bound with \eqref{eq:antalpisztora2}, \eqref{eq:garetmarchand}, and Theorem \ref{thm:holes}. 
\end{proof}
The following lemma controls the number of bad boxes.% when the  sidelength of the box is $N=\log^2n$.
\begin{lem}\label{lem:nobadbox} Let $\ep>0$ and
 $N=N_n:=\log^2n:=\lfloor(\log{n})^2\rfloor$. For $n$ large enough,
\aln{
\P\left(\#\{\mathbf{i}\in \Lambda_{n^2}\cap\Z^ d: \mathbf{i}\text{ is $\ep$-bad}\}>n \right)\leq e^{-n\log{n}}.
}
\end{lem}
\begin{proof}
By the union bound, we have
\begin{equation*}
   \P( \#\{\mathbf{i}\in \Lambda_{n^2}\cap\Z^ d: \mathbf{i}\text{ is $\ep$-bad}\}>n )\le \sum_{I\subset \Lambda_{n^2}\cap\Z^ d: |I|>n} \P( \forall \mathbf{i}\in I\quad \mathbf{i}\text{ is $\ep$-bad} ).
\end{equation*}
If $\|\mathbf{i}-\mathbf{j}\|_\infty\ge \rho$, then the states of $\mathbf{i}$ and $ \mathbf{j}$ are independent.
It is easy to see (e.g. \cite[Lemma 4.3]{CerfDembin}) that there exists $J\subset I$ such that $|J|\ge |I|/(4\rho)^ d$ and for all $\textbf{i}\ne \textbf{j} \in J$, $\|\textbf{i}-\textbf{j}\|_\infty\ge \rho$.
It follows from  Proposition~\ref{prop:path} that for $n$ large enough,
\aln{\label{eq:bad1}
    &\P( \#\{\mathbf{i}\in \Lambda_{n^2}\cap\Z^ d: \mathbf{i}\text{ is $\ep$-bad}\}>n )\le \sum_{\substack{I\subset \Lambda_{n^2}\cap\Z^ d:\\ |I|>n}} \exp\left(-c\frac{|I|}{(4\rho)^ d}\log ^ 2 n\right)\notag\\
   &\le \sum_{k> n} (4n^ 2)^{dk}\exp\left(-c\frac{k}{(4\rho)^ d}\log ^ 2 n\right)\leq e^{-n\log{n}}.
}
%This concludes the proof.
\end{proof}
\subsection {Construction of microscopic path from  macroscopic path} Let $\ep\in(0,1/2)$ and $N\in \N$.  Consider the macroscopic lattice of sidelength $N$. For $\mathbf{w}\in\Z^d$, denote by $\mathfrak C(\mathbf{w})$ the  $\ep$-bad  {$\Z^d$}-cluster of $\mathbf{w}$ in the macroscopic lattice, that is, the set of  all $\ep$-bad sites connected to $\mathbf{w}$ by a  macroscopic  {$\Z^d$}-path of bad sites. We define $\mathfrak C(\mathbf{w}):=\emptyset$ if $\mathbf{w}$ is $\ep$-good.
\begin{lem}\label{lem:maxsizebadclus}
Let $\delta,\ep>0$. If $N$ is large enough, then for any   set of macroscopic sites $\Gamma$, 
\[\P\left(\left|\bigcup_ {\mathbf{w}\in  \Gamma}\mathfrak C(\mathbf{w})\right|\ge \delta|\Gamma|\right)\le \exp(-\delta |\Gamma|/2).\]
%\sout{where $\mathbf{Bad}$ is the set of connected bad macroscopic components.}
\end{lem}
\begin{proof}
Let $(\mathfrak C_i)_{i\ge 1}$ be a family of independent cluster following the law of $\mathfrak C(\mathbf{0})$. 
We have
\begin{equation}\label{proof lem 2.7}
    \P\left(\left|\bigcup_ {\mathbf{w}\in  \Gamma}\mathfrak C(\mathbf{w})\right|\ge \delta|\Gamma|\right)\le \P\left(\sum_{i=1}^ {|\Gamma|}|\mathfrak C_i|\ge \delta|\Gamma|\right)\le \E(\exp( | \mathfrak C(\mathbf{0})|-\delta))^ {|\Gamma|},
\end{equation}
where we reference the proof of \cite[Lemma 3.6]{GMPT} for the first inequality, while the second inequality employs the exponential Markov inequality. For $N$ large enough,  using the same arguments as in \eqref{eq:bad1}, by \eqref{eq:animals}, 
\begin{equation}\label{eq:contbadcomponents}
\P(| \mathfrak C(\mathbf{0})|\ge k)\le\sum_{\ell\ge k}\sum_{A\in \textbf{Animal}_{\textbf{0}}^\ell}\P( \forall \mathbf{i}\in A\quad \mathbf{i}\text{ is $\ep$-bad} )\le \sum_{\ell\ge k}  7^ {d\ell}\,e^{- c\ell N/(4\rho)^d}\le \exp(-ck N/(8\rho)^d).
\end{equation}
%Using the same arguments as in \eqref{eq:bad1}, we have
%\[\P(| \mathfrak C(\textbf{0})|\ge k)\le 7^ {dk}\exp(-k cN/\rho^d)\le \exp(-c'k N).\]
If $N$ is large enough depending on $\rho,c$ and $\delta$, then we have
\[\E(\exp(| \mathfrak C(\textbf{0})|)) \le  1+\sum_{k\ge 1} e^k \P(| \mathfrak C(\textbf{0})|\ge k)\le 1+\sum_{k\ge 1} \exp(k(1-c N/(8\rho)^d)) \le \exp(\delta/2). \]
 Combined with \eqref{proof lem 2.7}, this yields the claim.
\end{proof}

We will also need the following lemma that is an easy adaptation of   \cite[Lemma 3.2]{dembin}.
\begin{lem} \label{lem:connexiongoodbox}Let $\Gamma$ be a  $*$-connected set of $\ep$-good sites. Let $x\in \sC(\mathbf{j}){\cap \underline\Lambda_N(\mathbf{j})}$ and $y\in \sC(\mathbf{k}){\cap \underline\Lambda_N(\mathbf{k})}$ with $\mathbf{j},\,\mathbf{k}\in\Gamma$. Then, we can find a microscopic, open path joining $x$ and $y$ of length at most $2d\mu(\mathbf{e}_1)N|\Gamma|$ {in $\cup_{\mathbf{i}\in\Gamma}\underline \Lambda_N^ \rho(\mathbf{i})$. }
\end{lem}
We say that $S$ is a macroscopic slab, if there exist $\mathrm I\subset [d]$ with $|\mathrm I|\le d-2$ and  $(k_i)_{i\in\mathrm I} \subset \Z^{|I|}$ such that
\[S=\{(k_1,k_2,\dots,k_d):~ (k_i)_{i\in [d]\setminus\mathrm I}\subset \Z\}.\]
We denote by $\sC_\infty(S)$ the infinite {$\Z^d$-}connected component of $\ep$-good sites contained in $S$. Let
\[\underline S:=\bigcup_{\mathbf{i}\in S}\underline \Lambda_N^ \rho (\mathbf{i}).\]
Given $A\subset \Z^d$, we define the chemical distance constrained on $A$: for $B,C\subset A$,
\al{
\cD^{\G_p}_A(B,C):=
\inf\{|\gamma|:~x\in B,\,y\in C,\,\gamma:x\overset{A}{\leftrightarrow} y\}.
}
We write $\cD^{\G_p}_A(x,C):=\cD^{\G_p}_A(\{x\},C)$, $\cD^{\G_p}_A(x,y):=\cD^{\G_p}_A(\{x\},\{y\})$.  
% Finally we denote $\kG_p(S)$ the graph of the open edges inside the slab $S$ of microscopic vertices inside the slab
% \[\kG_p(S):=(\underline S, \{e\in \E^d: e\subset \underline S, \text{ $e$ is open}\}).\]

\begin{prop}\label{prop:connexpoints}Let $\ep>0$. There exist $c>0$ and $N\in \N$ such that for any  macroscopic slab $S$ and  $\mathbf{x,y}\in S\cap \sZ^d$ satisfying $\mathbf{x}-\mathbf{y}\in \sZ\mathbf e_k$ with some $k\in[d]$,
\ben{\label{Brutas}
\P(\mathbf{x,y}\in\sC_\infty(S),~\exists x\in \sC(\mathbf{x})\cap \underline\Lambda_N(\mathbf{x})\,\exists y\in \sC(\mathbf{y})\cap \underline\Lambda_N(\mathbf{y}); \cD^{\kG_p}_{\underline S}(x,y)\ge (1+\ep)\mu(x-y))\le e^ {-c \|\mathbf{x}-\mathbf{y}\|_1}.
}
\end{prop}

\begin{proof} Without loss of generality,  we assume $\mathbf{y}=\mathbf{x}+K\mathbf e_1$ with some $K\in\N$. Let $\ep_0:=\ep/(2\ep+4d\mu(\mathbf{e}_1))^3\in (0,1/2)$. We take $N=N(\ep_0)\in\N$ large enough as in Lemma~\ref{lem:maxsizebadclus} with $\ep_0$ in place of $\ep$ and $\delta$.  We consider the macroscopic lattice of sidelength $N$ with $\ep_0$-good. Since the probability in the claim is always less than $1$, we can assume $\|\textbf{x}-\textbf{y}\|_\infty\geq 2N$ by taking $c=c(\ep_0,N)$ small enough. Let $\Gamma\subset \Z^d$ be the set of  macroscopic sites intersecting the line between $\mathbf x$ and $\mathbf y$, that is
\[\Gamma:=\{\mathbf x+i\mathbf{e_1}:i\in\{0,\dots,K\}\}.\]
%Note that $\Gamma$ is a deterministic, non-empty, connected set.
Our aim is to prove % Let us assume  $\textbf {x},\textbf{y}\in \sC_\infty(S)$. Our aim is to prove
\ben{\label{HAPPY}
\{\textbf{x,y}\in\sC_\infty(S),\,\exists x\in \sC(\mathbf{x}),\,\exists y\in \sC(\mathbf{y}),\,\cD^{\kG_p}_{\underline S}(x,y)\ge (1+\ep)\mu(x-y)\}\subset\left\{\left|\bigcup_ {\mathbf{w}\in  \Gamma}\mathfrak C(\mathbf{w})\right| \geq \ep_0|\Gamma|\right\}.
}%and the assumption $\|\textbf{x}-\textbf{y}\|_\infty\geq 2N$
Note that \eqref{HAPPY} implies the claim for sufficiently large $N=N(\ep_0)$ and sufficiently small $c=c(\ep_0,N)$ since by Lemma \ref{lem:maxsizebadclus}, the LHS of \eqref{Brutas} is bounded from above by 
$$\exp{(-\ep_0|\Gamma|/2)}\leq \exp{(-c\|\textbf{x}-\textbf{y}\|_1)}. $$
To prove the contrapositive of \eqref{HAPPY}, we suppose $\left|\bigcup_ {\mathbf{w}\in  \Gamma}\mathfrak C(\mathbf{w})\right| < \ep_0|\Gamma|$,  $\textbf{x,y}\in\sC_\infty(S),$ and $x\in \sC(\mathbf{x}),\,y\in \sC(\mathbf{y})$. Denote by $\mathfrak C_S(\mathbf{w})$ the set of  all $\ep$-bad sites connected to $\mathbf{w}$ by a  macroscopic  {$\Z^d$}-path of bad sites included in $S$. Since $\mathfrak C_S(\mathbf{w})\subset \mathfrak C(\mathbf{w})$, 
$\left|\bigcup_ {\mathbf{w}\in  \Gamma}\mathfrak C_S(\mathbf{w})\right| < \ep_0|\Gamma|.$

Let $x\in \sC(\mathbf{x})$ and $y\in \sC(\mathbf{y})$ and denote by $\rL$ the line joining $x$ and $y$.
 Since $\|\textbf{x}-\textbf{y}\|_\infty=\|\textbf{x}-\textbf{y}\|_1\ge 2N$, we can take $(x_k)_{k=0}^K\subset \rL$ such that $x_0=x$, $x_K=y$, and $\|x_i-x_{i-1}\|_\infty\in [N,2N]$ and $x_i\in \underline \Lambda_N(\mathbf x_i)$ where  $\mathbf{x}_i:=\mathbf x +i \mathbf e_1\in \Gamma$. To each good site $\mathbf{x}_k\in\Gamma$, choose the closest point $x_k'$ from $x_k$ in $\sC(\mathbf{x}_k)$ in a deterministic rule breaking ties. % such that {\SN it is at distance at least $N/2$ from $x_0,\dots,x_{k-1}$.}  
% Denote by $x'_k$ the closest point in $\rL$ from $x_k$.
Thanks to Property \eqref{enum1:2} of $\ep_0$-good, we have 
$\|x_k-x'_k\|_\infty\le \ep_0 N.$ % For a finite set $A\subset S$, we denote by $\hat{A}$ the infinite connected component of $\Z^d\setminus A$. 
%$:=\{\mathbf{x}\in S\setminus A: \mathbf{x} \text{ is connected to infinity in $S\setminus A$ and has a $*$-neighbor in $A$}\}.\] %\mathfrak{C}_ {S}(\mathbf{x}_{\tau_{out}(k)+1})
We decompose $\Gamma$ into portions of paths consisting of only good boxes as follows: 
We define 
$$\tau_{in}(1):=1\,,\qquad\tau_{out}(1):=\max\big\{\,j\geq \tau_{in}(1) :\, \forall i\in\{\tau_{in}(1),\dots,j\},\quad\mathbf{x}_i\text{ is good}\,\big\}\,.$$ 
 Suppose that $\tau_{in}(1),\dots,\tau_{in}(k)$ and $\tau_{out}(1),\dots,\tau_{out}(k)$ have been defined. We define
\al{
\tau_{in}(k+1)&:=\min\big\{\,j\geq \tau_{out}(k):\,\forall  i\ge j ,\quad\mathbf{x}_i\notin \mathfrak{C}_S(\mathbf{x}_{\tau_{out}(k)+1})\,\big\}\,,\\
\tau_{out}(k+1)&:=\max\big\{\,j\geq \tau_{in}(k+1) :\, \forall i\in\{\tau_{in}(k+1),\dots,j\},\quad\mathbf{x}_i\text{ is good}\,\big\}\,.
}
%Note that $\tau_{in}(k+1)$ is well-defined since $\textbf {x},\textbf{y}\in \sC_\infty(S)$.
We stop this procedure once $\tau_{out}(k+1)=K$. %Since $\textbf{x}_{\tau_{in}(k+1)-1}\notin \hat{\mathfrak{C}}_S(\mathbf{x}_{\tau_{out}(k)+1})$ 
By construction, using $\textbf {x},\textbf{y}\in \sC_\infty(S)$,
we have  $\mathbf{x}_{\tau_{out}(k)},\mathbf{x}_{\tau_{in}(k+1)}\in  \partial ^{\rm ext}_S\mathfrak{C}_S(\mathbf{x}_{\tau_{out}(k)+1})$, where
\[\partial ^{\rm ext}_S A:=\{\mathbf{x}\in S\setminus A: \mathbf{x} \text{ is connected to infinity in $S\setminus A$ and has a neighbor in $A$}\}.\] %\mathfrak{C}_ {S}(\mathbf{x}_{\tau_{out}(k)+1}).$$
% $\mathbf{x}_{\tau_{out}(k)},\mathbf{x}_{\tau_{in}(k+1)}\in  \mathfrak{C}'(\mathbf{x}_{\tau_{out}(k)+1})\cap S$, where $$\mathfrak{C}'(\mathbf{x}_{\tau_{out}(k)+1}):= \{\textbf{w}\in \Z^d:~\exists \textbf{w}'\in \mathfrak{C}(\mathbf{x}_{\tau_{out}(k)+1});~\|\textbf{w}-\textbf{w}'\|_\infty\leq 1\}.$$ 
Since $\partial ^{\rm ext}_S\mathfrak{C}_S(\mathbf{x}_{\tau_{out}(k)+1})$  consists only of good boxes and is  $*$-connected in $S$ by \cite[Lemma 2]{timar}, there  exists a $\ep_0$-good $*$-path from $\mathbf{x}_{\tau_{out}(k)}$ to $\mathbf{x}_{\tau_{in}(k+1)}$ in $\partial ^{\rm ext}_S\mathfrak{C}_S(\mathbf{x}_{\tau_{out}(k)+1})$. By Lemma \ref{lem:connexiongoodbox}, we can therefore build a microscopic, open path between $x'_{\tau_{out}(k)}$ and $x'_{\tau_{in}(k+1)}$ of length at most   $$2dN\mu(\mathbf{e}_1)|\partial ^{\rm ext}_S\mathfrak{C}_S(\mathbf{x}_{\tau_{out}(k)+1})|\leq (2d)^2N\mu(\mathbf{e}_1)|\mathfrak{C}_S(\mathbf{x}_{\tau_{out}(k)+1})|.$$
Between $x'_{\tau_{in}(k)}$ and $x'_{\tau_{out}(k)}$, thanks to Property \eqref{enum1:3} of $\ep_0$-good, we can build a microscopic, open path of length at most 
\begin{equation*}
    \begin{split}
        \sum_{j=\tau_{in}(k)}^{\tau_{out}(k)-1}[\mu(x'_{j+1}-x'_j)+\ep_0 N]&\le \sum_{j=\tau_{in}(k)}^{\tau_{out}(k)-1}[\mu(x_{j+1}-x_j)+ 4d\mu(\mathbf{e}_1)\ep_0 N]\\&\le \mu(x_{\tau_{out}(k)}-x_{\tau_{in}(k)})+4d\mu(\mathbf{e}_1)\ep_0 N(\tau_{out}(k)-\tau_{in}(k) ),
    \end{split}
\end{equation*}
where we  have used   $x_k\in L\cap \underline \Lambda_N(\mathbf x_k)$ in the last inequality.

Finally, by the assumption $\left|\bigcup_ {\mathbf{w}\in  \Gamma}\mathfrak C_S(\mathbf{w})\right| < \ep_0|\Gamma|$, 
$$\sum_k |\mathfrak{C}_S(\mathbf{x}_{\tau_{out}(k)+1})|\leq \ep_0 |\Gamma|,\quad\sum_k (\tau_{out}(k)-\tau_{in}(k))\leq |\Gamma|.$$
    Therefore, since $\|\textbf{x}-\textbf{y}\|_\infty \ge 2N $ and $|\Gamma|\leq 2d\mu(x-y)/N$, and $\ep_0=\ep/(2\ep+4d\mu(\mathbf{e}_1))^3$, we can build a microscopic, open path from $x$ to $y$ of length at most    
\[\mu(x-y)+ (4d)^2\mu(\mathbf{e}_1)\ep_0 N |\Gamma|\leq (1+\ep)\mu(x-y).\]
Therefore, we have \eqref{HAPPY}. 
\end{proof}

\subsection{Proof of Theorem~\ref{thm:3}}\label{proof of thm:3}
% We consider a slab
% $$\mathbb{S}_N(z)=z+\left(\Z^2\times [-N,N)^{d-2}\right).$$

To prove Theorem \ref{thm:3}, we  need the following proposition.
\begin{prop}\label{thm:slab}
Suppose $d\geq 3$. Let $S=\Z^2\times \{0\}^{d-2}$.  We denote
$$\Lambda_{\ep,\rho}(n;N):=[-\e n,\e n]^ 2\times [-{\rho N},{\rho N}]^{d-2}.$$
For any $\xi,\e>0$, there exist $N\in \N$ and $c>0$ such that for $n$ large enough, $$\P\left(\cD^{\G_p}_{\underline S}(\Lambda_{\ep,\rho}(n;N),n\mathbf{e}_1+\Lambda_{\ep,\rho}(n;N))>(\mu+\xi)n\right)\leq e^{-c n}.$$
%where $S=\Z^2\times \{0\}^{d-2}$.
\end{prop}
\begin{proof}
    Let $\xi,\ep>0$, and $N\in\N$ large enough. 
Let $\kF_n$ be the following event
\[\kF_n:= \begin{array}{c}\{\sC_\infty (S)\cap \{(k_1,0,0,\dots,0):k_1\in[-\ep n/(2N),\ep n/(2N)]\}=\emptyset \}\\\cup\{\sC_\infty (S)\cap \{(\lfloor n/N\rfloor+k_1,0,\dots,0):k_1\in[-\ep n/(2N),\ep n/(2N)]\}=\emptyset \}\end{array}\,.\]
If the event occurs, then there exists a  $*$-connected component of bad sites in $S$ such that it encloses $\{(k_1,0,0,\dots,0):k_1\in[-\ep n/(2N),\ep n/(2N)]\cap\Z\}$ or  $ \{(\lfloor n/N\rfloor+ k_1,0,0,\dots,0):k_1\in[-\ep n/(2N),\ep n/(2N)]\cap\Z\}$.
In particular, this $*$-connected component is
of size at least $\ep n/N$.
By a similar computation as in \eqref{eq:contbadcomponents},
%{\SN similar computations as in \eqref{eq:contbadcomponents}}\sncomment{Is this exactly from Lemma 2.6, no?{\BD yes sorry no}}, 
there exists $c=c(\ep,\rho,N,\kappa_2)>0$ such that for $n$ large enough,
$    \P(\kF_n)\le e^{-cn}.$

On the event $\kF_n^c$, there exist $\mathbf{x}\in\sC_\infty(S)\cap \{(k_1,0,0,\dots,0):k_1\in[-\ep n/(2N),\ep n/(2N)]\}$ and $\mathbf{y}\in\sC_\infty(S)\cap \{(\lfloor n/N\rfloor+k_1,0,0,\dots,0):k_1\in[-\ep n/(2N),\ep n/(2N)]\}$.
The result follows by using Proposition \ref{prop:connexpoints}.

\end{proof}

\begin{proof}[Proof of Theorem \ref{thm:3}] Let $\xi,\e>0$. Let $N=N(\xi,\ep)\in \N$ be as in Proposition \ref{thm:slab}. For $z\in 2{\rho} \Z^{d-2}$, denote $S(z):=\Z^2\times\{z\}$. Note that the slabs $(\underline S(z):z\in 2{\rho}\Z^{d-2})$ are all disjoint. Let $\bar{z}:=(0,0,z)$. We define
\[\cE(z):=\left\{\cD^{\G_p}_{\underline S(z)}(\bar{z}+\Lambda_{\ep,\rho}(n;N),\bar{z}+n\mathbf{e}_1+\Lambda_{\ep,\rho}(n;N))>(\mu+\xi)n\right\}.\]
Notice that 
\[\{\cD^{\G_p}(\Lambda_{2\e n}(0),\Lambda_{2\ep n}(n\mathbf{e}_1))>(\mu+\xi)n\}\subset \bigcap_{z\in{2\rho }\mathbb Z^{d-2}\cap [-\ep n/(2N),\ep n/(2N)]^{d-2}}\cE(z).\]
Since the events on the right hand side are independent, using Proposition \ref{thm:slab}, we get
\[\P(\cD^{\G_p}(\Lambda_{2\e n}(0),\Lambda_{2\ep n}(n\mathbf{e}_1))>(\mu+\xi)n)\le \exp\left(-cn\left(\frac{\ep N}{4N\rho}\right)^ {d-2}\right).\]
\end{proof}

\section{Space-time cut-points}\label{sec:3}
% In this section, we will use the following definition 
% \[\kA_{s,x}(n):=\left\{\exists s_n\ge sn\quad x_n\in \Lambda_{n^{\alpha_d}}(nx): \,{\rm B}_{s_n}\setminus {\rm B}_{s_n-1}=\{x_n\}  \right\}.\]
% This definition differs from the one in the introduction by the fact that $s_n\in(sn,2sn)$. This definition is better for technical reason. We will prove that the resulting rate function $I(s,x)$ is unchanged.
%In this section, we will only consider events $\kA_{s,x}(n)$ with $\|x\|_1\le s$. There is no loss of generality in restricting to this case, as if $s<\|x\|_1$, it is easy to check that $\kA_{\|x\|_1,x}(n)= \kA_{s,x}(n)$ for $n$ large enough.
\subsection{Rate function for space-time cut-point}

The proof of Theorem~\ref{prop:existencelimit} relies on the following key lemma. We postpone its proof until Section \ref{subsection:proofmainlemma}.
\begin{lem}\label{lem:aux}Let $s_0>0$. There exist $ c=c(s_0)>0$ and $ n_0=n_0(s_0)\in\N$ such that the following holds. For any $s,s'\in[0,s_0]$,  $x,x'\in [-s_0,s_0]^d$ and $n\ge m\ge n_0$,
\[\P(\kA_{s,x}(n))\P(\kA_{s',x'}(m))\le e^{cn^{\alpha_d}}\P(\kA_{\frac{n}{n+m}s+\frac{m}{n+m}s',\frac{n}{n+m}x+\frac{m}{n+m}x'}(n+m)).\]
Moreover, when $n=m$, we have
\[\P(\kA_{s,x}(n))\P(\kA_{s',x'}(n))\le e^{cn^{\alpha_d}}\P(\kA_{s+s',x+x'}(n)).\]
\end{lem}
\begin{rem}
Note that $\kA_{s/2,x/2}(2n)\ne \kA_{s,x}(n)$ in general. Hence, the second inequality  does not directly follow from the first inequality in Lemma \ref{lem:aux} by setting $n=m$.
\end{rem}
We have the following corollary:
\begin{cor}\label{continuity of cutpoint}
    Let $s_0>0$. There exist $c=c(s_0)>0$ and $ n_0=n_0(s_0)\in\N$ such that the following holds. For any $\delta>0$, $s,s'\in[0,s_0]$ and $x,x'\in [-s_0,s_0]^d$ with $|s- s'|\leq \delta$ and $\|x-x'\|_1\leq \delta$, and for any $n\ge n_0$,
\[\P(\kA_{s,x}(n))\le e^{c(\delta n+n^{\alpha_d})}\P(\kA_{s',x'}(n)).\]
\end{cor}
\begin{proof}
We first suppose $s\le s'$. Thanks to Lemma \ref{lem:aux}, we have
\[ \P(\kA_{s,x}(n))\P(\kA_{s'-s,x'-x}(n))\le \P(\kA_{s',x'}(n)) e^{c'n^ {\alpha_d}},\]
with some constant $c'>0$ independent of $\delta,n$. Since $|s'-s|+\|x'-x\|_1\leq 2\delta$, one can check that (as in \eqref{eq:boundA1} below), we have
\[\P(\kA_{s'-s,x'-x}(n))\ge e^{-c'' \delta n}\]
with some constant $c''>0$ independent of $\delta,n$. Therefore, with $c:=c'+c''$,
$$\P(\kA_{s,x}(n))\le e^{c( \delta n+n^ {\alpha_d})} \P(\kA_{s',x'}(n)) .$$
If $s>s'$, by the result above, then 
$$\P(\kA_{s,x}(n))\le \P(\kA_{s',x}(n))\leq e^{c( \delta n+n^ {\alpha_d})} \P(\kA_{s',x'}(n)) .$$
\end{proof}
Let us see how this lemma and corollary imply Theorem \ref{prop:existencelimit}.
\begin{proof}[Proof of Theorem \ref{prop:existencelimit}]
Let $s\geq 0$ and $x\in\R^d$. Let $n\ge m\ge 1$. Thanks to Lemma \ref{lem:aux}, 
\[\P(\kA_{s,x}(n))\P(\kA_{s,x}(m))\le e^{cn^{\alpha_d}}\P(\kA_{s,x}(n+m)),\]
with some constant $c>0$ independent of $n,m$. It follows that
\[ -\log \P(\kA_{s,x}(n+m))\le -\log \P(\kA_{s,x}(n))-\log\P(\kA_{s,x}(m))+cn^{\alpha_d}.\]
Using deBruijn and Erd\^os's subadditive lemma \cite{deBruijunErdos},  the following limit  exists:
\begin{equation}\label{eq:existencelimit}
    I(s,x):=\lim_{n\rightarrow\infty}-\frac{1}{n} {\log}\,\P(\kA_{s,x}(n))\in [0,+\infty)\,.
\end{equation}
By Corollary~\ref{continuity of cutpoint},  $I$ is continuous. In particular, $I$ is uniformly continuous on a compact set.

 %In order to prove the locally uniform convergence, we consider a mesh $(\delta \Z)^{d+1}$
 Let $K\subset [0,\infty)\times\R^d$ be a compact set and $s_0:={\rm diam}(K)+1$. Let $\delta\in (0,1/(4ds_0))$ arbitrary. Since the cardinality of $K_\delta:=(\delta \Z)^{d+1}\cap K $ is finite, the convergence \eqref{eq:existencelimit} is uniform over $ K_\delta$. We take $(s,x)\in K$ arbitrary. Let $(s',x')$ be the closest point from $(s,x)$ in $K_\delta$ satisfying $|s'-s|+ \|x'-x\|_1 \leq 4d\delta$  with  a deterministic rule breaking ties.
By Corollary~\ref{continuity of cutpoint} and the uniform convergence  over $K_\delta$, since $I$ is uniformly   continuous on $K$, we have %there exist $c=c(s_0)>0$ and $n_0=n_0(s_0)\in\N$ such that
\al{
\P(\kA_{s,x}(n))\leq e^{o_\delta(n)}\P(\kA_{s',x'}(n))\leq e^{o_\delta(n)} e^{-I(s',x')n}\leq e^{-I(s,x)n+o_\delta(n)},
}
where $o_\delta(n)$ are some positive constants such that $o_{\delta}(n)/n$ converges to $0$ uniformly over $(s,x)\in K$ when $n\to\infty$, and then $\delta\to 0$. Similarly, we have %or $n$ large enough depending only on $s_0$,
$$\P(\kA_{s,x}(n))\geq e^{-o_\delta(n)}\P(\kA_{s',x'}(n))\geq e^{-o_\delta(n)} e^{-I(s',x')n}\geq  e^{-I(s,x)n-o_\delta(n)}.$$
Thus, letting $\delta\to 0$, we find that the convergence \eqref{eq:existencelimit} is uniform on $K$.
%\sncomment{Can we simplify the proof below?}
% Thanks to Lemma \ref{lem:aux}, we have
%\[ \P(\kA_{s,x}(n))\P(\kA_{s'-s,x'-x}(n))\le \P(\kA_{s',\mathrm{x}}(n)) e^{cn^ {\alpha_d}},\]
%with some constant $c>0$ independent of $\delta,n$. Since $|s'-s|+\|x'-x\|_1\leq 4d\delta$, one can check that (as in \eqref{eq:boundA1} below), we have
%\[\P(\kA_{s'-s,x'-x}(n))\ge e^{-c' \delta n}\]
%for some constant $c'>0$ independent of $\delta,n$. Therefore, 
%$$\P(\kA_{s,x}(n))\le \P(\kA_{s',\mathrm{x}}(n)) e^{cn^ {\alpha_d}+c' \delta n}.$$
%Similarly, taking $(s'',\mathrm{x}')$ such that $2d\delta\ge s -s''\ge \|x-\mathrm{x}'\|_1$, we have
%$$\P(\kA_{s,x}(n))\geq \P(\kA_{s'',\mathrm{x}'}(n)) e^{-cn^ {\alpha_d}-c' \delta n}.$$

Let $s_0>0$. We take $c>0$ and $n_0\in \N$ as in Lemma \ref{lem:aux}.
Let $s,s'\in [0,s_0]$,  $x,x'\in[-s_0,s_0]^d$, and 
Let $\lambda\in(0,1)\cap\mathbb{Q}$. Let $(u_n)_{n\ge 1}$ be a sequence of integer larger than $n_0/\min(\lambda,1-\lambda)$ such that $\lambda u_n\in\mathbb{N}$ and $u_n\to\infty$ as $n\to\infty$.
\iffalse
 It follows that using \eqref{eq:existencelimit}
\begin{equation}\label{eq:existencelimit2}
    \lim_{n\rightarrow\infty}\frac{-1}{\lambda u_n}\log{\P(\kA_{s,x}(\lambda u_n))}=I(s,x)\quad\text{and}\quad\lim_{n\rightarrow\infty}\frac{1}{(1-\lambda) u_n}\log{\P(\kA_{s',x'}((1-\lambda) u_n))}=I(s',x') \,.
\end{equation}
\fi
Thanks to Lemma \ref{lem:aux}, we have
\[\P(\kA_{s, x}(\lambda u_n))\P(\kA_{s', x'}((1-\lambda) u_n))\le e^{c u_n^{\alpha_d}} \P(\kA_{\lambda s+(1-\lambda)s',\lambda x+(1-\lambda)x'}( u_n)).\]
By passing to the limit using  \eqref{eq:existencelimit} with $u_n$ in place of $n$, we get
\begin{equation}\label{eq:goalconvexity}
    I(\lambda s+(1-\lambda)s',\lambda x+(1-\lambda)x')\le \lambda I(s,x)+(1-\lambda)I(s',x').
\end{equation}
This together with the continuity of $I$ implies \eqref{eq:goalconvexity} for general $\lambda$.
Moreover,
thanks to Lemma \ref{lem:aux}  with $n=m$, we have
\[\P(\kA_{s, x}(n))\P(\kA_{s', x'}(n)) \le  e^{c n^{\alpha_d}} \P(\kA_{s+s', x+x'}( n)).\]
By passing to the limit using \eqref{eq:existencelimit}, we get
\begin{equation}\label{eq:propI}
   I(s+s',x+x')\le I(s,x)+I(s',x').
\end{equation}

Next, we consider the homogeneity. Let $k,m\in \N$, $s\geq 0$, and $x\in \R^d$. By \eqref{eq:propI}, $I(ks,kx)\leq kI(s,x)$. On the other hand, the opposite inequality is trivial because of $\kA_{ks,kx}(n)\subset \kA_{s,x}(kn)$. Therefore, we have $I(ks,kx)= kI(s,x)$. This implies \ben{\label{Q homog}
I(ks/m,kx/m)=mI(ks/m,kx/m)/m=I(ks,kx)/m= kI(s,x)/m.
}
For general $\lambda\geq 0$,  consider a sequence $\lambda_n\in [0,\infty)\cap \Q$ whose limit is $\lambda$. The homogeneity follows by the continuity of $I$ and \eqref{Q homog} with $k/m=\lambda_n$.

Finally, we prove $I(s,x)>0$ for any $s>0$ and $x\in\R^d$. 
%The fact that $I(s,x)<\infty$ is proved in Proposition \ref{prop:sizeboundary}, in particular we refer to \eqref{eq:boundA1}.
Note that by \eqref{eq:propI}
\[I(2s,0)\le I(s,x)+I(s,-x)=2I(s,x).\]
Moreover, thanks to  \eqref{eq:antalpisztora2}, $I(2s,0)>0$. Thus, we have $I(s,x)>0$.
\end{proof}

\subsection{Size of the ball corresponding to cut-point}
%We  control the size of boundary of $\rmB_s$. Since $\partial ^ {ext}\rmB_s $ may be turtuous, we consider a more regular boundary. 
For $K>0$, define
\[\kA_{s,x}^K(n):=\left\{
\begin{array}{l}
\exists s_n\geq s n,\,\exists w_n\in \Lambda_{n^{\alpha_d}}(nx),\,\exists \Gamma\subset \Z^ d;\\
\,{\rm B}_{s_n}\setminus {\rm B}_{s_n-1}=\{w_n\},\,|\Gamma|\le K n,\,\rmB_{s_n}\subset \mathrm{Int}(\Gamma)  \end{array}\right\}.\]
\begin{prop}\label{prop:sizeboundary}  For any $s_0>0$, there exist $n_0=n_0(s_0)\in\N$ and $K=K(s_0)>0$  such that for any $n\geq n_0$,  
\[\P\left(\bigcup_{s\in[0,s_0],x\in[- s_0,s_0]^d}(\kA_{s,x}(n)\setminus \kA_{s,x}^ K(n))\right)\le e^{-n}\min_{s\in[0,s_0],x\in[- s_0,s_0]^d}\P(\kA_{s,x}(n)).\]

\end{prop}

\begin{proof}
Let $s\in[0,s_0]$, $x\in[-s_0,s_0]^d$, and $n\ge 1$. \iffalse
If $s n<1$, then $\|nx\|_1\le s n <1$ and the event $\kA_{s,x}(n)$ occurs if there is only one open edge with $0$ as an endpoint.
It follows that 
\aln{\label{small en case}
\P(\kA_{s,x}(n))\ge p(1-p)^{2d-1}.
}
Let us now assume $s n \ge 1$.
\fi
Let $w_n\in \Lambda_{n^{\alpha_d}}(nx) {\setminus\{0\}}$ and $\gamma$ be a deterministic self-avoiding path from $0$ to $w_n$ such that $s_0 n\leq |\gamma|\leq 2{d}s_0n$. 
Let us define 
 \[
 %\cE:=\left\{(\tau_e)_{e\in \mathbb{E}^d}\in \%{1,\infty\}^{\E^d}:~\tau_e=1\,\mbox {if $e\in \gamma$},\quad %\tau_e=\infty\mbox  {if $|e\cap \gamma|=1$}
 \cE:=\left\{\forall e\in \gamma,\text{ $e$ is open},\quad \forall e\not\in \gamma\text{ with $|e\cap \gamma|=1$, $e$ is closed} 
\right\}.\]
By construction, we have $\cE\subset \kA_{s,x}(n)$, and 
\[\P(\cE)\ge p^{|\gamma|}(1-p)^{2d|\gamma|}\ge (p(1-p)^{2d})^ {2ds_0 n}.\]
 Hence, there exists $C>0$ depending on $s_0,\,p$, and $d$ such that for all $n\ge 1$, 
\begin{equation}\label{eq:boundA1}
   \P(\kA_{s,x}(n)) \ge  e^{-C n}.
\end{equation}
Let $ s_n\geq sn$ be such that ${\rm B}_{s_n}\setminus {\rm B}_{{s_n}-1}=\{w_n\}$. Set $\ep=1/2$. Let $N,K\in \N$ be chosen later.
 Let $\Gamma_N$ be the set of boxes intersecting $\rmB_{s_n}$:
\[\Gamma_N:=\{\mathbf{i}\in\Z^d: \underline \Lambda_N(\mathbf{i})\cap \rmB_{s_n}\ne \emptyset \}.\]
Note that $\Gamma_N$ is connected in $\Z^d$. Let us consider the exterior boundary of $\Gamma_N$ as 
$$\partial^{\rm ext}\Gamma_N:=\{\mathbf{i}\in \Z^d\setminus \Gamma_N:~\exists \mathbf{j}\in \Gamma_N;~\|\mathbf{i}-\mathbf{j}\|_1=1,\,\text{$\mathbf{i}$ is connected to infinity in $\Z^ d\setminus \Gamma_N$}\}.$$ %the subset of $\Gamma_N$ made of macroscopic sites that are connected to infinity in $\Z^ d\setminus \Gamma_N$ {\BD \sout {and that site}}.
Note that $\partial^{\rm ext}\Gamma_N$ is $*$-connected by \cite[Lemma 2]{timar}. Let $\mathbf{i}\in \Gamma_N$ and $\mathbf{j}\in \partial^{\rm ext}\Gamma_N$ be such that $\|\mathbf{i}-\mathbf{j}\|_1=1$ and $w\notin \underline \Lambda'_N(\mathbf{i})$. %or $\mathbf{j}$ has a $\ep$-bad $*$-neighbour.
Let us prove that $\mathbf{i}$ or $\mathbf{j}$ is $\ep$-bad.   By definition of $\Gamma_N$ and  $\partial^{\rm ext}\Gamma_N$, we have $\underline\Lambda_N(\mathbf{j})\cap \rmB_{s_n}=\emptyset$. %$\mathbf{j}$  {is} connected to infinity in  $\Z^ d\setminus \Gamma_N$. {\SN In particular, we have $\underline\Lambda_N(\mathbf{j})\cap \rmB_{s_n}=\emptyset$. }%{\BD Indeed if $\underline\Lambda_N(\mathbf{j})\cap \rmB_{s_n}\ne\emptyset$, then any $*$-path from $\mathbf{j}$ to infinity must  intersect $\Gamma_N$. } \sncomment{Why?{\BD I added explanation}}  
Let us assume that $\mathbf{i}$ and $\mathbf{j}$ are both good. Let $x\in\rmB_{s_n}\cap \underline \Lambda_N(\mathbf{i})$. Note that $x$ is connected to $w\notin \underline \Lambda_N' (\mathbf{i})$ by a microscopic, open path inside $\rmB_{s_n}$. By definition of good box, this yields that $ x\in \sC(\mathbf{i})\cap \rmB_{s_n}\cap \underline \Lambda_N(\mathbf{i})$. Since $\underline\Lambda_N(\mathbf{j})\cap \rmB_{s_n}=\emptyset$, and $\mathbf{i}$ and $\mathbf{j}$ are good, there exists
$y\in {\sC(\mathbf{j})\cap\sC(\mathbf{i})  \setminus \rmB_{s_n}}$.  %Since , $\sC(\mathbf{i})$ and $\sC(\mathbf{j})\cap \underline \Lambda'_N(\mathbf{i})$ are connected in $\underline \Lambda'_N(\mathbf{i})$.
Thus,  $x$ and $y$ are connected by a microscopic, open path in $\sC(\mathbf{i})\subset \underline \Lambda'_N(\mathbf{i})$. However, by definition of cut-point, any microscopic, open path  from $x\in\rmB_{s_n}$ to $y\notin\rmB_{s_n}$ must contain $w_n$, which derives a contradiction since $w\notin \underline \Lambda'_N(\mathbf{i})$. It follows that any $\mathbf{j}\in \partial^{\rm ext}\Gamma_N$ is $\ep$-bad, or has a $\ep$-bad neighbor, or has a neighbor whose enlarged box contains $w_n$. %Let $$\partial^{\rm ext}\Gamma_N:=\{\mathbf{i}\in \Z^d\setminus \Gamma_N:~\exists \mathbf{j}\in \partial^{\rm ext}\Gamma_N;~\|\mathbf{i}-\mathbf{j}\|_\infty\leq 1\}.$$ 
Let $A^+:=\{\mathbf{i}\in\Z^d:~\exists \mathbf{j}\in A;~ \|\mathbf{i}-\mathbf{j}\|_1\leq 1\}$ for $A\subset Z^d$. Therefore, $\partial^{\rm ext}\Gamma_N$ is $*$-connected,  its interior contains $\Gamma_N$, and $$|\{\mathbf{i}\in (\partial^{\rm ext}\Gamma_N)^+:~\mathbf{i}\text{ is bad}\}|\geq \frac{|\partial^{\rm ext}\Gamma_N|}{(4d)^d}-(4d)^d.$$
Let $\Gamma:=\bigcup_{\mathbf{j}\in\partial^{\rm ext}\Gamma_N}\underline \Lambda_N(\mathbf{j}).$ %\text{, and}\\
%\Gamma&:=\{x\in\Z^d\setminus (V\cup \mathrm{Int}(V)): \exists y\in V\quad \|x-y\|_1=1\}.
Since  the interior of $\partial^{\rm ext}\Gamma_N$ contains $\Gamma_N$, 
$\rmB_{s_n}\subset \mathrm{Int}(\Gamma).$

\indent Let $\kappa=1/(10d)^d$ and $K\geq (10d)^{10d}$. Denote by $\kF$ the event that there exists a macroscopic, $*$-connected set $A$ such that $|A|\geq Kn$, $|\{\mathbf{i}\in A^+:~\mathbf{i}\text{ is bad}\}|\geq \kappa|A|$,  and  its interior contains $0$. To compute its probability, we first fix such a $*$-connected set $A$. By pigeon-hole principle as in Proposition \ref{prop:path}, there exists a subset of $A^+$ containing at least $\frac{\kappa}{(4\rho)^d}|A|$ bad sites  at $\ell_\infty$ distance at least $\rho$ from each other (recall the notation $\rho $ from  \eqref{def:rho}). In particular, the states of these sites are independent.
By Proposition \ref{prop:path},  for $N$ large enough depending on $d,p$,
%\\\\\\\\\\\\\\\Let us compute the probability that a fixed set $A$ has a fraction at least $\kappa$ of bad sites and bad $*$-neighbors. 
\begin{equation*}
    \P(|\{\mathbf{x}\in A^+: \mathbf{x} \text{ is $\ep$-bad}\}|\ge \kappa |A|)\le 2^ {|A^+|}\exp\left(- cN\frac{\kappa}{(4\rho)^ d}|A|\right) .
\end{equation*}
This yields
\begin{align*}
    \P(\kF)&\le\sum_{k\ge K n}\sum_{y\in [-k,k]^d\cap\Z^d}\sum_{A\in\textbf{Animal}^k_y}\P(|\{\mathbf{x}\in A^+: \mathbf{x}\text{ is $\ep$-bad}\}|\ge \kappa |A|)\\
    &\le \sum_{k\ge K n}7^ {dk}(2k+1)^d2^ {4^d k}\exp\left(- cN\frac{\kappa}{(4\rho)^d}k\right) {\leq \exp\left(- cN\frac{\kappa}{(8\rho)^d}K n\right)},
\end{align*}
 where we have used \eqref{eq:animals} in the last line.
 Hence, thanks to \eqref{eq:boundA1}, for $K$ large enough depending on $s_0$, $p$, and $d$, we have for all $n\in \N$, 
\begin{equation}\label{eq:boundA2}
    \P(\kF)\le e^{-  (1+C) n}\le e^{-n}  {\min_{s\in[0,s_0],x\in[- s_0,s_0]^d}\P(\kA_{s,x}(n))}.
\end{equation}
On the event $\kF^c$, $|\Gamma|\le (2N)^d|\partial^{\rm ext}\Gamma_N|\le (2N)^d K n=:K'n\,.$ 
%Note that any path $\gamma$ from a point in $\rmB_{s_n}$ to infinity intersects $\Gamma$. Indeed, since $\gamma$ must intersect $\partial^{\rm ext} \rmB_{s_n}$, %say $y\in \underline{\Lambda}_N(\mathbf{i})$ with $\mathbf{i}\in \Z^d$, 
% the macroscopic path associated with $\gamma$ intersects $\partial^{\rm ext}\Gamma_N(B_{s_n})$, which implies that $\gamma$ intersects ${\rm V}$. 
Therefore, we obtain $\kA_{s,x}(n)\setminus \kA_{s,x}^ {K'}(n)\subset  \kF$, and by using \eqref{eq:boundA2}, it follows that 
\begin{equation*}
    \P\left(\bigcup_{s\in[0,s_0],x\in[- s_0,s_0]^d}(\kA_{s,x}(n)\setminus \kA_{s,x}^ {K'}(n))\right)\le \P(\kF)\le e^{- n}\min_{s\in[0,s_0],x\in[- s_0,s_0]^d}\P(\kA_{s,x}(n)).
\end{equation*}

\end{proof}
\subsection{Resampling arguments}
\begin{lem}[Resampling lemma]\label{lem:resample}Let $A\in\mathbb N$.
Recall $\tau=(\tau_e)_{e\in \E^d}$ from \eqref{Chemical-FPP}. Let $E_0=E_0(\tau)$, $ E_1=E_1(\tau)$ be random sets of edges inside  $[-n^2,n^2]^d$ depending on $\tau$ such that   
\[E_0\cap E_1=\emptyset,\quad\max\{|E_0|,|E_1|\}\le A\quad\text{a.s.}\] 
 Then, there exists another  random configuration $(\tau_e^{\rm r})_{e\in \E^d}\in \{1,\infty\}^{\E^d}$ with the same law as $\tau$ such that for any event $\cE\subset \{1,\infty\}^{\E^d}$ and $n\geq\max\{4d, (p(1-p))^{-1}\}$,
\al{
&\P(\tau\in\cE;\,\forall e\in E_0(\tau),\,  \tau_ e^{\rm r}=+\infty;\, \forall e\in E_1(\tau),\, \tau_ e^{\rm r}=1;\,\forall e\notin E_0(\tau)\cup E_1(\tau),\,  \tau^{\rm r}_e=\tau_e)\\
&\ge e^{-8d\, A\log n}\P(\tau\in\cE).
}
\end{lem}
The configuration $\tau^{\rm r}=(\tau_e^{\rm r})_{e\in \E^d}$ is called  the {\bf resampled configuration}.
\begin{proof} Let $\tau^*=(\tau_e^*)_{e\in \E^d}$ be an independent copy of $\tau$. Let $\mathfrak {E}_0$ and $\mathfrak{E}_1$ be two independent random variables distributed uniformly on the set 
\[ \{E\subset  \E^d:\quad \forall e\in E\quad e\subset [-n^2,n^2]^d,\quad |E|\le A\}.\]
In particular, we have
\[\P(\mathfrak{E}_0=E_0,\,\mathfrak{E}_1=E_1)\ge e^{-5d\,A\log n}.\]
We define the resampled configuration $\tau^{\rm r}$ as
$$\tau^{\rm r}_e=
\begin{cases}
\tau_e^*&\text{ if $e\in \mathfrak{E}_0\cup \mathfrak{E}_1$},\\
\tau_e&\text{ otherwise.}
\end{cases}$$
Recall that $n\geq\max\{4d, (p(1-p))^{-1}\}$.
We conclude%\bdcomment{why no dependence in $p$ ? maybe add one step with $(p(1-p))^A$ {\SN GOOD POINT!!}}
\al{
&  e^{- 8d\, A\log n}\,\P(\tau\in\cE)\\
&\leq \P(\tau\in\cE,\, \forall e\in \mathfrak{E}_1\,\,\tau_ e^{\rm r}=+\infty\,,\, \forall e\in \mathfrak{E}_0\,\,\tau_ e^{\rm r}=1\,,\,\mathfrak{E}_0=E_0, \mathfrak{E}_1=E_1 )\\
&\leq \P(\tau \in\cE;\,\forall e\in E_0(\tau),\, \tau_ e^{\rm r}=+\infty;\,\forall e\in E_1(\tau),\, \tau_ e^{\rm r}=1;\,\forall e\notin E_0(\tau)\cup E_1(\tau),\, \tau^{\rm r}_e=\tau_e).
}
\end{proof}
We say that an event $\cE$ is decreasing if for any $\tau\in\cE$ and $\tau'$ satisfying for all $e\in\E^ d$  $\tau'_e\ge \tau_e$,  $\tau'\in\cE$. The terminology decreasing is here related to percolation. In other words, the less open edges there are, the more likely the event  occurs. %Here the notion of decreasing event is related the notion of closed edges.
\begin{lem}[Making a cut-point]\label{lem:createcutpoin} Let $t_0\in\mathbb N$.
 Consider a family of decreasing events  $(\cE(w),w\in\Z^d)$, $\cE(w)\subset \{1,\infty\}^{\E^d}$. For any  $k\geq (4d+(p(1-p))^{-1})^2$, $t_0\geq \sqrt{k}$, and $\Lambda\subset\Z^d$,
%Then, there exists another  random configuration $(\tau_e^{\rm r})_{e\in \E^d}\in \{1,\infty\}^{\E^d}$ with the same law as $\tau$ such that for any event $\cE\subset \{1,\infty\}^{\E^d}$, $x\in\Z^ d$, $k\in \N$ and $t\geq \sqrt{k}$
\al{
&e^{-(8d)^2 \sqrt{k}\log{k}} \P(\exists t\ge t_0\quad \exists x\in \Lambda: |\rmB_t|\le k,\,x\in \rmB_t\backslash \rmB_{t-1},\,\tau\in \kE(x))\\
&\qquad\leq \P (\exists t\ge t_0\quad \exists x\in \Lambda: \,\rmB_t\setminus \rmB_{t-1}=\{x\},~ \tau\in \kE(x)).
}

%&\leq \P (\rmB_t^{\rm r}\subset \rmB_t,\,\forall s\in [t]\backslash [t-\sqrt{k}]\quad |\rmB_s^{\rm r}\backslash \rmB^{\rm r}_{s-1}|=1\,,\,\rmB^{\rm r}_t\setminus \rmB^{\rm r}_{t-1}=\{x\},\,\forall e\subset \rmB_t^c\quad \tau^{\rm r}_e=\tau_e\, ,\,\tau\in \kE)
%where $\rmB^{\rm r}_s$ denote the ball of radius $s$ centered at $0$ for the configuration $\tau^{\rm r}$.
\end{lem}
%\sncomment{we consider random $t,x$, or conditioning}
\begin{proof}Set $$\kE:=\{\tau\in \{1,\infty\}^{\E^d}:~\exists t\ge t_0\quad \exists x\in \Lambda: |\rmB_t|\le k,\,x\in \rmB_t\backslash \rmB_{t-1},\,\tau\in \kE(x)\}.$$
Given  $\tau\in\kE$,  let $t=t(\tau)\ge t_0$ and $x=x(\tau)\in\Lambda$ be such that $|\rmB_t|\le k,\,x\in \rmB_t\backslash \rmB_{t-1},\,\tau\in \kE(x)$. 
 We take a geodesic $\gamma_x$ from $0$ to $x$. In case there are several choices, we pick one of them according to a deterministic rule. % with a deterministic rule breaking ties. 
 Since the sets $(\rmB_r\backslash \rmB_{r-1}, r=1\dots t)$ are disjoint, we have
 \[\sum_{r=t-\lceil\sqrt{k}\rceil}^t |\rmB_r\backslash \rmB_{r-1}|\le|\rmB_t|\le k.\]
Hence, there exists $r\in [t]\backslash [t-\lceil\sqrt{k}\rceil]$ such that $|\rmB_r\backslash \rmB_{r-1}|\leq \sqrt{k}$. Let $E_1(\tau)=\emptyset$ and
$$E_0(\tau):=\{\langle v,w\rangle\in \E^d\setminus\gamma_x:~v\in \gamma_x\setminus  \rmB_{r-1}\}\cup \{e\in \E^d\setminus \gamma_x:~e\text{ connects $\rmB_{r}\backslash \rmB_{r-1}$ and $\rmB_{r-1}$}\}.$$
Note that $|E_0(\tau)|\leq 4d\sqrt{k}$ and $E_0(\tau)\subset [-k,k]^d$. 
Hence, we can take a resampled configuration $\tau^{\rm r}=(\tau_e^{\rm r})_ {e\in  
\E^d}$ as in Lemma~\ref{lem:resample} with $\kE\subset\{1,\infty\}^{\mathbb{E}^d}$ so that $\tau_e^{\rm r}\geq \tau_e$. Denote by  $\rmB^{\rm r}_s$ the ball of radius $s$ centered at $0$ for $\tau^{\rm r}$. Since  $\cE(x)$ is decreasing, $\tau^{\rm r}\in \cE(x)$. Thus, we get
\al{
&e^{-(8d)^2\sqrt{k}\log{k}}\P(\exists t\ge t_0\quad \exists x\in \Lambda: |\rmB_t|\le k,\,x\in \rmB_t\backslash \rmB_{t-1},\,\tau\in \kE(x))\\
&\leq \P(\exists t\ge t_0\quad \exists x\in \Lambda: |\rmB_t|\le k,\,x\in \rmB_t\backslash \rmB_{t-1},\,\forall e\in E_0(\tau)\quad\tau_e^{\rm r}=\infty\,,\,\tau\in \kE(x))\\
&\leq \P(\exists t\ge t_0\quad \exists x\in \Lambda:|\rmB_t^{\rm r}|\le k,\,\rmB_t^{\rm r}\setminus \rmB_{t-1}^{\rm r}=\{x\},\,\tau^{\rm r}\in \kE(x))\\
 &=\P(\exists t\ge t_0\quad \exists x\in \Lambda:|\rmB_t|\le k\,,\,\rmB_t\setminus \rmB_{t-1}=\{x\},\,\tau\in \kE(x)).
}
%This yields the result.
\end{proof}
\subsection{Drilling free line} We define the number of lines in direction $\mathbf{e}_i$ intersecting $A\subset \Z^d$ as
\begin{equation}\label{def:Ni}
    \rN_i(A):=|\{z\in \Z^d:~z_i=0,\,\rL_i(z)\cap A\neq \emptyset\}|.
\end{equation}
Given $s>0$ and $x\in \R^d$, let us define
\aln{\label{free line def}
\kA^{\rm free}_{s,x}(n):=\left\{\begin{array}{c}\exists i\neq j\in [d],\quad \exists s_n\geq   s n-3n^{\alpha_d},\quad \exists w_n\in  \Lambda_{4n^{\alpha_d}}(nx); \,\rmB_{s_n}\setminus \rmB_{s_n-1}=\{w_n\},\\\rL_i(w_n)\cap \rmB_{s_n}=\{w_n\},\,\forall k\neq j\quad \rN_k(\rH_j(w_n)\cap \rmB _{s_n})\le n^{\alpha_d},\,|\rmB_{s_n}|\le n^{7/4} \end{array}\right\}.
}
The following is a key result of this section, which is the most technical part of this paper. \iffalse As discussed in the introduction, the main problem to prove Lemma \ref{lem:aux} is to connect two balls corresponding to two realisations of cut-point events. To solve this issue, we modify the configuration to drill a free line at the cut-point and at the origin. This will help us to create the desired connection. When modifying the configurations, we have to make sure we do not change the  metric structure  inside the ball.
\fi
\begin{lem}[Drilling free line at cut-point]\label{lem:drilling} Let $s_0>0$ and  $K=K(s_0)>0$ be as in Proposition \ref{prop:sizeboundary}. There exists  $n_0=n(s_0)\in \N$ such that the following holds.
For any $s\in[0,s_0]$ and $x\in [-s_0,s_0]^d$, and  $n\ge n_0$, 
$$e^{-n^{\alpha_d}}\P(\kA_{s,x}^K(n))\leq \P\left(\kA^{\rm free}_{s,x}(n)\right).$$
\end{lem}
We will need the following deterministic lemma. We postpone its proof until the appendix.
\begin{lem}\label{lem:4.2}
Suppose $d\geq 2$. Let $S\subset\Z^d$. We define
 $$m(d,S):= \left({\frac{|S|}{2^{d-1}\Diam(S)}}\right)^ {\frac{1}{d-1}}.$$
There exist $i\neq j\in[d]$ and $S'\subset S$ with $|S'|\ge m(d,S)$ such that $z_i\neq z'_i$ and $z_j\neq z_j'$ for any $z\neq z'\in S'$. 
\end{lem}
From now on, if it is clear from context, we simply write $n^a$ instead of $\lfloor n^a\rfloor$ for $a>0$.  On the event $\kA_{s,x}^K(n)$,  let $s_n\ge sn$ and $w_n\in  \Lambda_{n^{\alpha_d}}(nx) $ such that $\rmB_{s_n}\setminus\rmB_{s_n-1}=\{w_n\}$. We take $\gamma_{0,w_n}=(w^1,\cdots, w^{s_n})$ to be a geodesic between $0$ and $w_n$ with a deterministic rule breaking ties. 
 Recall $\alpha_d=1-\frac{1}{6d}$. Set $\beta_d=\frac{5\alpha_d -4}{ {5}(d-1)}$. In particular, we have $1-\beta_d<\alpha_d$.

  If $s n<3n^{\alpha_d}$ and $x\in  {\Lambda_{4n^{\alpha_d}}}(0)$, then $\kA^{\rm free}_{s,x}(n)$ occurs with $s_n=0$ and $w_n=0$, in particular $$\P(\kA_{s,x}^K(n))\leq 1=\P(\kA^{\rm free}_{s,x}(n)).$$ Otherwise, then since $s_n\geq 3n^{\alpha_d}$ or $w_n\notin \Lambda_{3n^{\alpha_d}}(0)$, we eventually have $s_n\geq 3n^{\alpha_d}$  that we assume from now on.
Let \[S:=\{w^{s_n-2n^{\alpha_d}},\cdots, w^{s_n-n^{\alpha_d}}\}.\] We consider the events:
\al{
\kA^{[1]}_n&:=\kA_{s,x}^K(n)\cap \left\{\Diam(S)\ge  2^{-(d-1)}\,n^{\frac{4}{5}}\right\},\\
\kA^{[2]}_n&:=\kA_{s,x}^K(n)\cap \{\exists i\neq j\in [d],\, \exists S'\subset S;~|S'|=n^{\beta_d},\,\forall z\neq z'\in S'\,z_i\neq z'_{i},\,z_j\neq z'_{j},\,\}.
}
By Lemma \ref{lem:4.2},  $\kA_{s,x}^K(n)\subset \kA^{[1]}_n\cup\kA^{[2]}_n$. Hence, Lemma~\ref{lem:drilling} follows from the following.
\begin{lem}\label{lem:drilling2}
For $n\in\N$ large enough, 
\aln{
\P(\kA_n^{[1]})&\leq \frac{1}{2}e^{n^{\alpha_d}}\P\left(\kA^{\rm free}_{s,x}(n)\right),\label{case1}\\
\P( \kA_n^{[2]})&\leq \frac{1}{2}e^{n^{\alpha_d}}\P\left(\kA^{\rm free}_{s,x}(n)\right)\label{case2}.
}
\end{lem}
The proof of Lemma \ref{lem:drilling2} involves modifying the initial environment and adjusting the time of the edges such that a new ball in the resampled environment possesses a free line at its cut-point, meaning a line that intersects the ball solely at the cut-point. To achieve this, we drill a line near the end of geodesic $\gamma_{0,w_n}$, which is in close proximity to the cut-point $w_n$. The primary challenge in this process is maintaining the overall structure of the ball, ensuring the new space-time cut-point remains close to the original one in both position and time. We need to be cautious in order to prevent the creation of a shortcut within the ball during the resampling process.

The first case \eqref{case1} is simpler, as it pertains to situations where the geodesic $\gamma_{0,w_n}$ does not wiggle too much (i.e. the diameter of $S$ is sufficiently large). In this case, we can identify a hyperplane intersecting $S$ where the size of the ball's intersection with the hyperplane is negligible compared to $n$. As a result, we only incur a minimal cost to maintain the connections after the shift and avoid a specific line. However, the second case \eqref{case2} is considerably more difficult. This case applies to instances where the path has significant fluctuations around the endpoint, and the intersections of the ball with the hyperplane are not negligible. In this scenario, the shifting procedure differs; we employ two hyperplanes and shift each quadrant. It is not feasible to preserve all connections within the ball, as the expense may become exponentially high. Instead, we maintain the connections of the geodesic while blocking connections around bad boxes and the ball's boundary. We will show that this resampling process does not generate a shortcut inside the ball.
\begin{proof}[Proof of \eqref{case1}]
 We suppose the event $\kA_n^{[1]}$ occurs. Without loss of generality, we assume $\Diam_1(S)\ge  n^{\frac{4}{5}}/ 2^{d-1}$. Then, there exists $S'\subset S=\{w^{s_n-2n^{\alpha_d}},\cdots, w^{s_n-n^{\alpha_d}}\}$ such that $|S'|=n^{4/5}/ 2^{d-1}$ and  the hyperplanes $(\rH_{1}(z),\,z\in S') $ are disjoint.  Let $\Gamma\subset \Z^ d$ be such that $|\Gamma|\le K n$ and $\rmB_{s_n}\subset\mathrm{Int}(\Gamma)$.  Then, there exists $w^*=w^{s_*}\in S'$ such that 
$|(\rH_1(w^*-\mathbf{e}_1)\cup  \rH_1(w^*))\cap \Gamma|\leq 2^dK n^{\frac{1}{5}}.$ Hence, by the isoperimetry \eqref{eq:isoperimetry}, we have 
\ben{\label{scenario 1 estimate}
| (\rH_1(w^*-\mathbf{e}_1)\cup  \rH_1(w^*))\cap \rmB_{s_n}|\leq  C n^{2/5}
}
with some $C=C(d,K)\in \N$. We write $\overline \rH_1^*:=\rH_1(w^*-\mathbf{e}_1)\cup  \rH_1(w^*)$ and $\rH_1^*:= \rH_1(w^*)$. %Let $\gamma:=(w^1,\cdots, w^{s_*}).$

Recall $(\tau_e)_{e\in \E^d}$ from \eqref{Chemical-FPP}. Let $(\tau_e^*)$ be an independent copy of $(\tau_e)$, and $\mathbf{X}=(X_i)_{i=1}^d$ a uniform random variable on $[-Kn,Kn]^d\cap \Z^d$ independent both from $(\tau_e),(\tau^*_e)$. Let $\wp_n=2C  n^{2/5}\in 2\N$. We consider the following shift:
$$\SH(x)=\SH_{\mathbf{X}}(x):=
\begin{cases}
x-(\wp_n -1)\mathbf{e}_1 &\text{if $x_1< X_1$},\\
x+\wp_n \mathbf{e}_1 &\text{if $x_1\geq  X_1$}.
\end{cases}$$ 
Then, we define the resampled configuration $(\tau_e^{\rm r})$ as follows: for $e=\langle x,y\rangle$, 
\al{\tau_{e}^{\rm r}:=\left\{
\begin{array}{ll}
\tau_e^*&\text{ if $x_1$ or $y_1\in (X_1-\wp_n,X_1+\wp_n)$},\\
\tau_{\langle \SH^{-1}(x),\SH^{-1}(y)\rangle}&\text{ otherwise}.
\end{array}\right.
}
In the proof, we simply write $\kD=\kD^{\kG_p}$ for the chemical distance for the configuration $\tau$.
Denote by $\kD^{\rm r}$ the chemical distance for the configuration $\tau^{\rm r}$. We need the following claim.
\begin{claim}
\label{lem case 1}
Let $n\in \N$ and $(x^i)_{i=1}^n,(y^i)_{i=1}^n\subset \Z^d$ such that
$$2n=-x_1^i=y^i_1,\,x_j^i=y^i_j\,\forall j\geq 2.$$
Then, there exist disjoint paths $(\gamma^i)_{i=1}^n$ such that
$$\gamma^i:x^i\to y^i,\,\gamma^i\subset (-2n,2n)\times \Z^{d-1}\cup\{x^i,y^i\},\, \gamma^i\cap \rL_3(0)=\emptyset,\,|\gamma^i|\leq 8n.$$
\end{claim}
    \begin{proof}
    Without loss of generality, we can suppose $x^1_2\leq x^2_2\leq \cdots x^n_2$. We denote by $L(x,y)$ the straight line between $x$ and $y$. By abuse of notation, we can see $L(x,y)$ as a path by considering all the edges $e$ such that $e\subset L(x,y)$. 
    Set \[\ell:=\inf\{k\ge 1: x_2^k\ge 0\}.\]
   For $i<\ell$,  we take $\gamma^i=L(x^i,y^i)$; for $i\ge \ell$, we take $\gamma^i$ to be
    \al{
    &L(x^i,x^i+2(n-i)\mathbf{e}_1)\oplus L(x^i+2(n-i)\mathbf{e}_1,x^i+2(n-i)\mathbf{e}_1+\mathbf{e}_2)\\
    &\qquad \oplus L(x^i+2(n-i)\mathbf{e}_1+\mathbf{e}_2,y^i-2(n-i)\mathbf{e}_1+\mathbf{e}_2) \oplus L(y^i-2(n-i)\mathbf{e}_1+\mathbf{e}_2,y^i-2(n-i)\mathbf{e}_1) \\
    &\qquad\qquad\oplus L(y^i-2(n-i)\mathbf{e}_1,y^i).
    }
From the definition, $(\gamma^i)_{i=1}^n$ are disjoint paths and do not intersect $\rL_3(0)$ (See Figure \ref{fig5}).
\begin{figure}[!ht]
\def\svgwidth{0.4\textwidth}
 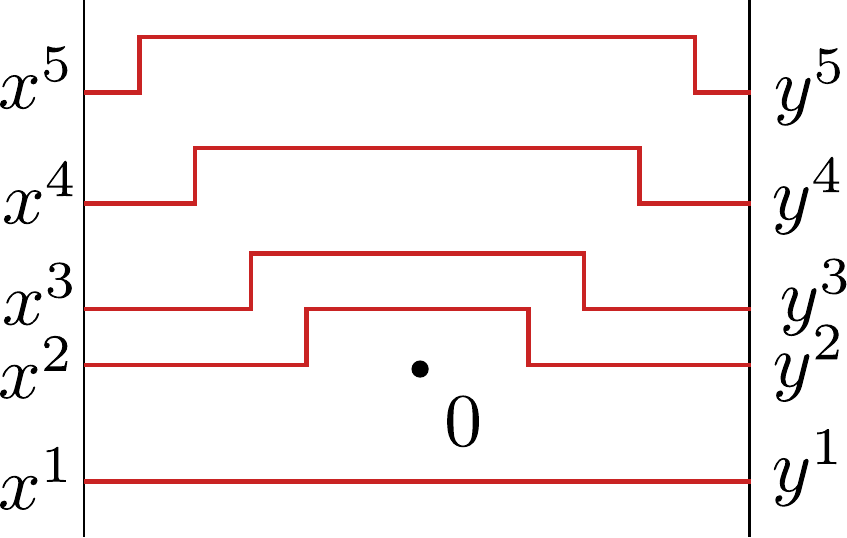
 \caption[fig5]{\label{fig5}Construction of $(\gamma^i)$ colored in  red.}
\end{figure}
        \end{proof}
From now on, we suppose $\mathbf{X}=w^*$. Since $|\overline \rH_1^*\cap \rmB_{s_n}|\leq \wp_n/2$, by Claim~\ref{lem case 1} with $\wp_n/2$ in place of $n$,  there exist disjoint paths $\bar{\gamma}_{x,y}=\bar{\gamma}_{y,x}:x-\wp_n \mathbf{e}_1\to x+\wp_n \mathbf{e}_1$ for each pair $x,y\in \overline \rH_1^*\cap \rmB_{s_n}$ with $y=x-\mathbf{e}_1$ such that $$\bar{\gamma}_{x,y}\backslash\{x-\wp_n\mathbf{e}_1,x+\wp_n\mathbf{e}_1\}\subset \Z^d\backslash (\SH(\Z^d)\cup \rL_{3}(w^*)),$$ and $|\bar{\gamma}_{x,y}|\leq 4\wp_n$.  We say that $e\in \E^d$ crosses $\rH^*_1$ if there exists $z\in  \rH_1^*$ such that $e=\langle z-\mathbf{e}_i,z\rangle$. Note that here we make no distinction on the orientation of the crossing. Let $(i_j)_{j=1}^{k}$ be such that $\langle w^{i_j},w^{i_j+1}\rangle$ is the j-th crossing of $\gamma_{0,w_n}$ over $\rH^*_1$ until $s_*$. 
Let us define 
\al{
\mathbf{E}_1&:=\left\{e\in \E^d:~e\in \bar{\gamma}_{w^{i_j},w^{i_j+1}}\text{ with  }j\leq k,\text{ or }e\subset L(w^*,\SH(w^*))\right\},\\
\mathbf{E}_2&:=\{e\in \E^d\backslash \mathbf{E}_1:~|e\cap e'|=1\text{ with some $e'\in \mathbf{E}_1$ or  $e\cap \SH (\rmB_{s_n}\cap \overline \rH_1^*)\neq \emptyset$}\}.
}
 Then, we consider the event 
\al{
\mathcal{A}^{\rm r}:=\left\{\mathbf{X}=w^*;\,
\tau^*_e=1\,\text{ for }e\in \mathbf{E}_1;\,\tau^*_e=\infty\,\text{ for }e\in \mathbf{E}_2
\right\}.
}
    Since $|\mathbf{E}_1\cup\mathbf{E}_2|\leq 4d\wp_n^2\le  16C^ 2dn^{4/5}$ on $\mathcal{A}^{[1]}_n$, for $n$ large enough depending on $p$, $d$, and $s_0$,
\begin{equation}\label{eq:addar}
\begin{split}
\P(\mathcal{A}^{\rm r}\cap\mathcal{A}^{[1]}_n)&= \E[\P(\mathbf{X}=w^*,\,
\tau^*_e=1\,\text{ for }e\in \mathbf{E}_1;\,\tau^*_e=\infty\,\text{ for }e\in \mathbf{E}_2|~\mathcal{A}^{[1]}_n)\,\mathbf{1}_{\mathcal{A}^{[1]}_n}]\\\
&\geq  \E[(2Kn+1)^{-d} p^{|\mathbf{E}_1|} (1-p)^{|\mathbf{E}_2|}\mathbf{1}_{\mathcal{A}^{[1]}_n}]\geq e^ {-c_0n^{4/5}}\P(\mathcal{A}^{[1]}_n),
\end{split}
\end{equation}
where $c_0$ is a positive constant depending on $d$, $s_0$ and $p$.
Set $\tilde{0}:=\SH_{\mathbf{X}}(0)$ and  $t_*:=\kD^{\rm r}(\tilde{0},w^*)$.  We will
prove that on $\mathcal{A}^{\rm r}\cap \mathcal{A}^{[1]}_n$, the following occur:
\begin{enumerate}[label=(\roman*)]
    \item \label{case:i} $t_* \in [s n-3n^{\alpha_d},\infty)$ and $w^* \in \Lambda_{4n^{\alpha_d}}(nx+\tilde{0})$;
    \item \label{case:ii}$\rL_3(w^*)\cap {\rm B}^{\rm r}_{t_*}(\tilde{0})=\{w^*\}$;
    \item \label{case:iii}$\forall k\ne 1\quad \rN_k( \rH_1^*\cap \rmB^{\rm r}_{t_*}(\tilde{0}))\le C_0n^{4/5}$ where $C_0:=16dC^2$;
    \item\label{case:iv}$|\rmB_{t_*}^{\rm r}(\widetilde {0})|\le n^{7/4}$. %$\exists \Gamma'\subset \Z^ d\quad|\Gamma'|\le 2K n,\,\rmB_{t_*}^{\rm r}(\widetilde {0})\subset \mathrm{Int}(\Gamma')$.
\end{enumerate}
% there exists $t\geq \epsilon n-n^{99/100}$ such that for any $k\neq 1$,
% \aln{\label{case1 final}
% w^*\in {\rm B}^{\rm r}_{t}(\tilde{0})\setminus{\rm B}^{\rm r}_{t-1}(\tilde{0}),\,\rL_3(w^*)\cap {\rm B}^{\rm r}_{t}(\tilde{0})=\{w^*\},\, \rN_k( \rH_1^*\cap \rmB^{\rm r}_t(\tilde{0}))\le n^{49/50},\,|\rmB^{\rm r}_t(\tilde{0})|\leq n^{7/4}.
% }
Since $(\tau^{\rm r}_{e+\tilde{0}})\overset{\rm law}{=}(\tau_e)$, together with \eqref{eq:addar}, this will imply
\al{
&\P\left(\begin{array}{c}
\exists i\neq j\in [d],\quad \exists t\geq  s n-3n^{\alpha_d},\quad \exists w_n\in \Lambda_{4n^{\alpha_d}}(nx);\\ 
w_n\in {\rm B}_{t}\setminus {\rm B}_{t-1},\,\rL_i(w_n)\cap \rmB_t=\{w_n\},\, \rN_k(\rH_j(w_n)\cap \rmB_t)\le C_0n^{4/5}\,\forall k\ne j,|\rmB_{t}|\le n^{7/4}
\end{array}\right)\\
&\geq \P\left(\begin{array}{c}
\exists i\neq j\in [d],\quad \exists t\geq  s n-3n^{\alpha_d},\quad \exists w_n\in \Lambda_{4n^{\alpha_d}}(nx+\tilde{0});\,w_n\in {\rm B}^{\rm r}_{t}(\tilde{0})\setminus {\rm B}^{\rm r}_{t-1}(\tilde{0}),\\ 
\rL_i(w_n)\cap B^{\rm r}_t(\tilde{0})=\{w_n\}, \,\rN_k(\rH_j(w_n)\cap \rmB^{\rm r}_t(\tilde{0}))\le C_0n^{4/5}\,\forall k\ne j,|\rmB_{t}^{\rm r}(\widetilde {0})|\le n^{7/4}
\end{array}\right)\\
&\geq e^ {-c_0n^{4/5}}\P(\mathcal{A}^{[1]}_n).
}
 Together with Lemma~\ref{lem:createcutpoin}, this yields the claim. Let us now prove the conditions (i)--(iv).

Since $\|\widetilde 0-0\|_\infty \le \wp_n$ and $\|w^*-w_n\|_\infty \le 2n^{\alpha_d}$, for $n$ large enough depending on $d,s_0$, we have
$w^*\in \Lambda_{4n^{\alpha_d}}(nx+\tilde{0})$. We define a modified path:
\al{
\tilde{\gamma}&:=(\SH(w^{1}),\SH(w^{2}),\cdots,\SH(w^{i_1}))\oplus \bar{\gamma}_{w^{i_1},w^{i_1+1}}\oplus (\SH(w^{i_1+1}),\cdots,\SH(w^{i_2}))\\
&\qquad\qquad \oplus  \bar{\gamma}_{w^{i_2},w^{i_2+1}}\oplus \cdots \oplus (\SH(w^{i_{k}+1}),\cdots,\SH(w^{*}))\oplus L(\SH(w^*),w^*).
}
On the event $\kA^{\rm r}\cap\kA_n^{[1]}$, by \eqref{scenario 1 estimate},  $\tilde{\gamma}$ is an open path from $\tilde{0}=\SH(w^{1})$ to $w^*$ for $\tau^{\rm r}$ such that 
$$|\tilde{\gamma}|\leq s_*+|\overline \rH_1^*\cap \rmB_{s_n}|\times 8\wp_n+\wp_n\leq s_*+ 5\wp_n^2.$$
Hence, since $s_*\le s_n-n ^{\alpha_d}$, for $n$ large enough depending on $p,d,s_0$, 
\aln{\label{upper case1}
\max\{\kD^{\rm r}(\tilde{0},\SH(w^*)),\kD^{\rm r}(\tilde{0},w^*)\}\leq s_*+20C^2n^{4/5}<s_n.
}
Next, we will prove that
\aln{\label{case1 long time}
t_*:=\kD^{\rm r}(\tilde{0},w^*)\geq s_*-\wp_n.
}
To this end, we take a geodesic $\gamma^{\rm r}=(v^i)_{i=1}^\ell$ from $\tilde{0}$ to $\SH(w^*)$ for $\tau^{\rm r}.$ Let \al{
i_1&:=\inf\{k\ge 1:~\,v^k_1\in (w^*_1-\wp_n,w^*_1+\wp_n) \}-1,\\
j_1&:=\inf\{k\geq i_1:~\,v^k_1\in (w^*_1-\wp_n,w^*_1+\wp_n)^c\}.
}
Assuming that $i_1,j_1,\dots,i_{l-1},j_{l-1}$ have been defined, we define 
\al{
i_l&:=\inf\{k>j_{l-1}:~\,v^k_1\in (w^*_1-\wp_n,w^*_1+\wp_n) \}-1,\\
j_l&:=\inf\{k\geq i_l:~\,v^k_1\in (w^*_1-\wp_n,w^*_1+\wp_n)^c \}.
}
We continue this procedure until $i_l=\infty$  (using the convention $\inf\emptyset=\infty$). Let $m$ be the first number such that $i_{m+1}=\infty$.  Note that $i_{l+1}>j_l$ for any $l<m$, since $\gamma^{\rm r}$ is self-avoiding. Let us prove by induction that for all $ {l}\le m$,
\begin{equation}\label{eq:inductioncase1}
\SH^{-1}(v^{i_l})\in\rmB_{s_n-1},\,\SH^{-1}(v^{j_l})\in\rmB_{s_n},\,\tau_{\langle \SH^{-1}(v^{i_l}),\SH^{-1}(v^{j_l})\rangle}=1,\text{ and } \kD^{\rm r}(\tilde{0},v^{i_{l}})\ge \kD(0,\SH^{-1}(v^{i_{l}})).
\end{equation}
We first note that the path $(\SH^{-1}(v^{1}),\dots,\SH^{-1}(v^{i_{1}}))$ is open for $\tau$ by definition of the resampled environment and $(v^1,\cdots, v^{i_1})\subset \SH(\Z^d)$. Hence,  by \eqref{upper case1},
\al{
  s_n>\kD^{\rm r}(\tilde{0},v^{i_{1}})\ge  \kD(0,\SH^{-1}(v^{i_{l}})),
  }which implies $\SH^{-1}(v^{i_1})\in\rmB_{s_n-1}$. Let us assume $\SH^{-1}(v^{i_l})\in \rmB_{s_n-1}$ with $l<m$. Since $\SH^{-1}(v^{i_l})\in\overline \rH_1^*$,  $\SH^{-1}(v^{i_l})\in \rmB_{s_n}\cap \overline \rH_1^*$. On the event $\kA^{\rm r}\cap\kA_n^{[1]}$, we have necessarily that $\langle \SH^{-1}(v^{i_l}),\SH^{-1}(v^{j_l})\rangle$ is an edge colinear to $\mathbf{e}_1$ and corresponds to a crossing of $\gamma_{0,w_n}$ over $\rH^ *_1$. In particular, the edge $\langle \SH^{-1}(v^{i_l}),\SH^{-1}(v^{j_l})\rangle$ is open for $\tau$, and it follows that   $\SH^{-1}(v^{j_l})\in\rmB_{s_n}$. The path $(\SH^{-1}(v^{j_l}),\dots,\SH^{-1}(v^{i_{l+1}}))$ is  open for $\tau$ by definition of the resampled environment. Recalling  $0=\SH^{-1}(v^1)$, by induction hypothesis and \eqref{upper case1},
\begin{equation*}
    \begin{split}
        s_n>\kD^{\rm r}(\tilde{0},v^{i_{l+1}})&= \kD^{\rm r}(\tilde{0},v^{i_{l}})+ \kD^{\rm r}(v^{i_l},v^{j_{l}})+\kD^{\rm r}(v^{j_l},v^{i_{l+1}})\\&\ge  \kD(0,\SH^{-1}(v^{i_{l}}))+1+\kD(\SH^{-1}(v^{j_{l}}),\SH^{-1}(v^{i_{l+1}}))\\&\ge \kD(0,\SH^{-1}(v^{i_{l+1}})).
    \end{split}
\end{equation*}
 It follows that  $\SH^{-1}(v^{i_{l+1}})\in \rmB_{s_n-1}$. Since $\langle \SH^{-1}(v^{i_{l+1}}),\SH^{-1}(v^{j_{l+1}})\rangle$ is an open edge for $\tau$ on $\kA^{\rm r}\cap\kA_n^{[1]}$ as before, $\SH^{-1}(v^{j_{l+1}})\in \rmB_{s_n}$. This concludes the induction.
 
 Since $\SH(w^*)\in\SH(\Z^d)$, by maximality of $m$, the path $(v^{j_m},\dots,v^\ell)$ stays in $\SH(\Z^d)$. In particular, by definition of the resampling, the path
$(\SH^{-1}(v^{j_m}),\cdots, \SH^{-1}(v^\ell))$
is open for $\tau$. Thanks to \eqref{eq:inductioncase1},
\aln{\label{case 1 shape}
\kD^{\rm r}(\tilde{0},\SH(w^*)) &=\kD^{\rm r}(\tilde{0},v^{i_m})+\kD^{\rm r}(v^{i_m},v^{j_m})+\kD^{\rm r}(v^{j_m},\SH(w^*))\\
&\geq \kD(0,\SH^{-1}(v^{i_m}))+1+\kD(\SH^{-1}(v^{j_m}),w^*)\geq  \kD(0,w^*)=s_*.\notag
}
Thus, we obtain  for $n$ large enough depending on $s_0,p,d$, 
\al{
t_*&=\kD^{\rm r}(\tilde{0},w^*)\geq \kD^{\rm r}(\tilde{0},\SH(w^*))-\kD^{\rm r}(w^*,\SH(w^*))\\
&\geq s_*-\wp_n\geq s n-2n^{\alpha_d}-2Cn^{2/5}\ge s n -3n^{\alpha_d}.
}
This concludes the proof of \ref{case:i}.

Next, we will prove that
\aln{\label{case1 shape}
\rmB^{\rm r}_{t_*}(\tilde{0})\subset \SH (\rmB_{s_n})\cup \bigcup_{j=1}^k \bar{\gamma}_{w^{i_j},w^{i_j+1}}\cup L(w^*,\SH(w^*)).
}
The proof is similar to that of \eqref{case1 long time}. To this end, we take $v\in \rmB^{\rm r}_{t_*}(\tilde{0})$ arbitrary and take a geodesic $\gamma^{\rm r}=(v^i)_{i=1}^\ell$ from $\tilde{0}$ to $v$. We define $(i_l,j_l)_{l=1}^m$ as in \eqref{case1 long time} with $v$ in place of $\SH(w^*)$. If $v\in \SH(\Z^d)$, then by \eqref{upper case1}, the same argument as in \eqref{case 1 shape} shows 
$$s_n> t_*\geq \kD^{\rm r}(\tilde{0},v)\geq \kD(0,\SH^{-1}(v)).$$
Hence, $v\in \SH (\rmB_{s_n})$.  Let us assume $v\notin \SH(\Z^d)$. Thanks to \eqref{eq:inductioncase1}, we have $\SH^{-1}(v^{i_m})\in  \rmB_{s_n}\cap \rH^ *_1$. Thus, on the event $\kA^{\rm r}\cap\kA_n^{[1]}$, $v\in \bar{\gamma}_{w^{i_j},w^{i_j+1}}$ with some $j\leq k$ or $v\in L(w^*,\SH(w^*))$, and 
 \eqref{case1 shape} follows.
% Note that \ref{case:ii} directly follows from \eqref{case1 shape}. Moreover, set 
% \[\Gamma':=\SH(\Gamma)\cup \{z\in\Z^d:\mathrm{d}_\infty(z,\cup_{j=1}^k \bar{\gamma}_{w^{i_j},w^{i_j+1}}\cup L(w^*,\SH(w^*)))\le 1\}.\]
% We have $|\Gamma'|\le |\Gamma|+ 10d\wp_n^2\le 2Kn$. Let us prove that $\rmB^{\rm r}_{t_*}(\tilde{0})\subset \mathrm{Int}(\Gamma')$. Let us consider a path $\gamma_0$ from $\rmB^{\rm r}_{t_*}(\tilde{0})$ to infinity. If $\gamma_0$ intersects $ \bigcup_{j=1}^k \bar{\gamma}_{w^{i_j},w^{i_j+1}}\cup L(w^*,\SH(w^*))$, then by definition of $\Gamma'$, $\gamma_0\cap\Gamma'\ne\emptyset$.
 {By the definition of $\bar{\gamma}_{w^{i_j},w^{i_j+1}}$ and isoperimetry \eqref{eq:isoperimetry}, since $\SH$ is injective}, 
 one can check that \ref{case:iv}  follows from \eqref{case1 shape}.
 
Finally, by $|\rmB_{s_n}\cap \rH^ *_1|\leq \wp_n/2$, \eqref{case1 shape}, and $ \rH_1^*\cap \SH(\Z^d)=\emptyset$, we have 
$$| \rH_1^*\cap {\rm B}^{\rm r}_{t_*}(\tilde{0})|\leq \left|\bigcup_{j=1}^k \bar{\gamma}_{w^{i_j},w^{i_j+1}}\cup L(w^*,\SH(w^*))\right|\leq 4d\wp_n^2\leq  16dC^2n^{4/5}.$$
Thus, we have $\rN_k( \rH_1^*\cap \rmB^{\rm r}_t(\tilde{0}))\le 16dC^2n^{4/5}$ for $k\neq 1$. Therefore, we obtain \ref{case:iii} for $n$ large enough depending on $s_0,p,d$. This concludes the proof.
\end{proof}

\begin{proof}[Proof of \eqref{case2}]
 Without loss of generality, we assume $\kA^{[2]}_n$ occurs with $i=1,\,j=2$.  Thanks to \eqref{eq:antalpisztora2}, there exists $c\in (0,1)$ depending only on $s_0,p,d$ such that
 \[\P(\exists w\in\Z^d;\,\|w\|_1\le 2s_0 n, \, n\log n\le\cD^{\kG_p}(0,w)<\infty)\le  (5s_0n)^de^{-cn\log n}.\] 
 Let $\ep:=1/2$. Set $N:=\log^2{n}$ and \[\mathbf{V}_n:=\{x\in [-n^2,n^2]^d{\cap \Z^d}:~x\in\underline{\Lambda}_N(\mathbf{i})\text{ with $\mathbf{i}$: $\ep$-bad}\}.\]
Thanks to Lemma \ref{lem:nobadbox},
we have
\begin{equation}\label{eq:controlBn}
    \P(|\mathbf{V}_n|\ge (2N)^d n)\le e^{-n\log n}.
\end{equation}
Therefore, we have 
 \[\P(\kA_n^{[2]})\le \P(\kA_n^{[2]},|\mathbf{V}_n|\le (2N)^d n,s_n\le n\log n)+   2 (5s_0n)^d e^{-cn\log n}.\]
 We distinguish two cases, either $\P(\kA_n^{[2]})\le e^{-\frac{c}{2}n\log n}$ or $\P(\kA_n^{[2]})> e^{-\frac{c}{2}n\log n}$. There is no need to examine the first case since it is negligible compared to $\P(\kA_n^{[1]})$ by \eqref{eq:boundA1}. For the second case, we have for $n$ large enough,
  \[\P(\kA_n^{[2]})\le 2\P(\kA_n^{[2]},|\mathbf{V}_n|\le (2N)^d n,s_n\le n\log n).\]
 From now on, with abuse of notation, we write $\kA_n^{[2]}$ instead of $\kA_n^{[2]}\cap\{|\mathbf{V}_n|\le (2N)^d n\}\cap\{s_n\le n\log n\}$.
 Let $S'\subset S$ and $\Gamma\subset \Z^ d$ be such that $|\Gamma|\le K n$ and $\rmB_{s_n}\subset\mathrm{Int}(\Gamma)$  as in the event $\kA_n^{[2]}$.  Recall that $\gamma_{0,w_n}$ is a geodesic from $0$ to $w_n$. Since  
 the hyperplanes $(\rH_1(z),z\in S')$ are disjoint by definition of $S'$,  there exist $(z^k)_{k=1}^{n^{\beta_d}/2}\subset S'$ such that for any $k\in[n^{\beta_d}/2]$, \[|({\rm H}_{1}(z^k)\cup {\rm H}_{1}(z^k-\mathbf{e}_1))\cap (\gamma_{0,w_n}\cup \mathbf{V}_n\cup \Gamma)|\leq \frac{5|\gamma_{0,w_n}\cup \mathbf{V}_n\cup \Gamma|}{n^{\beta_d}}.\]Similarly, there exists $ w^* \in  (z^k)_{k=1}^{n^{\beta_d}/2}$ such that \[|({\rm H}_{2}(w^*)\cup {\rm H}_{2}(w^*-\mathbf{e}_2))\cap (\gamma_{0,w_n}\cup \mathbf{V}_n\cup \Gamma)|\leq \frac{5|\gamma_{0,w_n}\cup \mathbf{V}_n\cup \Gamma|}{n^{\beta_d}}.\] Therefore, 
\al{
&|({\rm H}_{1}(w^*)\cup {\rm H}_{1}(w^*-\mathbf{e}_1)\cup {\rm H}_{2}(w^*)\cup {\rm H}_{2}(w^*-\mathbf{e}_2))\cap (\gamma_{0,w_n}\cup \mathbf{V}_n\cup \Gamma)|\\
&\quad\leq  |({\rm H}_{1}(w^*)\cup {\rm H}_{1}(w^*-\mathbf{e}_1))\cap (\gamma_{0,w_n}\cup \mathbf{V}_n\cup \Gamma)|+|({\rm H}_{2}(w^*)\cup {\rm H}_{2}(w^*-\mathbf{e}_2))\cap (\gamma_{0,w_n}\cup \mathbf{V}_n\cup \Gamma)|\\&\quad\leq \frac{10|\gamma_{0,w_n}\cup \mathbf{V}_n\cup \Gamma|}{n^{\beta_d}}.
}
Recall that $s_n\le n\log n$ on $\kA_n^{[2]}$ and so $|\gamma_{0,w_n}|\le n\log n$. Thus, on the event $\kA_n^{[2]}$, we have $|\gamma_{0,w_n}\cup \mathbf{V}_n\cup \Gamma|\leq 4^d n N^{d}$. Let $s_*\in  \{s_n-2n^{\alpha_d}+1,\cdots,s_n-n^{\alpha_d}\}$ be such that $w^*=w^{s_*}$. Set $\rH^*:={\rm H}_1(w^*)\cup {\rm H}_2(w^*)\cup {\rm H}_1(w^*-\mathbf{e}_1)\cup {\rm H}_2(w^*-\mathbf{e}_2).$ It yields %Let $C=40^d>0$. 
\aln{\label{H star condition}
|\rH^* \cap (\gamma_{0,w_n}\cup \mathbf{V}_n\cup \Gamma)|\leq 40^d n^{1-\beta_d}N^{d}.
}
%Set $w^*=w^{s_*}$ and $\rH^ *=\rH(w^{s_*})$.

Let $(\tau_e^*)$ be an independent copy of $(\tau_e)$, and $\mathbf{X}=(X_i)_{i=1}^d$ an independent uniform random variable on $[-Kn,Kn]^d\cap \Z^d$. We consider the following shift: for $x=(x_i)_{i=1}^d\in\Z^ d$,
$$\SH(x)=\SH_{\mathbf{X}}(x):=
\begin{cases}
x+\mathbf{e}_1+\mathbf{e}_2 &\text{if $x_1\geq X_1$ and $x_2\geq X_2$},\\
x-\mathbf{e}_1+\mathbf{e}_2 &\text{if $x_1< X_1$ and $x_2\geq X_2$},\\
x-\mathbf{e}_1-\mathbf{e}_2 &\text{if $x_1< X_1$ and $x_2< X_2$},\\
x+\mathbf{e}_1-\mathbf{e}_2 &\text{if $x_1\ge  X_1$ and $x_2< X_2$}.
\end{cases}$$ 
 \begin{figure}[!ht]
\def\svgwidth{0.3\textwidth}
 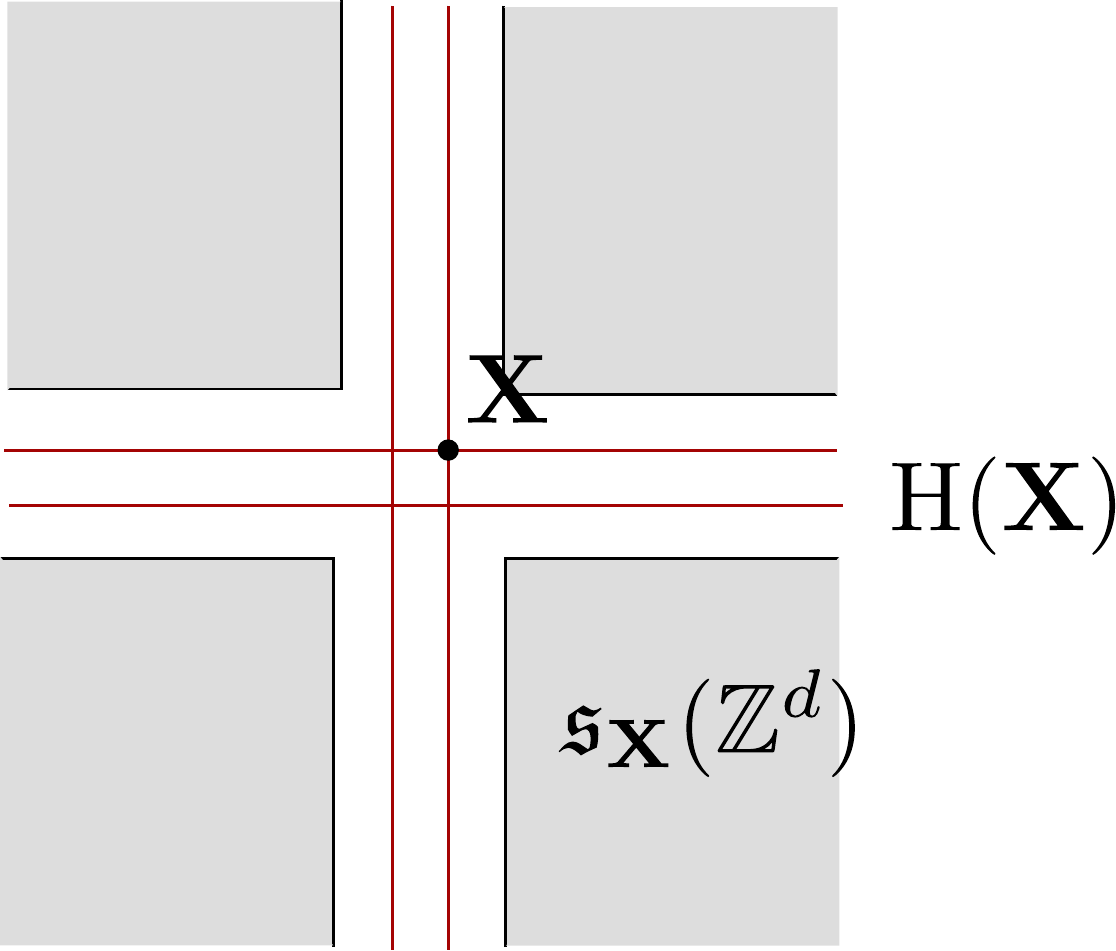
 \caption[fig4]{\label{fig4}Grey region represents $\SH_{\mathbf{X}}(\Z^d)$. Red lines represent $\rH(\mathbf{X})$}
\end{figure}
We note that for any $x\in\Z^d$, $\SH_{\mathbf{X}}(x)_i\neq X_i$ for $i=1,2$, $\rH^*=\Z^d\setminus \SH_{\mathbf{X}}(\Z^d)$ and $\SH_{\mathbf{X}}$ is injective. We define the resampled configuration $(\tau_e^{\rm r})$: for $e=\langle x,y\rangle$, 
\al{\tau^{\rm r}_e=\left\{
\begin{array}{ll}
\tau_{\langle \SH^{-1}(x),\SH^{-1}(y)\rangle}&\text{ if $x,y\in \SH(\Z^d)$},\\
\tau_e^*,&\text{ otherwise}.
\end{array}\right.
}
From now on, we assume $\mathbf{X}=w^*$. We write $\rH_1^*=\rH_{1}(w^*),\rH_2^*=\rH_{2}(w^*)$. For any $x\in \rH^*_1$, we consider the straight path $\bar{\gamma}_{x-\mathbf{e}_1,x}=\bar{\gamma}_{x,x-\mathbf{e}_1}$ between $\SH(x)$ and $\SH(x-\mathbf{e}_1)$; for any $x\in \rH^*_2$, we consider the straight path $\bar{\gamma}_{x-\mathbf{e}_2,x}=\bar{\gamma}_{x,x-\mathbf{e}_2}$ between $\SH(x)$ and $\SH(x-\mathbf{e}_2)$. In particular, these paths are made of three edges.  We say that $e\in \E^d$ crosses $\rH^*$ if there exist $i\in \{1,2\}$ and $z\in \rH^*_i$ such that $e=\langle z-\mathbf{e}_i,z\rangle $.  Let $(i_j)_{j=1}^{k}$ be such that $\langle w^{i_j},w^{i_j+1}\rangle$ is the $j$-th crossing of $\gamma_{0,w_n}$ over $\rH^*$ until $s_*$.  Since $\gamma_{0,w_n}$ is self-avoiding and $w^{s_*}=w^*$, we have $w^{i_j}\ne w^{s_*}$. Moreover, there exists a path $\bar\gamma_{\SH(w^*),w^*}$ between $\SH(w^*)$ and $w^*$  with $|\bar\gamma_{\SH(w^*),w^*}|=2$ that do not have any edge in common with $\bar{\gamma}_{w^{i_j},w^{i_j+1}}$ $j\leq k$;  namely
\[\bar\gamma_{\SH(w^*),w^*}:=\left\{\begin{array}{ll}(\SH_{w^*}(w^*), \SH_{w^*}(w^*)-\mathbf{e}_1, w^*)&\mbox{if $w^{s_*-1}=w^{s_*}-\mathbf{e}_2$},\\
(\SH_{w^*}(w^*), \SH_{w^*}(w^*)-\mathbf{e}_2, w^*)&\mbox{otherwise.}%if  $w^{s_*-1}=w^{s_*}-\mathbf{e}_1$}.
\end{array}\right.\]
Define
\al{
\mathbf{E}_1&:=\{e\in \E^d:~e\in \bar{\gamma}_{w^{i_j},w^{i_j+1}}\text{ with  }j\leq k,\,\text{ or }e\in \bar\gamma_{\SH(w^*),w^*}\},\\
\mathbf{E}_2&:=\{e\in \E^d\setminus \mathbf{E}_1:~\mathrm{d}_\infty(e,(\gamma_{0,w_n}\cup  \mathbf{V}_n\cup \Gamma)\cap \rH^*)\leq 2N\},\\
\mathbf{E}_3&:=\{e\in \E^d\setminus \mathbf{E}_1:~|e\cap \rL_3(w^*)|=1,\,\mathrm{d}_\infty(e,\mathrm{Int}(\Gamma)\cap \rH^*)\leq N\}.
}Let $\overline{\rm H}_1^*={\rm H}_1(w^*)\cup {\rm H}_1(w^*-\mathbf{e}_1)$ and $\overline{{\rm H}}_2^*={\rm H}_2(w^*)\cup {\rm H}_2(w^*-\mathbf{e}_2)$. We now prove \aln{\label{E3 estimate}
|\mathbf{E}_3|\leq 12d N (\rN_2(\mathrm{Int}(\Gamma)\cap \overline{\rm H}_1^*)+\rN_1(\mathrm{Int}(\Gamma)\cap \overline{\rm H}_2^*)).
}
 Let us take $e=\langle u,v\rangle \in \mathbf{E}_3$ with $u=u^e\in \rL_3(w^*)$. Note that $u_3=v_3$. Then, we select $z^e\in \mathrm{Int}(\Gamma)\cap \rH^*$ such that $|z^e_3-u_3|\leq N$. If $z^e\in \mathrm{Int}(\Gamma)\cap \overline{\rH}_1^*$, setting $g^e_2=0$ and $g_i^e=z_i^e$ for $i\neq 2$, we have $z^e\in\rL_2(g^e)\cap \mathrm{Int}(\Gamma)\cap \overline{\rm H}_1^*$, in particular $\rL_2(g^e)\cap \mathrm{Int}(\Gamma)\cap \overline{\rm H}_1^*\neq \emptyset$. Note that $u^e_3\in [g_3^e-N,g_3^e+N]\cap \Z$. Similarly, if $z^e\in \mathrm{Int}(\Gamma)\cap \rH_2^*$, setting $r^e_1=0$ and $r_i^e=z_i^e$ for $i\neq 1$, we have $\rL_1(r^e)\cap \mathrm{Int}(\Gamma)\cap \overline{\rm H}_2^*\neq \emptyset$. Note that $u^e_3\in [r_3^e-N,r_3^e+N]\cap \Z$. Therefore, \eqref{E3 estimate} follows from the fact that $|A|\leq  |B|\max_{y\in B}|f^{-1}(y)|$ for any map $f:A\to B$, and for any $g=(g_i)\in \Z^d$ and $r=(r_i)\in \Z^d$ with $g_2=0$ and $r_1=0$, $$\max\{|\{e\in \mathbf{E}_3:\,~g=g^e\}|,|\{e\in \mathbf{E}_3:\,~ r=r^e\}|\}\leq  12dN.$$   

    Next we estimate $\rN_k(\{z\in \Z^d:~\mathrm{d}_\infty(z,\mathrm{Int}(\Gamma)\cap \overline{\rm H}_1^*)\leq N\})$ for $k\neq 1$. Note that we here consider a stronger estimate for a later purpose. Let $z\in\Z^d$ be such that $z_k=0$ and $\rL_k(z)\cap \{z'\in \Z^d:~\mathrm{d}_\infty(z',\mathrm{Int}(\Gamma)\cap \overline{\rm H}_1^*)\leq N\} \neq \emptyset$. Then, there exists $z^0=(z^0_i)\in \Z^d$ such that $\|z-z^0\|_\infty\leq N$, $z^0_k=0$, and $ \rL_k(z^0)\cap \mathrm{Int}(\Gamma)\cap \overline{\rm H}_1^*\neq \emptyset$. 
It follows that $ \rL_k(z^0)\cap \Gamma\cap \overline{\rm H}_1^*\neq \emptyset$. Recalling that $|\Gamma\cap \rH^*|\le 40^d n^{1-\beta_d}N^d$, this yields
\begin{equation}\label{3 estimate}\begin{split}
  \rN_k&(\{z\in \Z^d:~\mathrm{d}_\infty(z,\mathrm{Int}(\Gamma)\cap \overline{\rm H}_1^*)\leq N\})\\
 &\hspace{1cm}\leq  |\{z^0\in \Z^d:~\|z^0\|_\infty\leq N\}|\times |\Gamma\cap \overline{\rm H}_1^*|\le 10^{2d} n^{1-\beta_d} N^{2d}.
 \end{split}
\end{equation}
 By the same argument, we have for $k\neq 2$,
 $$\rN_k(\{z\in \Z^d:~\mathrm{d}_\infty(z,\mathrm{Int}(\Gamma)\cap \overline{\rm H}_2^*)\leq N\})\le  10^{2d} n^{1-\beta_d}N^{2d}.$$
 Combined with \eqref{E3 estimate} and \eqref{3 estimate}, this gives
 \begin{equation}\label{eq:E3estimate}
 |\mathbf{E}_3|\le 10^{3d}n^{1-\beta_d}N^{2d+1}.
 \end{equation}

We consider the event 
\al{
\mathcal{A}^{\rm r}:=\left\{\mathbf{X}=w^*;\,
\tau^*_e=1\,\text{ for }e\in \mathbf{E}_1;\,\tau^*_e=\infty\,\text{ for }e\in \mathbf{E}_2\cup \mathbf{E}_3
\right\}.
}
Thanks to \eqref{H star condition}, $|\mathbf{E}_1\cup\mathbf{E}_2|\le  {10^{3d}}n^{1-\beta_d} N^{2d}$. Recall that $1-\beta_d<\alpha_d$. Combining  this with \eqref{eq:E3estimate}, we have for $n$ large enough depending on $s_0,d,p$,
\al{
\P(\mathcal{A}^{\rm r}\cap \mathcal{A}^{[2]}_n)&=\E[\P(\mathcal{A}^{\rm r}|~\mathcal{A}^{[2]}_n)\mathbf{1}_{\mathcal{A}^{[2]}_n}]\\
&\geq  \E[(2Kn+1)^{-d} p^{|\mathbf{E}_1|} (1-p)^{|\mathbf{E}_2\cup \mathbf{E}_3 |}
 \mathbf{1}_{\mathcal{A}^{[2]}_n}]\geq e^ {-n^{\alpha_d}}\P(\mathcal{A}^{[2]}_n).
 }
%  \sncomment{I'm not sure if the explanations are necessary because they already appeared before the proof and the (i)-(iv) below are easier (for me?) to know a situation.}
%  {\SN 
% The aim of what follows is to prove that the ball for the resampled times has a structure   not that different from the original ball, i.e. it is roughly made of the same points and the chemical distances to $0$ from these points are roughly unchanged. Moreover, the new ball is built in such a way that it only intersects $\rL_3(w^*)$  at $w^*$. The main problem is that the intersection of $\rmB_{s_n}$ with the hyperplanes is potentially of higher order than $n$, so we cannot preserve exactly all the connections as in the first case. We only preserve the connection of the geodesics and prevent connections around bad boxes or around the boundary. We need to check that it is sufficient. In particular, we need to check that the resampling does not create a shortcut that changes  a metric structure.}
 Let $\tilde{0}:=\SH_{\mathbf{X}}(0)$ and $t_*:=\kD^{\rm r}(\tilde{0}, w^*)$. We will prove that on $\mathcal{A}^{\rm r}\cap \mathcal{A}^{[2]}_n$, the following occur:
\begin{enumerate}[label=(\roman*)]
    \item \label{case2:i}  $t_* \in [ s n-3n^{\alpha_d},\infty)$ and $w^* \in \Lambda_{4n^{\alpha_d}}(nx+\tilde{0})$;
    \item \label{case2:ii}$\rL_3(w^*)\cap {\rm B}^{\rm r}_{t_*}(\tilde{0})=\{w^*\}$;
    \item \label{case2:iii}$\forall k\ne 1\quad \rN_k( \rH_1^*\cap \rmB^{\rm r}_{t_*}(\tilde{0}))\le n^{\alpha_d}$;
     \item\label{case2:iv} $|\rmB_{t_*}^{\rm r}(\widetilde {0})|\le n^{7/4}$.
     %$\exists \Gamma'\subset \Z^ d\quad|\Gamma'|\le 2K n,\,\rmB_{t_*}^{\rm r}(\widetilde {0})\subset \mathrm{Int}(\Gamma')$.
\end{enumerate}
%there exists $t\geq \epsilon n-n^{99/100}$ such that for $k\ne 1$,
%$$w^*\in {\rm B}^{\rm r}_{t}(\tilde{0})\setminus{\rm B}^{\rm r}_{t-1}(\tilde{0}),\,\rL_3(w^*)\cap {\rm B}^{\rm r}_{t}(\tilde{0})=\{w^*\},\, \rN_k( \rH_1^*\cap \rmB^{\rm r}_t(\tilde{0}))\le n^{49/50},\,|\rmB^{\rm r}_t(\tilde{0})|\leq n^{7/4}.$$
Since $(\tau^{\rm r}_{e+\tilde{0}})\overset{\rm law}{=}(\tau_e)$, this implies
\al{
&\P\left(\begin{array}{c}
\exists i\neq j\in [d],\quad \exists t\geq  s n-3n^{\alpha_d},\quad \exists w_n\in \Lambda_{4n^{\alpha_d}}(nx);\\ 
w_n\in {\rm B}_{t}\setminus {\rm B}_{t-1},\,\rL_i(w_n)\cap \rmB_t=\{w_n\},\, \rN_k(\rH_j(w_n)\cap \rmB_t)\le n^{\alpha_d}\,\forall k\ne j,\,|\rmB_{t}|\le n^{7/4}
\end{array}\right)\\
&\geq \P\left(\begin{array}{c}
\exists i\neq j\in [d],\quad \exists t\geq  s n-3n^{\alpha_d},\quad \exists w_n\in \Lambda_{4n^{\alpha_d}}(nx+\tilde{0});\,w_n\in {\rm B}^{\rm r}_{t}(\tilde{0})\setminus {\rm B}^{\rm r}_{t-1}(\tilde{0}),\\ 
\rL_i(w_n)\cap \rmB^{\rm r}_t(\tilde{0})=\{w_n\},\,\rN_k(\rH_j(w_n)\cap \rmB^{\rm r}_t(\tilde{0}))\le n^{\alpha_d}\,\forall k\ne j,\,|\rmB_{t}^{\rm r}(\widetilde {0})|\le n^{7/4}
\end{array}\right)\\
&\geq e^ {-n^{\alpha_d}}\P(\mathcal{A}^{[2]}_n).
}
Combined with Lemma~\ref{lem:createcutpoin}, this yields the claim. 

We prepare some claims to prove \ref{case2:i}--\ref{case2:iv}.
 Given $x\neq y\in \rH^*$, we define
$$x\overset{\rm H}{\sim} y \overset{\rm def}{\Leftrightarrow} %\langle x,y\rangle\in \gamma_{0,w_n}\text{ or }
\exists \gamma:x\leftrightarrow y\subset \rH^*\setminus (\gamma_{0,w_n}\cup\Gamma \cup \mathbf{V}_n),$$
    with the convention $x\overset{\rm H}{\sim}x$. Let $\kC_{\rH}(x)$ be the set of all points connected to $x\in\rH^* $ for the relation $\overset{\rm H}{\sim}$.  Note that this relation only depends on the configuration $\tau$. Note that 
\ben{\label{H relation conclustion}
x\in \gamma_{0,w_n}\Rightarrow \kC_\rH(x)=\{x\},\quad \langle x,y\rangle\in \gamma_{0,w_n}\Rightarrow y\notin\kC_H(x).
}%We  need the following claims.

\begin{claim}\label{lem En relation}
Assume $\mathcal{A}^{\rm r}\cap \mathcal{A}^{[2]}_n$. Let $a, b\in \SH(\rH^*)$ and a self-avoiding open path $\gamma:a\leftrightarrow b$   for $\tau^{\rm r}$
 such that $\gamma\backslash\{a,b\} \subset \rH^*$.
Then $\SH^{-1}(a)\overset{\rm H}{\sim}\SH^{-1}(b)$ or $\langle \SH^{-1}(a),\SH^{-1}(b)\rangle\in\gamma_{0,w_n}$.
%     \item \label{case2lemrelH} Suppose $b\in \rH^*$. Then  $\SH^{-1}(a)\overset{\rm H}{\sim}b$.

\end{claim}

\begin{proof}
Note that $\SH^{-1}(a)\oplus (\gamma\setminus\{a,b\})\oplus\SH^{-1}(b)$ is a $\Z^d$-path inside $\rH^*$. 
We extract a self-avoiding path $\gamma'$ from this path. If $\gamma'\cap (  {\gamma_{0,w_n}\cup}\Gamma\cup \mathbf{V}_n)=\emptyset$, then by definition we have
$\SH^{-1}(a)\overset{\rm H}{\sim}\SH^{-1}(b)$.  Otherwise, there exists $e\in\gamma\setminus\{a,b\}$ at $\ell_1$-distance less than $1$ from $ {\gamma_{0,w_n}\cup}\Gamma\cup \mathbf{V}_n$. On the event $\kA^{\rm r}\cap\kA_n^{[2]}$, this implies that $e\in\mathbf{E}_1$ since otherwise the edge would be closed for $\tau^{\rm r}$.  On the event $\kA^{\rm r}\cap\kA_n^{[2]}$, since the edges with one endpoint in common with the set $\mathbf{E}_1$ are closed, the path $\gamma\setminus \{a,b\}$ has to be included in  $\bar \gamma_{\SH(w^*),w^*}$ or $\bar{\gamma}_{w^{i_j},w^{i_j+1}}$ for some $j$.  {However, since $\gamma$ is a path between $a,b\in \SH(\rH^*)$, it cannot be included in $\bar \gamma_{\SH(w^*),w^*}$. Therefore,} $a$ and $b$ are the extremities of $\bar{\gamma}_{w^{i_j},w^{i_j+1}}$ for some $j$ and $\langle\SH^{-1}(a),\SH^{-1}(b)\rangle\in\gamma_{0,w_n}$.

%Consider Case \eqref{case2lemrelH}. The path $\SH^{-1}(a)\oplus (\gamma\setminus\{a\})$ is inside $\rH^*$, we can extract from it a self-avoiding path  $\gamma'$. By the same reasoning as above, the case $\gamma'\cap (\Gamma\cup\mathbf{V}_n)=\emptyset$ is clear and if $\gamma'\cap (\Gamma\cup\mathbf{V}_n)\neq \emptyset$, then $\gamma\setminus\{a,b\}$ is included in $\bar \gamma_{\SH(w^*),w^*}$ or $\bar{\gamma}_{w^{i_j},w^{i_j+1}}$ for some $j$. We conclude that $\|\SH ^{-1}(a)-b\|_1=1$ and $\SH^{-1}(a)\overset{\rm H}{\sim}b$. {\BD When do we use case 2 ?}
\end{proof}
Recall that $\sC(\mathbf{a})$ denotes the largest open cluster of the $\ep$-good box $\underline{\Lambda}'_N(\mathbf{a})$. 
\begin{claim}\label{claim:bounddis}  Let $\mathbf{a},\mathbf{b}\in\Z^ d$ be good macroscopic sites. %On the event $\kA^{[2]}_n$, i
 Let $a\in \rH^*\cap\sC(\mathbf{a})$ and $b\in  \sC(\mathbf{b})\cap  \kC_{\rH}(a)$. It holds that
$$\kD(a,b)\leq  8d^2\mu(\mathbf{e}_1)N |\{e\in \E^d:~e\subset \rH^*,\,|e\cap \kC_{\rH}(a)|=1\}|.$$
\end{claim}
\begin{proof}
We claim that there exists a $\Z^d$-path $\gamma\subset \kC_H(a)$ from $a$ to $b$ in $\rH^*$ such that
 \[|\gamma|\leq  {4 d}|\{e\in \E^d:~e\subset \rH^*,\,|e\cap \kC_{\rH}(a)|=1\}|.\]
Indeed, consider a $\Z^d$-path $\gamma'$ from $a$ to $b$ such that $\gamma'\subset \rH^*$ and  $|\gamma'|=\|a-b\|_1$; we arbitrarily take a path  $\gamma$ from $a$ to $b$ in 
\[(\gamma'\cap \kC_{\rH}(a))\cup \{x\in \kC_{\rH}(a): \exists y\in {\rH^*\setminus}\kC_{\rH}(a)\,\, \|x-y\|_ {\infty}=1\}.\]
 {Note that the set $\{x\in \kC_{\rH}(a): \exists y\in {\rH^*\setminus}\kC_{\rH}(a)\,\, \|x-y\|_ {\infty}=1\}$ is $\Z^d $-connected (see for instance \cite[Lemma 2.23]{Kesten:StFlour}).
Moreover, we have 
\[|\{x\in \kC_{\rH}(a): \exists y\in {\rH^*\setminus}\kC_{\rH}(a)\,\, \|x-y\|_ {\infty}=1\}|\le 2d|\{e\in \E^d:~e\subset \rH^*,\,|e\cap \kC_{\rH}(a)|=1\}|.\]
}
Note that since $b\in   \kC_{\rH}(a)$,
$$\|a-b\|_1\leq |\{x\in \kC_{\rH}(a): \exists y\in {\rH^*\setminus}\kC_{\rH}(a)\,\, \|x-y\|_1=1\}|\leq |\{e\in \E^d:~e\subset \rH^*,\,|e\cap \kC_{\rH}(a)|=1\}|.$$ 
This yields the existence of such a path $\gamma$. By construction, $\gamma$ only intersects good boxes since $\kC_{\rH}(a)\cap (\gamma_{0,w_n}\cup\Gamma\cup\mathbf{V}_n)=\emptyset$ and $\gamma \subset \kC_{\rH}(a)$. 
The path $\gamma$ intersects at most $|\gamma|$ boxes and all these boxes are good by construction and connected. By Lemma \ref{lem:connexiongoodbox}, there exists an open path from $a$ to $b$ of length at most $2d\mu(\mathbf e_1)N|\gamma|$ and
% {\SN In particular, by   $a\in\sC(\mathbf{a})$, $b\in\sC(\mathbf{b})$ and $\mu(\mathbf{e}_1)\geq 1$, we can use this sequence of good boxes to build an open path for $\tau$ (see the proof of Proposition \ref{prop:connexpoints})}  such that
% \al{
% \kD(a,b)&\leq 2d\mu(\mathbf{e}_1)\|a-b\|_1+ {2d}|\{e\in \E^d:~e\subset \rH^*,\,|e\cap \kC_{\rH}(a)|=1\}|\\
% &\leq 4d\mu(\mathbf{e}_1)|\{e\in \E^d:~e\subset \rH^*,\,|e\cap \kC_{\rH}(a)|=1\}|.
% }
{
\al{
\kD(a,b)&\leq 2d\mu(\mathbf{e}_1)N |\gamma|\leq 8d^2\mu(\mathbf{e}_1)N|\{e\in \E^d:~e\subset \rH^*,\,|e\cap \kC_{\rH}(a)|=1\}|.
}}
%This concludes the proof.
\end{proof}

\noindent{\bf Proof of \ref{case2:i}.}
%Under the condition 
Note that  $\|0-\widetilde{0}\|_\infty \le 1$ and $\|w^*-w_n\|_\infty <  {2n^{\alpha_d}}$. It follows from $w_n\in \Lambda_{n^{\alpha_d}}(nx)$ that $w^*\in\Lambda_{4n^{\alpha_d}}(\widetilde{0}+nx)\,.$ We define a modified path as
\al{\tilde{\gamma}&:=(\tilde{0},\SH(w^{2})\cdots,\SH(w^{i_1}))\oplus \bar{\gamma}_{w^{i_1},w^{i_1+1}}\oplus (\SH(w^{i_1+1}),\cdots,\SH(w^{i_2}))\oplus  \bar{\gamma}_{w^{i_2},w^{i_2+1}}\\
&\qquad \oplus \cdots\oplus (\SH(w^{i_k+1}),\cdots,\SH(w^*))\oplus \bar{\gamma}_{\SH(w^*),w^*}.
}
Recall that $s_*\le s_n-n ^{\alpha_d}$. On the event  $\mathcal{A}^{\rm r}$,  $\tilde{\gamma}$ is still an open path from $\tilde{0}$ to $w^*$ for $\tau^{\rm r}$. Hence, we have using \eqref{H star condition} for $n$ large enough depending on $d,s_0,p$, \aln{\label{cases 2 t star estimate}
\kD^{\rm r}(\tilde{0},w^*)\leq |\tilde{\gamma}|\leq |\gamma|+3|\rH^*\cap\gamma_{0,w_n}|+ {|\bar{\gamma}_{\SH(w^*),w^*}|}\leq s_*+10^{3d} n^{1-\beta_d}N^d<s_n- \frac{1}{2}n^{\alpha_d}.
}
% {\BD We should move this lemma.}
%\begin{lem}\label{not exit lemma}
% Let $a\in \SH(B_{s-2})$, $b\in %\SH(\kC_{\rH}(\SH^{-1}(a)))\setminus \SH(B_{s-1})$ and $c\in \SH(\rH^*)\setminus \SH(\kC_{\rH}(\SH^{-1}(a)))$. Then, there  is no open path $\gamma$ from $b$ to $c$ for $\tau^{\rm r}$ such that $\gamma\subset \SH(\Z^d)$.
% \end{lem}
% \begin{proof}
% On the contrary, suppose that we can take a open path $\gamma=(v^i)_{i=1}^{|\gamma|}$ from $b$ to $c$ with $\gamma\subset \SH(\Z^d)$. By the definition of $\overset{\rH}{\sim}$, $a\notin \SH(\kC_{\rH}(z))$ for any $z\in \gamma_{0,w_n}\cap \rH^*$. Let $\ell=\min\{i>1:~v^i\in \SH(\rH^*)\}$. Since $a$ and $b$ are connected for $\tau^{\rm r}$, we can find a macroscopic sequence of good boxes from  $B_N(\SH^{-1}(a))$ to $B_N(\SH^{-1}(b))$. Hence, since $B_{s-1}$ contains the largest open cluster in $B_N(\SH^{-1}(b))$ and $b\notin  \SH(B_{s-1})$, the diameter of the open cluster for $\tau$ including $\SH^{-1}(b)$ is less than $N/2$. Therefore, $|b-c|_1\leq N/2$. Moreover, since $B_N(z)$ is good  for any $z$ with $|z-b|_1\leq 2N$ on $\kA^{\rm r}$, $c\in \SH(\kC_{\rH}(\SH^{-1}(a)))$, which leads to a contradiction. 
% \end{proof}

Let $v\in \rmB^{\rm r}_{s_n-n^{\alpha_d}/2}$. Let us take a geodesic $\gamma^{\rm r}:\tilde{0}\to v$ with $\gamma^{\rm r}=(v^i)_{i=1}^{\ell}$  for $\tau^{\rm r}$. 
We define $i_1,j_1,\dots,i_m,j_m$ inductively as follows: Let
\al{
i_1&:=\inf\{k:~\,v^k\in\SH(\rH^* ),v^{k+1}\in\rH^*\},\\
j_1&:=\sup\{k\geq i_1:~v^k\in \SH(\kC_{\rH}(\SH^{-1}(v^{i_1})))\text{ or } \langle \SH^{-1}(v^{i_1}),\SH^{-1}(v^k) \rangle\in \gamma_{0,w_n}\}.
}
Assuming that $i_1,j_1,\dots,i_{l-1},j_{l-1}$ have been defined, we define 
\al{
i_l&:=\inf\{k\geq j_{l-1}:~\,v^k\in\SH(\rH^* ){\setminus \SH(\kC_H(\SH^{-1}(v^{i_{l-1}}))}),v^{k+1}\in\rH^*\},\\
j_l&:=\sup\{k\geq i_l:~v^k \in \SH(\kC_{\rH}(\SH^{-1}(v^{i_l})))\text{ or } \langle \SH^{-1}(v^{i_l}),\SH^{-1}(v^k) \rangle\in \gamma_{0,w_n} \}.
}
% Let $i_1$ be the first hitting time of $(v^i)_{i=1}^\ell$ in $\SH(\rH^*)$. Let $j_1$ be the last time of $(v^i)_{i=1}^\ell$ in $\SH(\kC_{\rH}(x^1))$ with $x^1= \SH^{-1}(v^{i_1})$. Set $y^1=\SH^ {-1}(v^{j_1})$. Suppose that we have defined $(i_k,j_k)_{k\leq m}$. Then, let $i_{m+1}$ be the first hitting time  after $j_m$ of $(v^i)_{i=1}^\ell$ in $\SH(\rH^*)$. Let $j_{m+1}$ be the last time of $(v^i)_{i=1}^\ell$ in $\SH(\kC_{\rH}(x^{m+1}))$ with $x^{m+1}=\SH^{-1}(v^{i_{m+1}})$. Set $y^{m+1}=\SH^ {-1}(v^{j_{m+1}})$. 
Let $m$ be the smallest integer such that $i_{m+1}{=\infty}$. 
For $k\le m$, we set \[x^k:=\SH^{-1}(v^{i_k})\quad\text{and} \quad y^k:=\SH^{-1}(v^{j_k}).\]
 {Note that if $i_{l+1}<\infty$,  by Claim~\ref{lem En relation}, then $v^k \in \SH(\kC_{\rH}(\SH^{-1}(v^{i_l})))\text{ or } \langle \SH^{-1}(v^{i_l}),\SH^{-1}(v^k) \rangle\in \gamma_{0,w_n}$ with $k:=\min\{j>i_l:~v^j\in \SH(\Z^d)\}$, thus we have $i_l<j_{l}$. Therefore, $m$ is finite.}
\begin{claim}\label{claim:openpath}  The paths $(v^{j_k},\cdots, v^{i_{k+1}})$ for $k<m$ are contained in $\SH(\Z^d)$. In particular, the paths $(\SH^{-1}(v^{j_m}),\cdots,\SH^{-1}( v^{\ell}))$ for $k<m$ are open for $\tau$.
Moreover, the path $(v^{j_m+1},\cdots,v^{\ell})$ is either included in $\SH(\Z^d)$ or included in $\rH^*$.
\end{claim}
\begin{proof}
Let $k<m$.  {Note that $j_k>i_k$ mentioned above, and thus $x^k\neq y^k$.} If $j_k=i_{k+1}$, then the claim is trivial. Hence, we assume $j_k<i_{k+1}.$ %which implies $x^k\overset{H}{\sim} y^k$ since $.$ 
%We claim that for any $i\in (j_k,i_{k+1})$, $v^{i}\in \SH(\Z^d)$. 
To prove the claim, we assume the contrary, i.e. there exists $i\in (j_k,i_{k+1})$ such that $v^{i}\in  \rH^*$. Consider the smallest such $i.$
If $i>j_k+1$, then the edge $\langle v^{i-1}, v^{i} \rangle $ crosses $\SH(\rH^*)$ before $i_{k+1}$ and $v^{i-1}\in \SH(\rH^* ) \setminus \SH(\kC_H(x^k))$ due to $i-1>j_k$, which contradicts the minimality of $i_{k+1}$. Hence, we have  $i=j_k+1$, i.e. $v^{j_k+1}\in  \rH^*$. Note that, by \eqref{H relation conclustion}, $\langle x^k,y^k\rangle\in\gamma_{0,w_n}$ {implies $y^k\notin \kC_{\rm H}(x^k)$ and  $j_k=i_{k+1}$, which contradicts the assumption  $j_k<i_{k+1}$.   Thus, it holds  $x^k\overset{H}{\sim} y^k$. Let $j:=\min\{i'\geq i:~v^{i'}\in \SH(\Z^d)\}$. Since $\langle y^k,\SH^{-1}(v^j)\rangle\in\gamma_{0,w_n}$ again implies $j_k=i_{k+1}$ by \eqref{H relation conclustion},  we have $v^j\in \SH(\kC_H(y^k))=\SH(\kC_H(x^k))$ by Claim~\ref{lem En relation}, which contradicts the maximality of $j_{k}$.} Therefore the path $(v^{j_k},\cdots, v^{i_{k+1}})$ is contained in $\SH(\Z^d)$. Hence, by definition of $\tau^{\rm r}$, the path $(\SH^{-1}(v^{j_k}),\cdots,\SH^{-1}( v^{i_{k+1}}))$ is open for $\tau$. 

%Therefore, we have for any $i\in (j_k,i_{k+1})$, $v^{i}\in \SH(\Z^d)$.
%Indeed, if it was not the case, then for any $i\in (j_k,i_{k+1})$, $v^{i}\in \rH^*$ since {\SN otherwise the path $(v^i)_{j_k +1}^{i_{k+1}-1}$ crosses $\SH(\rH^*)$ before $i_{k+1}$.}\sncomment{What if it crosses at $j_k$?}  
%By Claim~\ref{lem En relation}, this implies $x^{k+1}\in \kC_{\rH}(x^k)$  {or $\langle  x^{k},x^{k+1}\rangle \in \gamma_{0,w_n}$,} which contradicts the definition of $j_k$.
We assume  $(v^{j_m+1},\cdots,v^{\ell})$ is not included in $\SH(\Z^d)$, i.e. there exists $i\in \{j_m+1,\dots,\ell\}$ such that  $v^{i}\in \rH^*$. We take the smallest such $i$. If $i>j_{m}+1$, then by the same reasoning as above,  it contradicts the maximality of $m$. Hence, we have $i=j_{m}+1$. We further suppose that  $(v^{j_m+1},\cdots,v^{\ell})$ is not included in $\rH^*$.  Define $j=\min\{j> j_{m}:~v^j\in \SH(\Z^d)\}$. 
If $y^m\in \kC_\rH(x^m)$, then  by \eqref{H relation conclustion} and Claim~\ref{lem En relation}, $\langle x^m,\SH^{-1}(v^{j})\rangle=\langle y^m,\SH^{-1}(v^{j})\rangle\in \gamma_{0,w_n}$ or  $\SH^{-1}(v^{j})\in \kC_{\rH}(y^m)=\kC_{\rH}(x^m)$, which contradicts the maximality of $j_m.$ 
If $y^m\notin \kC_\rH(x^m)$, then %since $\langle x^m,y^m\rangle\in \gamma_{0,w_n}$ and $y^m\notin \kC_\rH(x^m)=\{x^m\}$, 
we have $j_m=i_{m+1}$ by definition of $i_k$, which contradicts the maximality of $m.$ Therefore, $(v^{j_m+1},\cdots,v^{\ell})$ is included in $\rH^*$. %the path contains an edge between $ \SH(\rH^*)$  and $\rH^*$,  which contradicts the maximality of $m$}. 
\end{proof}
Let us prove by induction that \aln{\label{second trial first goal}
 \text{$\forall k< m$, $x^k,y^k\in \rmB_{s_n-1}$  and $x^m\in \rmB_{s_n-1}$.}}
Let $k\leq m$ and suppose that we have proved the claim for any $i< k$, i.e.  $x^i,y^i\in \rmB_{s_n-1}$. Thanks to Claim \ref{claim:openpath}, we have
\aln{\label{triangle ineq}
\kD(0,x^1)+\sum_{i=1}^{k-1} \kD(y^i,x^{i+1})\leq\kD^{\rm r}(\tilde{0},\SH(x^1))+\sum_{k=1}^{k-1} \kD^{\rm r}(\SH(y^i),\SH(x^{i+1}))\leq \kD^{\rm r}(\tilde{0},v^{i_{k}})\leq \ell.
}
Let $C:=10^{3d}\mu(\mathbf{e}_1)$. Next, let us prove
\ben{\label{Hn relation bound}
\sum_{i=1}^{k-1} \kD(x^i,y^i)\leq C N^{d+1} n^{1-\beta_d}.
}
%we have $y^i\in\kC_{\rH}(x^i)$ and
For any $i<k$,  either (1) $\langle x^i,y^i\rangle\in\gamma_{0,w_n}$, 
    or (2) $y^i\in \kC_{\rm H}(x^i)$ holds. 
 For case (1), we have $\kD(x^i,y^i)= 1$.
For case (2), by induction hypothesis $x^i,y^i\in\rmB_{s_n-1}$, they belong to the largest open cluster of their respective good boxes.
Thanks to Claim \ref{claim:bounddis}, we have 
$$\kD(x^i,y^i)\leq   8d^2\mu(\mathbf{e}_1)N|\{e\in \E^d:~e\subset \rH^*,\,|e\cap \kC_{\rH}(x^i)|=1\}|.$$
% since $y^i\in\kC_{\rH}(x^i)$, we claim that there exists a path $\gamma^i$ from $x^i$ to $y^i$ in $\rH^*$ that does not intersect $\Gamma\cup\mathbf{V}_n$ such that
%  \[|\gamma^i|\leq 2d|\{e\in \E^d:~e\subset \rH^*,\,|e\cap \kC_{\rH}(x^i)|=1\}|.\]
% To prove it, consider $\gamma'$ to be a $\Z^d$-path from $x^i$ to $y^i$ such that $|\gamma'|=\|x^i-y^i\|_1+1$.
% The existence of the path $\gamma^i$ follows from the fact that there exists a path from $x^i$ to $y^i$ in 
% \[\gamma'\cap \kC_{\rH}(x^i)\cup \{x\in \kC_{\rH}(x^i): \exists y\notin\kC_{\rH}(x^i)\,\, \|x-y\|_1=1\}.\]
% By construction $\gamma^i$ only intersects good boxes since it does not intersect the set $\mathbf{V}_n$. Moreover, since $x^i,y^i\in \rmB_{s_n-1}$ both are contained in the largest cluster of their corresponding good box. In particular, we can use this sequence of good boxes to build an open path for the configuration $\tau$ (see the proof of Proposition \ref{prop:connexpoints}) such that
% $$\kD(x^i,y^i)\leq  4d\mu(\mathbf{e}_1)|\{e\in \E^d:~e\subset \rH^*,\,|e\cap \kC_{\rH}(x^i)|=1\}|.$$
Finally, since $|\{i\leq k:~|e\cap  \kC_{\rH}(x^i)|=1\}|\leq 2$ for any  $e\subset  \rH^*$, using \eqref{H star condition}
\begin{equation}\label{eq:conttaille}
\begin{split}
\sum_{i=1}^{k-1} \kD(x^i,y^i)&\leq   8d^2\mu(\mathbf{e}_1)N \sum_{i=1}^{k-1} |\{e\in \E^d:~e\subset \rH^*,\,|e\cap \kC_{\rH}(x^i)|=1\}|+|\rH^*\cap \gamma_{0,w_n}|\\
&\leq    {32d^3\mu(\mathbf{e}_1)N}\mu(\mathbf{e}_1)  |\rH^*\cap (\mathbf{V}_n\cup \Gamma\cup  \gamma_{0,w_n})|\leq C N^{d+1} n^{1-\beta_d}.
\end{split}
\end{equation}
Therefore, since $\ell\le  s_n- \frac{1}{2}n^{\alpha_d}$ by $v\in \rmB^{\rm r}_{s_n- \frac{1}{2}n^{\alpha_d}}$, \eqref{triangle ineq} and \eqref{Hn relation bound} imply
\aln{\label{case2 SH w star2}
\kD(0,x^{k})&{\leq \kD(0,x^{1}) +\sum_{i=1}^{k-1} \kD(y^i,x^{i+1})+\sum_{i=1}^{k-1} \kD(x^i,y^i)} \\
&\leq \ell+\sum_{i=1}^{k-1} \kD(x^i,y^i)\leq \ell+C N^{d+1} n^{1-\beta_d}\leq s_n-\frac{1}{4}n^{\alpha_d}.\notag
}
This, in particular, yields  $x^{k}\in \rmB_{s_n-2}$.

It remains to prove $y^{k}=\SH^{-1}(v^{j_{k}})\in\rmB_{s_n-1}$ for $k<m$. 
Let us assume $y^{k}\notin\rmB_{s_n-1}$. Then, since $\cD(x^{k},y^{k})\ge 2$, $\langle x^{k},y^{k}\rangle\notin \gamma_{0,w_n}$ and thus $y^{k}\in\kC_{\rH}(x^{k})\setminus\{x^{k}\}$. In particular, the boxes containing $x^{k}$ and $y^{k}$ are good, and  the two points  $x^{k}$ and $y^{k}$ must be connected by a path in $\rH^*\setminus (\gamma_{0,w_n}\cup \Gamma\cup \mathbf{V}_n)$. If $y^{k}$ is in the largest open cluster, since $x^{k}\in \rmB_{s_n}$ is in the largest open cluster of its good box, by the same argument as in \eqref{eq:conttaille}, we have  $\kD(x^{k},y^{k})\le C N^{d+1} n^{1-\beta_d}$, which  implies $y^{k}\in\rmB_{s_n-1}$ by \eqref{case2 SH w star2}.
Otherwise, if $y^{k}$ is not in the largest open cluster, since by Claim \ref{claim:openpath}, the path $(y^{k},\SH^{-1}(v^{j_{k}+1}),\cdots,\SH^{-1}(v^{i_{k+1}-1}), x^{k+1})$ is open for $\tau$. Since there are no two distinct large open clusters in $\underline{\Lambda}'_N(y^k)$,  $\|y^{k}-x^{k+1}\|_\infty\leq \frac N 2$. {Moreover, since $x^{k+1}\in \gamma_{0,w_n}$ implies that $y^{k}$ must be in the largest open cluster, $x^{k+1}=\SH^{-1}(v^{i_{k+1}})\notin \gamma_{0,w_n}$. Thus, the edge $\langle v^{i_{k+1}},v^{i_{k+1}+1}\rangle$ does not belong 
 to $\mathbf{E}_1$.} 
Moreover, since  the edge $\langle v^{i_{k+1}},v^{i_{k+1}+1}\rangle$ crosses $\SH(\rH^*)$ and all the edges in $\mathbf{E}_2$ are closed, ${\rm d}_\infty(x^{k+1},\rH^*\cap (\gamma_{0,w_n}\cup\Gamma\cup {\rm V}_n))>2N$. 
Thus, there exists a path from $y^{k}$ to $x^{k+1}$ inside $\rH^*$ not intersecting with $\rH^*\cap (\gamma_{0,w_n}\cup\Gamma\cup {\rm V}_n)$, and  $x^{k+1}\in \kC_{\rH}(y^{k}){=\kC_{\rH}(x^{k})}$, which derives a contradiction.  Therefore, we have \eqref{second trial first goal} with $k<m$.

We can also conclude from what is above that either $y^m \in \rmB_{s_n-1}$ or $y^m\in\kC_{\rH}(x^m)\setminus\{x^m\}$. In both cases, we have that $y^m\in\mathrm{Int}(\Gamma)\cap \rH^*$. When $y^m \in \rmB_{s_n-1}$, by the same argument as in \eqref{triangle ineq} and \eqref{Hn relation bound}, we have %\kD(0,x^1)+\sum_{i=1}^{k-1} \kD(y^i,x^{i+1})\leq\kD^{\rm r}(\tilde{0},\SH(x^1))+\sum_{k=1}^{k-1} \kD^{\rm r}(\SH(y^i),\SH(x^{i+1}))\leq \kD^{\rm r}(\tilde{0},v^{i_{k}})\leq \ell.
\ben{\label{case2 SH w star2 star}
\kD(0,y^{m})\leq \kD^{\rm r}(\tilde{0},\SH(y^{m}))+CN^{d+1}n^{1-\beta_d}.%\leq s_n-\frac{1}{4}n^{\alpha_d}.
}
If $y^m \notin \rmB_{s_n-1}$, then $y^m$ does not belong to the largest cluster of its good box, as we proved  its contrapositive. If, in addition, $v\in\SH(\Z^d)$, by Claim \ref{claim:openpath}, the path $(v^ {j_m},\dots,v^{\ell})$ is included in $\SH(\Z^d)$. Since there are no two distinct large open clusters in $\underline{\Lambda}'_N(y^m)$,  $(\SH^{-1}(v^ {j_m}),\dots,\SH^{-1}(v^{\ell}))$ is open for $\tau$ and  inside $\Lambda_{N/2}(y)$. In particular,%and $\|y^m-\SH^{-1}(v)\|_\infty<N/2$. 
\ben{\label{eq:toreferyi}
\kD(y^m,\SH^{-1}(v))\leq N^d.
}
\indent Let us now furthermore assume that $v\in \SH(\rmB_{s_n-n^{\alpha_d}})\cap {\rm B}^{\rm r}_{s_n-\frac{1}{2}n^{\alpha_d}}(\tilde{0})$. We will prove that
\aln{\label{new trial goal}
\kD^{\rm r}(\tilde{0},v)\geq \kD(0,\SH^{-1}(v))-\frac{1}{2}n^{\alpha_d}.
}%the path $(v^{j_m},\dots,v^\ell)$ is contained in $\SH(\Z^d)$,
%and derive a contradiction.
%By Claim~\ref{claim:openpath} and the conclusion above, since %$v\in\SH(\Z^d)$, $y^m$ is not in the largest cluster of its good %box,  and $(\SH^{-1}(v^{j_m}),\dots,\SH^{-1}(v))$ is open for %$\tau$. %Then since  $\SH^{-1}(v)\in \rmB_{s_n-\frac{1}{2}n^{\alpha_d}}$, $y^m$ is connected by an open path for $\tau$ to %$\rmB_{s_n-\frac{1}{2}n^{\alpha_d}}$. 
If $y^m\notin \rmB_{s_n-1}$, then by \eqref{eq:toreferyi}, we have % {$y^m$ is connected to $\SH^{-1}(v) \in \rmB_{s_n-n^{\alpha_d}}$ by an open %path for $\tau$ of  diameter less than $N/2$, that is of length %at most $N^d$, which implies
 $$\kD(0,y^m)\leq \kD(0,\SH^{-1}(v))+N^d\leq s_n-n^{\alpha_d}+N^d<s_n,$$ 
 % {%Using inequality \eqref{case2 SH w star2}, 
 which contradicts $y^m\notin \rmB_{s_n-1}.$  Therefore, we have $y^m\in \rmB_{s_n-1}.$ 
 %Suppose $v\in \}SH(B_s)$. Then $\kD(0,x^I)\leq \ell+n^{\alpha_d}$. If $y^I\notin B_{s_n-1}$, then the diameter of  $(\SH^{-1}(v^{j_{I}}),\cdots,\SH^{-1}( v^{\ell}))$ is less than $N/2$, which implies $\kD(x^I,\SH^{-1}(v))<< n^{\alpha_d}$ since the box of $y^I$ is good and we can find a sequence of good boxes from $x^I$ to $y^I$. Otherwise, we have the same estimate as above.}\\
 Moreover, by \eqref{case2 SH w star2 star},%\eqref{triangle ineq} and \eqref{Hn relation bound} %(these inequalities are written for $k<m$, but the same argument works for $k=m$ when $y^m\in\rmB_{s_n-1}$ since $\langle x^m,y^m\rangle \in \gamma_{0,w_n}$ or $y^{m}$ is in the largest open cluster of its good box),
\begin{equation}
    \begin{split}\label{ym estimate}
\kD(0,\SH^{-1}(v))&\leq \kD(0,y^m)+\kD(y^m,\SH^{-1}(v))\\
&\leq \kD^{\rm r}(\tilde{0},v)+C N^{d+1} n^{1-\beta_d} \leq \kD^{\rm r}(\tilde{0},v)+\frac{1}{2}n^{1-\alpha_d}.
\end{split}
\end{equation}
 This yields \eqref{new trial goal}.
In particular,  since $\kD^{\rm r}(\tilde{0}, w^*)\leq s_n-\frac{1}{2}n^{\alpha_d}$ by \eqref{cases 2 t star estimate},  applying the previous inequality for $v=w^*$, we have
$$\kD^{\rm r}(\tilde{0}, w^*)\geq \kD^{\rm r}(\tilde{0}, \SH(w^*))-\kD^{\rm r}(w^*, \SH(w^*))\geq s_*-C N^{d+1} n^{1-\beta_d}-2>s_n- 3 n^{\alpha_d}.$$
~\\
\noindent{\bf Proof of \ref{case2:iv}.}
Recall $t_*:=\kD^{\rm r}(\tilde{0}, w^*)\le s_n-\frac{1}{2}n^{\alpha_d}$. Let us prove that
\aln{\label{case2 shape}
\rmB^{\rm r}_{t_*}(\tilde{0})\subset \SH(\rmB_{s_n})\cup \{z\in \Z^d:~\mathrm{d}_\infty(z,\mathrm{Int}(\Gamma)\cap \rH^*)\leq N\}.
} 

From now on, we assume that $v\in \rmB^{\rm r}_{t_*}(\tilde{0})$. We keep the same notation $\gamma^{\rm r}$ to be a geodesic for $\tau^{\rm r}$ between $\widetilde 0$ and $v$, and $i_k,j_k$ are defined in the same way as before.
Thanks to \eqref{case2 SH w star2}, we have $x^m\in\rmB_{s_n-2  }$.  Recall that we have proved $y^m\in\mathrm{Int}(\Gamma)\cap \rH^*$ above \eqref{eq:toreferyi}. We divide the proof into several cases.
Let us first assume that $v\notin \SH(\Z^d)$.   By Claim~\ref{claim:openpath}, the path $(v^{j_m+1},\dots,v)$ is included in $\rH^*$.  On the event $\kA^{\rm r}\cap\kA_n^{[2]}$, since $y^m\in\mathrm{Int}(\Gamma)\cap \rH^*$ and the path $(v^{j_m+1},\dots,v)$ is open for $\tau^{\rm r}$, we have either that the path is included in $\mathbf{E}_1$ or it cannot go at distance less than $2N$ from $\Gamma$. In both cases, $\mathrm{d}_\infty(v,\mathrm{Int}(\Gamma)\cap \rH^*)\leq N$.

Let us next assume that $y^m\in \rmB_{s_n-1}$ and $v\in\SH(\Z^ d)$.  By \eqref{case2 SH w star2 star}, we have %{ym estimate} (noting that we donot use $v\in \SH(\rmB_{s_n-n^{\alpha_d}})$ for the first two inequality there), we have
%By $t_*\leq s_n-\frac{1}{2}n^{\alpha_d}$, we have $\SH(y^m)\in \SH(\rmB_{s_n-1})\cap \rmB_{s_n-\frac{1}{2}n^{\alpha_d}}^{\rm r}(\tilde{0})$
$$\kD(0,y^m)\le\kD^{\rm r}(\tilde{0},\SH(y^m))+CN^{d+1} n^{1-\beta_d}.$$
By Claim~\ref{claim:openpath}, the path $(v^{j_m},\dots,v)$ is in $\SH(\Z^d)$ and the path $(\SH^{-1}(v^{j_m}),\dots,\SH^{-1}(v))$ is open for $\tau$. Hence, by $\SH(y^m)=v^{j_m}\in \gamma^{\rm r}$ and $t_*\le s_n-\frac{1}{2}n^{\alpha_d}$, those imply
\al{
\kD(0,\SH^{-1}(v))&\leq \kD(0,y^m)+\kD(y^m,\SH^{-1}(v))\\
&\leq \kD^{\rm r}(\tilde{0},\SH(y^m))+CN^{d+1} n^{1-\beta_d}+\kD^{\rm r}(\SH(y^m),v)\\
&= \kD^{\rm r}(\tilde{0},v)+CN^{d+1} n^{1-\beta_d}\leq t_*+CN^{d+1} n^{1-\beta_d}\le s_n,
}
which yields $v\in\SH(\rmB_{s_n})$.

Finally, assume $y^m\notin\rmB_{s_n-1}$ and $v\in\SH(\Z^ d)$. Thanks to \eqref{eq:toreferyi} and a claim thereover, we have $y^m\in\mathrm{Int}(\Gamma)\cap \rH^*$ and 
\[\mathrm{d}_\infty(v,\mathrm{Int}(\Gamma)\cap\rH^*)\le N.\]
Thus,  we have \eqref{case2 shape}. Condition-\ref{case2:iv} follows from \eqref{eq:isoperimetry}, $\rmB_{s_n}\subset \mathrm{Int}(\Gamma)$, $|\Gamma|\leq K n$, and \eqref{case2 shape}.
\newline

\noindent {\bf Proof of \ref{case2:ii}.}
Let us prove that $\rmB^{\rm r}_{t_*}(\tilde{0})\cap \rL_3(w^*)=\{w^*\}$ on the event $\kA^{\rm r}\cap \kA^{[2]}_n$. To this end, let us take $v\in \rmB^{\rm r}_{t_*}(\tilde{0})\cap \rL_3(w^*)$ and a geodesic $\gamma^{\rm r}=(v^i)_{i=1}^\ell$ from $\tilde{0}$ to $v$ for $\tau^{\rm r}$.
Let \[i_*:=\min\{i\leq \ell:~v^i\in \rL_3(w^*)\}.\] If $v^{i_*}=w^*$, then $i_*=\ell$ and $v=w^*$.  Indeed, since any edge $e\in \mathbf{E}_2$ is closed and  the path  $\gamma^{\rm r}$ is self-avoiding, it must end at $w^*$. Otherwise, if $v^{i_*}\neq w^*$, then  $\mathrm{d}_\infty(v^{i_*},\mathrm{Int}(\Gamma)\cap \rH^*)>N$ since any edge $e\in \mathbf{E}_3$ is closed. However, since $\SH(\Z^d)\cap \rL_3(w^*)=\emptyset$, $v^{i_*}\notin \SH(\rmB_{s_n})$. By \eqref{case2 shape}, this implies $\mathrm{d}_\infty(v^{i_*},\mathrm{Int}(\Gamma)\cap \rH^*)\leq N$, which leads to a contradiction.\newline
 
 \noindent {\bf Proof of \ref{case2:iii}.}
 Finally, we prove $\rN_k( \rH_1^*\cap \rmB ^{\rm r}_{t_*}(\tilde{0}))\leq n^{\alpha_d}$ for any $k\neq 1$. We fix $k\neq 1$. Suppose $\rL_k(z')\cap \rH_1^*\cap \rmB^{\rm r}_{t_*}(\tilde{0})\neq \emptyset$ with $z'\in \Z^d$ such that $z_k'=0$. By \eqref{case2 shape} and $ \rH_1^*\cap \SH(\Z^d)=\emptyset$, 
 \aln{\label{first implication}
  \rL_k(z')\cap \{z\in \Z^d:~\mathrm{d}_\infty(z,\mathrm{Int}(\Gamma)\cap  \rH_1^*)\leq N\}\neq \emptyset.
 }
Combined with \eqref{3 estimate}, this implies that for $n$ large enough depending on $s_0,p,d$,
 \al{
 \rN_k( \rH_1^*\cap \rmB ^{\rm r}_{t_*}(\tilde{0}))
 \leq \rN_k(\{z\in \Z^d:~\mathrm{d}_\infty(z,\mathrm{Int}(\Gamma)\cap  \rH_1^*)\leq N\})\le n^{\alpha_d}.
 }
% This concludes the proof.
 \end{proof}

\subsection{Proof of Lemma \ref{lem:aux}}\label{subsection:proofmainlemma}
We will need the following lemma that gives a lower bound for the number of almost disjoint paths connecting two given subsets of two separate hyperplanes. We postpone its proof until the appendix.
\begin{lem}\label{lem:disjpathscont}Let $d\ge 3$. There exists $C_d>0$ such that the following holds.
Let $S^1,S^2\subset \Z^d$ that satisfy one of the following.
\begin{itemize}
    \item[-] There exist $i\in[d]$, $K\ge \ell\ge 1$ such that $S^1\subset \rH_i(0)$, $S^2\subset \rH_i(\ell\mathbf{e}_i)$ and 
\aln{
\max_{x\in S_1,\,y\in S_2} \|x-y\|_{\infty}= K.
}
    \item[-] There exist  $i\ne j\in[d]$, $K\ge 1$ such that $S^ 1\subset \rH_i(0)$,  $S^ 2\subset \rH_j(0)$ and $S^1\cup S^2\subset [-K,K]^d$.
\end{itemize}
Then, we can find $\Z^d$-paths $(\fp_i)_{1\le i\le m}$ from $S^ 1$ to $S^ 2$ with $m:=\min(|S^ 1|,|S^ 2|)/2$ such that the length of each path is less than $2d K$ and  there exists a constant $\chi_d\in\N$ depending only on $d$ such that
\[\forall x\in\Z^ d\quad|\{i\in[m]:x\in \fp_i\}|\le \chi_d {\left(\frac K \ell\right)^{d-1}}.\]
\end{lem}
\begin{proof}[Proof of Lemma \ref{lem:aux}]
Let $s_0>0$.  Let $K$ be as in Proposition \ref{prop:sizeboundary} with $2s_0$ in place of $s_0$. Let $n_0$ be so that  Proposition \ref{prop:sizeboundary} and Lemma \ref{lem:drilling} hold. Let $s,s'\in[0,s_0],\,x,x'\in[-s_0,s_0]^d$ and $n\ge m\ge n_0$.
Denote
\[\kA^{{\rm free},K}_{s,x}(n):=\left\{\begin{array}{c}\exists i\neq j\in [d],\quad \exists s_n\geq  s n-3n^{\alpha_d},\quad \exists w_n\in \Lambda_{4n^{\alpha_d}}(nx);\\  \,B_{s_n}\setminus \rmB_{s_n-1}=\{w_n\},\,\rL_i(w_n)\cap \rmB_{s_n}=\{w_n\},\,\forall k\neq j\quad \rN_k(\rH_j(w_n)\cap \rmB _{s_n})\le n^{\alpha_d},\\
\exists \Gamma\subset \Z^d;\quad \rmB_{s_n}\subset\mathrm{Int}(\Gamma),|\Gamma|\le K n
\end{array}\right\}.\]
 By  Proposition \ref{prop:sizeboundary}, we have 
$$\P(\kA^{\rm free}_{s,x}(n)\setminus \kA^{{\rm free},K}_{s,x}(n))\le \P\left(\bigcup_{t\in[0,2s_0],y\in[-2s_0,2s_0]^d}(\kA_{t,y}(n)\setminus \kA_{t,y}^ K(n))\right)\le e^{-n}\P(\kA_{s,x}(n)).$$
It follows that by Lemma \ref{lem:drilling} and Proposition \ref{prop:sizeboundary},
\begin{equation*}
    \frac{1}{2}\P(\kA_{s,x}(n))\le \P( \kA^{K}_{s,x}(n))\le e^{n^{\alpha_d}}\P(\kA^{{\rm free}}_{s,x}(n))\le e^{n^{\alpha_d}}(\P(\kA^{{\rm free},K}_{s,x}(n))+  e^{- n}\P(\kA_{s,x}(n))).
\end{equation*}
\iffalse
We have by Proposition \ref{prop:sizeboundary},\sncomment{The first ineq is not true since we change sidelength}
\begin{equation*}
    \P(\kA^{\rm free}_{s,x}(n)\setminus \kA^{{\rm free},K}_{s,x}(n))\le \P(\kA_{s,x}(n)\setminus \kA^{K}_{s,x}(n))\le e^{- n}\P(\kA_{s,x}(n)).
\end{equation*}
It follows that by Lemma \ref{lem:drilling} and Proposition \ref{prop:sizeboundary}
\begin{equation*}
    \frac{1}{2}\P(\kA_{s,x}(n))\le \P( \kA^{K}_{s,x}(n))\le e^{n^{\alpha_d}}(\P(\kA^{{\rm free},K}_{s,x}(n))+  e^{- n}\P(\kA_{s,x}(n))).
\end{equation*}
\fi
Finally, we have for $n$ large enough depending on $s_0$,
\begin{equation}\label{eq:combiningproplem}
    \P(\kA_{s,x}(n))\le 4e^{n^{\alpha_d}}\P(\kA^{{\rm free},K}_{s,x}(n)).
\end{equation}
Define for $t\ge 1$, $i\ne j \in[d]$, $w\in\Z^d$,
\al{
\kE^ {\rm e}_{t,w,i,j}&:=\left\{\begin{array}{c} \rmB_{t}\setminus \rmB_{t-1}=\{w\}, \rL_{i}(w)\cap \rmB_{t}=\{w\},\,\exists \Gamma\subset \Z^d;\\ \rmB_t\subset\mathrm{Int}(\Gamma),|\Gamma|\le K n,\, \forall k\neq j\,\rN_k(\rH_j(w)\cap \rmB _{s_n})\le n^{\alpha_d}\end{array}\right\},\\
\kE^ {\rm b}_{t,w,i,j}&:=\left\{\begin{array}{c} \rmB_{t}\setminus \rmB_{t-1}=\{w\}, \rL_{i}(0)\cap \rmB_{t}=\{0\},\,\exists \Gamma\subset \Z^d;\\ \rmB_t\subset\mathrm{Int}(\Gamma),|\Gamma|\le K n,\,\forall k\neq j\,\rN_k(\rH_j(w)\cap \rmB _{s_n})\le n^{\alpha_d}\end{array}\right\}.
}
Note that on the event  $\kA_{s,x}^{{\rm free},K}(n)$, since there exists $\Gamma\subset \Z^ d$ such that $|\Gamma|\le K n$ and $\rmB_t\subset \mathrm{Int}(\Gamma)$, it follows that 
$t\le |\mathrm{Int}(\Gamma)|\le (Kn)^ {d}.$
   By pigeon-hole principle, using \eqref{eq:combiningproplem}, there exist $s_n\in[s n -3n^{\alpha_d},(Kn)^ d]$, $w_n\in\Lambda_{4n^{\alpha_d}}(nx)$, $i_n\ne j_n\in[d]$ such that
\begin{equation}\label{eq:drillpige}
    \P(\kE^{\rm e}_{s_n,w_n,i_n,j_n})\ge \frac{1}{(Kn)^{2d}}e^{-n^{\alpha_d}}\P(\kA_{s,x}(n)).
\end{equation}
 %The drilling Lemma can be used to drill a line at $0$, to do so 
 In order to make $\kE^{\rm b}_{s_m,w_m,i_m,j_m}$ occur from $\kA_{s',x'}(m)$, we first apply Lemma \ref{lem:createcutpoin} to create a cut point at $0$, i.e. $\rmB_{s_n}(w)\setminus \rmB_{s_n-1}(w)=\{0\}$, and then we apply Lemma \ref{lem:drilling} to create a free line at $0$. Next, we use Lemma \ref{lem:createcutpoin} again to make $w$ a cut-point.  {Finally, by the same argument as in \eqref{eq:combiningproplem}, we can find a desired $\Gamma$ as in $\kE^{\rm b}_{s_m,w_m,i_m,j_m}$.}  Hence, there exist $s_m\in[s' m -3m^{\alpha_d},(Kn)^ d]$, $w_m\in\Lambda_{4m^{\alpha_d}}(x'm)$, and $i_m\ne j_m\in[d]$ such that
\begin{equation}\label{eq:drillpige2}
    \P(\kE^{\rm b}_{s_m,w_m,i_m,j_m})\ge \frac{1}{(Kn)^{2d}}e^{-m^{\alpha_d}}\P(\kA_{s',x'}(m)).
\end{equation}
Denote by $\sC^{\rm e}_{t,w,i,j}$ (respectively $\sC^{\rm b}_{t,w,i,j}$) the set of admissible connected graph $C=\rm B_{t}$ for the event $\kE^{\rm e}_{t,w,i,j}$ (respectively $\kE^{\rm b}_{t,w,i,j}$).
We have
\begin{equation*}
    \begin{split}
        \P(\kE^{\rm e}_{s_n,w_n,i_n,j_n})\P(\kE^{\rm b}_{s_m,w_m,i_m,j_m})=\sum_{{\rm C}_1\in\sC^{\rm e}_{s_n,w_n,i_n,j_n} }\sum_{{\rm C}_2\in\sC^{\rm b}_{s_m,w_m,i_m,j_m}}\P(\rmB_{s_n}={\rm C}_1)\P(\rmB_{s_m}={\rm C}_2).
    \end{split}
\end{equation*}
Let ${\rm C}_1\in\sC^{\rm e}_{s_n,w_n,i_n,j_n}$ and ${\rm C}_2\in\sC^{\rm b}_{s_m,w_m,i_m,j_m}$. Let us first prove that we can do a small translation of $w_n+{\rm C}_2$ such that its intersection with ${\rm C}_1$ is empty. We take {$\Gamma_1 ,\Gamma_2\subset\Z^d$ such that $\max\{|\Gamma_1|,|\Gamma_2|\}\le K n$ and ${\rm C}_1\subset\mathrm{Int}(\Gamma_1)$, ${\rm C}_2\subset\mathrm{Int}(\Gamma_2)$. By isoperimetry of $\Z^d$ (see \eqref{eq:isoperimetry}),
\ben{\label{isoperimetry Gamma}
\max\{|\mathrm{Int}(\Gamma_1)|,|\mathrm{Int}(\Gamma_2)|\}\le \kappa_d(Kn)^{\frac{d}{d-1}}\le\kappa_d K^{2} n^{\frac{d}{d-1}}.
}}
\iffalse
Note that if there exist $w_1,w_2\in\Z^d$ such that $\mathrm{Int}(\Gamma_1)\subset (w_1+\mathrm{Int}(\Gamma_2))$ and $\mathrm{Int}(\Gamma_2)\subset (w_2+\mathrm{Int}(\Gamma_1))$ then $\mathrm{Int}(\Gamma_1)=w_1+\mathrm{Int}(\Gamma_2)$. Otherwise, {\SN up to changing the roles of ${\rm C}_1$ and ${\rm C}_2$},\sncomment{I'm not sure it is trivial since we assume $n\geq m$} we can assume without loss of generality that $\mathrm{Int}(\Gamma_1)$ is not contained in   $w+\mathrm{Int}(\Gamma_2)$ for any $w\in\Z^d$. Hence, if $\mathrm{Int}(\Gamma_1)\subset (z+\mathrm{Int}(\Gamma_2))$, then  $\mathrm{Int}(\Gamma_1)= z+\mathrm{Int}(\Gamma_2)$.
\fi
 {Since $\sum_{z\in\Z^d}\mathbf{1}_{x\in (z+A)}=|A|$ for any $x\in \Z^d$ and $A\subset \Z^d$,} we have
\begin{equation*}
\begin{split}
&\left|\left\{z\in \Lambda_{9n^{\alpha_d}}(w_n)\cap \Z^d:\,(\mathrm{Int}(\Gamma_1)\cup \rL_{i_n}^{Kn}(w_n))\cap (z+(\Gamma_2\cup \rL_{i_m}^{Kn}(0)))\neq \emptyset\right\}\right|\\
   &\leq \sum_{z\in \Lambda_{9n^{\alpha_d}}(w_n)\cap \Z^d}|(\mathrm{Int}(\Gamma_1)\cup \rL_{i_n}^{Kn}(w_n))\cap (z+(\Gamma_2\cup \rL_{i_m}^{Kn}(0)))|\\
    &=\sum_{x\in \mathrm{Int}(\Gamma_1)\cup \rL_{i_n}^{Kn}(w_n)}\sum_{z\in \Lambda_{9n^{\alpha_d}}(w_n)\cap \Z^d}\mathbf{1}_{x\in (z+\Gamma_2\cup\rL_{i_m}^{Kn}(0) )}\\
    &\le (|\Gamma_2|+4Kn)(|\mathrm{Int}(\Gamma_1)|+4Kn)\le   c_d\,n^{1+\frac{d}{d-1}}.
    \end{split}
\end{equation*}
for some constant $c_d>0$ depending only on $d,K$.
Besides,  {since 
$\|nx-w_n\|_\infty\le 4n^{\alpha_d}$ and $\|mx'-w_m\|_\infty \le4 m^{\alpha_d}\leq 4n^{\alpha_d}$,   $z\notin\Lambda_{9n^{\alpha_d}}(w_n)$ (or equivalently $\|z-w_n\|_\infty>9n^{\alpha_d}$) implies $$\|z+w_m-nx-mx'\|_\infty\geq \|z-w_n\|_\infty-\|nx-w_n\|_\infty-\|mx'-w_m\|_\infty  >  n^{\alpha_d}.$$ 
% $(nx+mx'-w_m+\Lambda_{n^ {\alpha_d}})\subset (w_n+\Lambda_{3n^{\alpha_d}})$. %$\|w_n+w_m-nx-mx'\|_\infty\le 2n^{\alpha_d}$. %, and thus $(w_n+w_m-nx-mx'+\Lambda_{n^ {\alpha_d}})\subset \Lambda_{3 {n^{\alpha_d}}} $.}
Therefore, we reach 
\al{
&\left|\{z\in\Lambda_{9n^{\alpha_d}}(w_n)\cap \Z^d:~\|z+w_m-nx-mx'\|_\infty\le n^{\alpha_d}\}\right|\\
&=\left|\{z\in\Z^d:~\|z+w_m-nx-mx'\|_\infty\le n^{\alpha_d}\}\right|
=|\Lambda_{n^ {\alpha_d}}(nx+mx'-w_m)\cap \Z^ d|\ge 2n^ { {d}\,\alpha_d}.
}
%|(w_n+w_m-nx-mx'+\Lambda_{3n^{\alpha_d}})\cap \Lambda_{n^ {\alpha_d}}\cap \Z^d|the number of $u\in\Z^d$ such that $\|u\|_\infty\le 3n^{\alpha_d}$ and $\|w_n+w_m-nx-mx'-u\|_\infty\le n^{\alpha_d}$ is $|(w_n+w_m-nx-mx'+\Lambda_{n^ {\alpha_d}})\cap \Z^ d|\ge n^ {\alpha_d}$. 
As a consequence, we have
\al{
&\left|\left\{z\in\Lambda_{ {9n}^{\alpha_d}}(w_n)\cap \Z^d:\begin{array}{c} (\mathrm{Int}(\Gamma_1)\cup \rL_{i_n}^{Kn}(w_n))\cap (z+(\Gamma_2\cup \rL_{i_m}^{Kn}(0)))=\emptyset,\\ \|z+w_m-nx-mx'\|_\infty\le n^{\alpha_d}\end{array}\right\}\right|\\
&\ge 2n^{d\,\alpha_d}-c_d n^{1+\frac{d}{d-1}}\geq  {n^{d\,\alpha_d}}.
}

}
  {Thus, by pigeon-hole argument and $\max\{|\Gamma_1|,|\Gamma_2|\}\le K n$}, there exists  $z\in \Lambda_{9n^{\alpha_d}}(w_n)$;
\begin{itemize}
    \item[-] $ (\mathrm{Int}(\Gamma_1)\cup \rL_{i_n}^{Kn}(w_n))\cap (z+(\Gamma_2\cup \rL_{i_m}^{Kn}(0)))=\emptyset$;
      \item[-]  $\|(z+w_m)-(nx+mx')\|_\infty\le n^{\alpha_d}$;
    \item[-] $|\Gamma_1\cap \rH_{j_m}(z)|\le Kn^{1-\alpha_d} \le n^{\alpha_d}$;
    \item[-] $|(z+\Gamma_2)\cap \rH_{j_n}(w_n)|\le Kn^{1-\alpha_d} \le n^{\alpha_d}$;
    \item[-] if $j_m=j_n$,   $|(z-w_n)\cdot \mathbf{e}_{j_n}|>n^{\alpha_d}/4$.
\end{itemize}
 %We claim that $\mathrm{Int}(\Gamma_1)\cap (z+\Gamma_2)=\emptyset$ implies that $\mathrm{Int}(\Gamma_1)\cap (z+\mathrm{Int}(\Gamma_2))=\emptyset$ or $\mathrm{Int}(\Gamma_1)\subset (z+\mathrm{Int}(\Gamma_2))$. Indeed, if $\mathrm{Int}(\Gamma_1)\cap (z+\mathrm{Int}(\Gamma_2))\ne \emptyset$, let us prove that $\mathrm{Int}(\Gamma_1)\subset (z+\mathrm{Int}(\Gamma_2)) $. If there exists $x\in \mathrm{Int}(\Gamma_1)\setminus (z+\mathrm{Int}(\Gamma_2))$ and $y\in \mathrm{Int}(\Gamma_1)\cap (z+\mathrm{Int}(\Gamma_2))$, then there exists a path  between $x$ and $y$ in $\mathrm{Int}(\Gamma_1)$. Denoting by $w$ the first moment when the path exits $z+\mathrm{Int}(\Gamma_2)$, we have $w\in (z+\Gamma_2)\cap \mathrm{Int}(\Gamma_1) $, which is a contradiction. Note that if $\mathrm{Int}(\Gamma_1)\subset (z+\mathrm{Int}(\Gamma_2))$, it corresponds to the case where $\mathrm{Int}(\Gamma_1)= z+\mathrm{Int}(\Gamma_2)$, such a $z$ is unique. Therefore, we can just change our choice of $z$ to get that ${\rm C}_1\cap(z+{\rm C}_2)=\emptyset$.
Let us fix such a $z$.  {By the first condition together with $|\Gamma_2|\leq K n$,}   $\mathrm{Int}(\Gamma_1)\cap (z+\mathrm{Int}(\Gamma_2))=\emptyset$. We will establish a  path $w_n$ with $z$  avoiding ${\rm C}_1$ and $z+{\rm C}_2$.
Let $k_n\in [d]\setminus\{i_n,j_n\}$. 
Define 
\[E_n:=\{\ell\in\{5n^{\alpha_d},\dots,10n^ {\alpha_d}\} : \rL_{k_n}(w_n+\ell {\mathbf e}_{i_n})\cap ({\rm C}_1\cup (z+\Gamma_2))=\emptyset\}.\]
Since the lines are disjoint, $|\rH_{j_n}(w_n)\cap  (z+\Gamma_2)|\le n^ {\alpha_d}$, and $\rN_{k_n}(\rH_{j_n}(w_n)\cap {\rm C}_1)\le n^{\alpha_d}$, we have  $|E_n|\ge n^ {\alpha_d}$.
It follows that  all these lines do not intersect with either ${\rm C}_1$ or $z+{\rm C}_2$.
Similarly, let $k_m\in [d]\setminus\{i_m,j_m\}$ and
\[ E_m:=\{\ell\in\{5n^{\alpha_d},\dots,10n^ {\alpha_d}\} : \rL_{k_m}(z+\ell {\mathbf e}_{i_m})\cap ( {\rm C}_1\cap (z+\Gamma_2))=\emptyset\}.\]
It holds that $|E_m|\ge n^{\alpha_d}$. Denote
\al{
\underline{E}_n:=\{x\in \rL_{k_n}^{n^{\alpha_d}}(w_n+\ell {\mathbf e}_{i_n}): \ell \in E_n\},\quad\text{  and }\quad
\underline {E}_m:=\{x\in \rL_{k_m}^{n^{\alpha_d}}(z+\ell {\mathbf e}_{i_m}): \ell \in E_m\}.
}
It is easy to check that the sets $\underline{E}_m$, $\underline{E}_n$ are contained in 
$\Lambda_{20n^{\alpha_d}}(w_n)$. In the case $j_n=j_m$,  {recall that} we   {chose} $z$ so that the hyperplanes $\rH_{j_n}(w_n)$ and $\rH_{j_n}(z)$ are separated  at distance at least $n^{\alpha_d}/4$. 

 Thanks to Lemma \ref{lem:disjpathscont} with $K=20n^{\alpha_d}=80\ell$,  there exist $\Z^d$-paths $(\fp_i)_{1\le i\le n^{2\alpha_d}}$ from  $\underline{E}_n$ to $\underline{E}_m$ such that  each path has a length less than $100dn ^{\alpha_d}$ and
\[\forall x\in\Z^ d\quad|\{i\in[n^{2\alpha_d} ]:x\in \fp_i\}|\le {80^{d-1} \chi_d}=:\chi'_d.\]
Let us assume all these paths intersect ${\rm C}_1\cup (z+{\rm C}_2)$. Since each vertex cannot be contained in more than ${\chi'_d}$ paths, it follows that
$|{\rm C}_1\cup (z+{\rm C}_2)|\ge n^ {2\alpha_d}/ {\chi'_d}$, which contradicts \eqref{isoperimetry Gamma}. It follows that, there exists at least one path $\fp_i$ that does not intersect ${\rm C}_1\cup (z+{\rm C}_2)$ between some $x_1\in\rL_{k_n}^{n^{\alpha_d}}(w_n+\ell_1 {\mathbf e}_{i_n})$ and  $x_2\in  \rL_{k_m}^{n^{\alpha_d}}(z+\ell_2 {\mathbf e}_{i_m})$ with $\min\{|\ell_1|,|\ell_2|\}\geq 5n^{\alpha_d}$. Consider the path:
 \[\fp:=L(w_n,w_n+\ell_1 {\mathbf e}_{i_n})\oplus L(w_n+\ell_1 {\mathbf e}_{i_n},x_1)\oplus \fp_i\oplus L(x_2,{z}+\ell_2 {\mathbf e}_{i_m})\oplus L({z}+\ell_2 {\mathbf e}_{i_m},z).\]
 %  that goes from $w_n$ to $w_n+\ell_1 {\mathbf e}_{i_n}$ on the line $\rL_{i_n}(w_n)$, then from $w_n$ to $w_1$ on the line $\rL_{k_n}^{n^{\alpha_d}}(w_n+\ell_1 {\mathbf e}_{i_n})$, then follows $\fp_0$ from $x_1$ to $x_2$, from $x_2$ to $w_m+\ell_2 {\mathbf e}_{i_m}$ it follows the line $\rL_{k_m}^{n^{\alpha_d}}(z+\ell_2 {\mathbf e}_{i_m})$ and finally from $w_m+\ell_2 {\mathbf e}_{i_m}$ to $z$ it follows the line $\rL_{i_m}(z)$.
It is easy to check that $\fp\setminus \{w_n,z\}\subset\Z^d \setminus ({\rm C}_1\cup (z+{\rm C}_2))$ and $10n^{\alpha_d}\leq |\fp|\le 200d n^{\alpha_d}$.  Moreover, by construction, $\rmB_{s_n+s_m+|\fp|}={\rm C}_1\cup \fp\cup (z+{\rm C}_2)$ implies  $\rmB_{s_n+s_m+|\fp|}\setminus \rmB_{s_n+s_m+|\fp|-1}=\{z+w_m\}$.

Let us now build a resampled configuration. We denote by $\mathbf{E}({\rm C}_1)$ the set of edges whose values are involved with the event $ \{\rmB_{s_n}={\rm C}_1\}$. We define $\mathbf{E}({\rm C}_2)$ similarly. Given a set $E\subset {\rm E}^d$ and $z\in \Z^d,$ we define $E+z:=\{\langle x+z,y+z\rangle:~\langle x,y\rangle\in E\}$. %Let us take $\tau^1\in \{1,+\infty\}^{\mathbf{E}({\rm C}_1)}$ and $\tau^2\in \{1,+\infty\}^{\mathbf{E}({\rm C}_2)+z}$ independently.
Let $\tau^*$ be a configuration independent of $\tau$. Consider the following resampled configuration: 
$$\forall e\in\E^d\qquad\tau^{\rm r}_e:=
\begin{cases}
\tau_e&\text{ if $e\in \mathbf{E}({\rm C}_1)\cup (\mathbf{E}({\rm C}_2)+z)$},\\
%\tau_e^2&\text{ if $e\in \mathbf{E}({\rm C}_2)+z$},\\
\tau_e^*&\text{ otherwise.}%if $e\cap \fp\ne \emptyset$ and %$e\notin \mathbf{E}({\rm C}_1)\cup(\mathbf{E}({\rm C}_2)+z)$ }.
\end{cases}$$
We denote by $\rmB^{\rm r}$ the ball for the configuration $\tau ^{\rm r}$. %from Lemma~\ref{lem:resample} 
It follows that there exists $c'>0$ depending on $p$ and $d$ such that
\begin{equation*}
    \begin{split}
        e^ {-c' {n^{\alpha_d}}}\P&(\rmB_{s_n}={\rm C}_1)\P(\rmB_{s_m}={\rm C}_2)\\
        &\le\P(\rmB_{s_n}={\rm C}_1, \rmB_{s_m}(z)={\rm C}_2,\forall e\in\fp\quad \tau_e^*=1 , \forall |e\cap\fp|=1\quad \tau_e^*=\infty)\\
        &\le \P(\rmB^{\rm r}_{s_n+s_m+|\fp|}={\rm C}_1\cup \fp\cup (z+{\rm C}_2))\\
        &=\P(\rmB_{s_n+s_m+|\fp|}={\rm C}_1\cup \fp\cup (z+{\rm C}_2))=:\P(\rmB_{s_n+s_m+|\fp|}=\pi({\rm C}_1,{\rm C}_2)),
    \end{split}
\end{equation*}%\sout{, \rmB^{\rm r}_{s_n+s_m+|\fp|}\setminus \rmB^{\rm r}_{s_n+s_m+|\fp|-1}=\{z+w_m\}}
where we denote by $\pi$ the map associating  $({\rm C}_1,{\rm C}_2)\in\sC^{\rm e}_{s_n,w_n,i_n,j_n}\times \sC^{\rm b}_{s_m,w_m,i_m,j_m}$ to the graph ${\rm C}_1\cup \fp\cup (z+{\rm C}_2)$.
 One can check that $\pi$ is injective.
Indeed, if $\pi({\rm C}_1,{\rm C}_2)$ is given, then  ${\rm C}_1$ is recovered by considering all the points at distance  at most $s_n$ from $0$. To recover ${\rm C}_2$, we find $z+w_m$ as the furthest point from $0$ in $\pi({\rm C}_1,{\rm C}_2)$, this enables us to recover $z$ and then ${\rm C}_2$ by considering all the points that are at distance  at most $s_m $ from  $z$  not connected to ${\rm C}_1$ by a path in $\pi({\rm C}_1,{\rm C}_2)\setminus\{z\}$.
 Note that thanks to our choice of $z$, we have 
\[\|(z+w_m)-(nx+mx')\|_\infty\le n^{\alpha_d}\le (n+m)^{\alpha_d}.\] Finally, since $$s_n+s_m+|\fp|\geq sn+s'm-6n^{\alpha_d}+10 n^{\alpha_d}\geq sn+s'm,$$ we have
\begin{equation}\label{eq:almostfinal1}
    \begin{split}
        &\P(\kE^{\rm e}_{s_n,w_n,i_n,j_n})\P(\kE^{\rm b}_{s_m,w_m,i_m,j_m})\\
        &=\sum_{{\rm C}_1\in\sC^{\rm e}_{s_n,w_n,i_n,j_n} }\sum_{{\rm C}_2\in\sC^{\rm b}_{s_m,w_m,i_m,j_m}}\P(\rmB_{s_n}={\rm C}_1)\P(\rmB_{s_m}={\rm C}_2)\\
        &\leq\sum_{{\rm C}_1\in\sC^{\rm e}_{s_n,w_n,i_n,j_n} }\sum_{{\rm C}_2\in\sC^{\rm b}_{s_m,w_m,i_m,j_m}}e^{c' n^{\alpha_d}}\P\left(\rmB_{s_n+s_m+|\fp|}=\pi({\rm C}_1,{\rm C}_2)\right)\\
        &\le e^{c'n^{\alpha_d}}\P(\kA_{\frac{n}{n+m}s+\frac{m}{n+m}s',\frac{n}{n+m}x+\frac{m}{n+m}x'}(n+m)).
    \end{split}
\end{equation}%\begin{array}{c}\rmB_{s_n+s_m+|\fp|}={\rm C}_1\cup \fp\cup (z+{\rm C}_2),\\ \rmB_{s_n+s_m+|\fp|}\setminus %\rmB_{s_n+s_m+|\fp|-1}=\{z+w_m\}\end{array}
Moreover, when $n=m$, 
$$       \P(\kE^{\rm e}_{s_n,w_n,i_n,j_n})\P(\kE^{\rm b}_{s_m,w_m,i_m,j_m})\le e^{c'n^{\alpha_d}}\P(\kA_{s+s',x+x'}(n)).$$
The result follows by combining the last three inequalities with \eqref{eq:drillpige} and \eqref{eq:drillpige2}.
\end{proof}

\section{On properties of $J_x$}\label{section: Jx}
  {
 {The following lemma shows that we can constrain a range of the infimum in \eqref{def J} to a compact set.}
 \begin{lem}\label{lem:J prop} Let $x\in \R^d\setminus\{0\}$. We define $R=R_{x,\xi}:=J_x(\xi)/I(1,0)\geq 0$. It holds:
 \ben{\label{restricted J}
 J_x(\xi)=\inf_{\substack{y\in[-R,R]^d,s\in [0,R]:\\s+\mu(y-x)\ge (1+\xi)\mu(x)}}I(s,y).
 }
\end{lem}
\begin{proof}
Recall that $I(s,y)= I(s\lor \|y\|_1,y)$. If $s\land \|y\|_{\infty}\geq R$, by Theorem~\ref{prop:existencelimit}, then  $$I(s,y)\geq I(R,y)= \frac{1}{2}(I(R,y)+I(R,-y))\geq \frac{1}{2}I(2R,0)=R\, I(1,0)= J_x(\xi).$$
 Therefore, the further restriction in the infimum does not change the value.
\end{proof}
}
\begin{proof}[Proof of Proposition \ref{prop:Jcont}]
By definition, it is trivial to see that $J_x$ is  non-decreasing. { Moreover, by Lemma~\ref{lem:J prop}, since $I$ and $\mu$ are continuous (see a remark below \eqref{time constant} and Theorem~\ref{prop:existencelimit}), there is a minimizer in \eqref{restricted J}, say $(s_*,y_*)$, such that $(s_*,y_*)\neq (0,0)$ because of the condition $s+\mu(y-x)\ge (1+\xi)\mu(x)$. By  $I(s_*,y_*)= I(s_*\lor \|y_*\|_1,y_*)>0$, we have $J_x(\xi)>0$ for any $\xi>0$.}

Next, we will prove the continuity. Let $\xi,\ep>0$. We take $y\in\R^d,s\ge 0$ to be such that \[s+\mu(y-x)\ge (1+\xi)\mu(x),\quad\text{ and }
\quad J_x(\xi)\ge I(s,y)-\ep.\]
By continuity of $I$ (Theorem~\ref{prop:existencelimit}), there exists $\delta>0$ small enough such that
\[I(s+\delta\mu(x),y)\le I(s,y)+\ep.\]
We have
\[s+\delta\mu(x)+\mu(y-x)\ge (1+\xi+\delta)\mu(x).\]
Hence, it yields
\[J_x(\xi)\leq J_x(\xi+\delta)\le I(s+\delta\mu(x),y)\le I(s,y)+\ep\le J_x(\xi)+2\ep.\]
Therefore, $J_x(\xi)=\lim_{a\to\xi+}J_x(a)$.% and we have the right-continuous. 

Finally, we consider the left-limit. Given $n\in\N$, by Lemma~\ref{lem:J prop}, since the function $\xi\to R_{x,\xi}$ is non-decreasing, for any $\ep>0$, there exists $(s_n,y_n)\in [0,R]\times [R,R]^d$ such that  
$$s_n+\mu(y_n-x)\geq (1+\xi-(1/n))\mu(x),\quad J_x(\xi-(1/n))\geq I(s_n,y_n)-\ep.$$
By the Bolzano–Weierstrass theorem, we can find a subsequence $(n_k)_{k\in\N}$ and $(s_*,y_*)\in [0,\infty)\times\R^d$ such that $s_{n_k}\to s_*$ and $y_{n_k}\to y_*$. By continuity of the time constant $\mu$ and $I$, since $J$ is non-decreasing, $(s_*,y_*)$ satisfies $$s_*+\mu(y_*-x)\geq (1+\xi)\mu(x),\quad\lim_{k\to\infty}J_x(\xi-(1/n_k))\geq I(s_*,y_*)-\ep.$$
Therefore, since $J_x$ is non-decreasing, we have
$$\lim_{a\to \xi-}J_x(a)=\lim_{k\to\infty}J_x(\xi-(1/n_k))\ge I(s_*,y_*)-\ep \geq J_x(\xi)-\ep\geq \lim_{a\to \xi-}J_x(a)-\ep .$$
We conclude the proof by letting $\ep$ go to $0$.
%Hence, we have the left-continuous.
\end{proof}
{ We introduce $\xi_0$ appearing in  Theorem~\ref{thm:main}.}
\begin{lem}\label{lem:xi_0} Let $x\in \R^d\setminus\{0\}$. { The following quantity is positive:}
\begin{equation}\label{eq:defxi0}
    \xi_0(x):=\sup\left\{\xi>0: J_x(\xi)<\inf_{y\in\R^d, s\ge 1/2}I\left(s,y\right)\right\}.
\end{equation} 
\end{lem}
\begin{proof}%[Proof of Lemma \ref{lem:xi_0}]
By \eqref{eq:propI}, 
\aln{\label{I center}
I(2s,0)\leq I(s,x)+I(s,-x)=2I(s,x).
} Thus, 
\[\inf_{y\in\R^d, s\ge 1/2}2I\left(s,y\right)\geq I(1,0)>0.\]
Since  $J_x$ is continuous and $J_x(0)=I(0,0)=0$, $\xi_0(x)>0$.
\end{proof}
\section{Rate function for upper tail large deviations (Theorem~\ref{thm:main})}\label{sec:5}
The aim of this section is to prove that the rate function for upper tail large deviations coincides with our function $J_x$. Hereafter, we focus only on the direction $\mathbf{e}_1$ and  we write $J(\xi):=J_{\mathbf{e}_1}(x)$ and $\xi_0:=\xi_0(\mathbf{e}_1)$. However, all of our results can be  extended to any direction without difficulty. In particular, the symmetries in regards to the direction $\mathbf{e}_1$ are never used in our proofs. We split the proof of the theorem into two subsections; namely the lower bound and the upper bound. 
\subsection{Lower bound}\label{section:lower bound}
We will need the following lemma for the lower bound.  %(see Section~\ref{section:lower bound} below). 
Let $\beta\in(\alpha_d,1)$. %that we can assume that the cut-point is connected to $n\mathbf{e}_1$ up to paying a negligible price.
\begin{lem}\label{lem:toinfinity} {There exists $n_0\in\N$ such that for any $s\in[0,1/2)$, $\|x\|_1\le 1/2$} with $I(s,x)\le \inf_{y\in \mathbb R^d}I(1/2,y)$, and $n\ge n_0$, 
\begin{equation*}
    \begin{split}
        \P( \exists s_n\ge sn\quad w\in \Lambda_{n^{\alpha_d}}(nx): \,{\rm B}_{s_n}\setminus {\rm B}_{s_n-1}=\{w\}, {n\mathbf{e}_1\notin {\rm B}_{s_n}}, w\leftrightarrow n\mathbf{e}_1)\ge e^{-n^{{\beta}}\log ^2n}\P(\kA_{s,x}(n)).
    \end{split}
\end{equation*}
\end{lem}

\begin{proof}%[Proof of Lemma~\ref{lem:toinfinity}]
Let $n_0$ be as in the statement of Lemma \ref{lem:drilling} and $K$ be as in the statement of Proposition \ref{prop:sizeboundary}. 
Thanks to Lemma \ref{lem:drilling} and Proposition \ref{prop:sizeboundary} for $n\ge n_0$
\begin{equation}\label{eq:applydrilling}
    \frac{1}{2}e^{-n^{\alpha_d}}\P(\kA_{s,x}(n))\le e^{-n^{\alpha_d}}\P( \kA_{s,x}^K(n))\leq \P\left(\kA^{\rm free}_{s,x}(n)\right).
\end{equation}
Let $\beta\in(\alpha_d,1)$.
We define the event $\kF$ as
\[\kF:=\left\{\exists w\in\Lambda_n(nx): \frac{|\sC_\infty \cap \Lambda_{n^\beta}(w)|}{|\Lambda_{n^\beta}(w)|}\le \frac{3}{4}\theta(p)\right\}\cup \{\sC_\infty \cap \Lambda_{n^{\alpha_d}}(n\mathbf{e}_1)=\emptyset\}.\]
 Thanks to Theorem~\ref{thm:holes} and Theorem~\ref{thm:density}, we have for $n$ large enough depending on $p,d$,
\begin{equation}
    \begin{split}
        \P(\kF)&\le 2e^{-cn^{2\alpha_d}}.
    \end{split}
\end{equation}
Besides, note that \[I\left(\frac 2 3,x\right)=\frac{4}{3}I\left(\frac 1 2,\frac 3 4x\right)\ge \frac{4}{3} \inf_{y\in\R^d}I\left(\frac 1 2,y\right)>I(s,x).\]
Denote 
\[\kA_{2/3}(n):=\{\exists s_n\ge \frac 2 3 n,\, \exists w_n\in \Lambda_{n}; \,\rmB_{s_n}\setminus\rmB_{s_n-1}=\{w_n\}\}.\]
In particular, this yields for $n$ large enough depending on $p,d$,
\[\P(\kA^{\rm free}_{s,x}(n))\le \P(\kA^{\rm free}_{s,x}(n)\setminus \kA_{2/3}(n))+\P( \kA_{2/3}(n))\le 2\P(\kA^{\rm free}_{s,x}(n)\setminus \kA_{2/3}(n)).\]
Therefore, we have for $n$ large enough depending on $p,d$,
\begin{equation}\label{eq:lemtoinf}
\P(\kA^{\rm free}_{s,x}(n))\le \P(\kA^{\rm free}_{s,x}(n)\cap \kF^ c{\setminus \kA_{2/3}(n)})+\P(\kF)\le 3\P(\kA^{\rm free}_{s,x}(n)\cap \kF^ c{\setminus \kA_{2/3}(n)}).
\end{equation}
On the event $\kA^{\rm free}_{s,x}(n)\cap\kF^c {\setminus \kA_{2/3}(n)}$, denote by $w_n$ the cut-point and $s_n$ the associated time, where we have  $s_n\le 2n/3$ and $\Lambda_{n^{\alpha_d}}(n\mathbf e_1)\cap \rmB_{s_n}=\emptyset$.
Denote by $\mathcal P$ the set of points in $\Lambda_{n^{\beta}}(w_n)\setminus \rmB_{s_n}$ connected to $w_n$ in $\Z^d\setminus \rmB_{s_n-1}$ with less than $dn^ {\beta}$ edges. 
%By definition the set $\mathcal P$ is connected.
On the event ${\kA^{\rm free}_{s,x}(n)\cap \kF^ c{\setminus \kA_{2/3}(n)}}$, 
we claim that
\begin{equation}\label{eq:boundonP}
    \frac{|\mathcal P|}{|\Lambda_{n^\beta}|}\ge 1-\frac 2 3\theta(p).
\end{equation}
Since $\frac{|\sC_\infty\cap \Lambda_{n^\beta}|}{|\Lambda_{n^\beta}|}>3\theta(p)/4$,  this implies $\mathcal P\cap\sC_\infty \ne \emptyset$.
% Let $k\in[d]\setminus\{i,j\}$ where $i,j$ correspond to the indices in the event $\kA^{\rm free}_{s,x}(n)$.

We move on to \eqref{eq:boundonP}. Let $i,j$ the indices in the event $\kA^{\rm free}_{s,x}(n)$. Without loss of generality, we assume that $i=2$ and $j=1$.
In particular, we have
$\rL_2(w_n)\cap \rmB_{s_n}=\{w\}$ and $\forall k\neq 1\quad \rN_k(\rH_1(w_n)\cap \rmB _{s_n})\le n^{\alpha_d}$.
We define  $E_n^1,E_n^2,\dots,E_n^d$ inductively as follows: Set
\[E_n^1:=\{w_n\}.\]
 Assume $E_n^1,E_n^2,\dots,E_n^{k-1}$ have already been defined. We define
 \[E_n^{k}:=\bigcup_{y\in E_n^{k-1}:\, \rL_{k}(y)\cap \rmB_{s_n-1}=\emptyset} \rL_{k}^{n^\beta}(y).\]
Finally, define 
 \[\mathcal T:= \bigcup _{x\in E_n^d: \,\rL_1(x)\cap \rmB_{s_n-1}=\emptyset }\rL_1^{n^{\beta}}(x).\]
 By construction, all the points in $E_n^{k+1}$ are connected to a point in $E_n^k$ by a path in $\Z^d\setminus \rmB_{s_n-1}$ of length at most $n^ {\beta}$. It yields that all the points in
 $\mathcal T$ are connected to $w_n$ by a path in $\Z^d\setminus \rmB_{s_n-1}$ of length at most $dn^ {\beta}$.
 Since $\rN_{k+1}(\rH_1(w_n)\cap \rmB _{s_n})\le n^{\alpha_d}$, we have
 $|E_n^{k+1}|\ge 2n ^\beta  {(|E_n^k|- n^{\alpha_d})}$. 
 Moreover, since $|\Gamma|\le K n$, we have 
 and $|\mathcal T|\ge 2n^\beta {(|E_n^d|- Kn)}$.
 Finally it is easy to check that for $n$ large enough the density of $\mathcal T$ is larger than $1- 2\theta(p)/3$ and that $\mathcal T\subset \mathcal P$.
 
 % Hence, we can find $n^{2\alpha_d}$ points connected to $w_n$ in $\Z^d\setminus \rmB_{s_n-1}$ with less than $2n^{\alpha_d}$ edges (recall that $\rL_i(w_n)\cap \rmB_{s_n}=\{w_n\}$) and inequality \eqref{eq:boundonP} follows. 

As we mentioned below \eqref{eq:boundonP}, on the event $\kA^{\rm free}_{s,x}(n)\cap\kF^c$, the set $\mathcal P$ must intersect $\sC_\infty$. On the event ${\kA^{\rm free}_{s,x}(n)\cap \kF^c\setminus \kA_{2/3}(n)}$, there exists a path $\mathfrak{p}$ in $\Z^ d\setminus \rmB_{s_n-1}$ starting at $w_n$ and ending in $\mathcal P\cap \sC_\infty$ of length at most $d n^{\alpha_d}$. Moreover, on the event ${\kA^{\rm free}_{s,x}(n)\cap \kF^c\setminus \kA_{2/3}(n)}$, since $s_n\leq 2n/3,$ there exists a path $\mathfrak{p}'$ between $n\mathbf{e}_1$ and $\sC_\infty{\setminus \rmB_{s_n}}$ of length at most $d n^{\alpha_d}$.
Applying Lemma \ref{lem:resample} with $E_1=\mathfrak p \cup \mathfrak p '$ and $E_0=\emptyset$, we have for $n$ large enough depending on $d,p$,
\al{
 10 e^{-(\log n)^2 n^ {\beta}}&\P({\kA^{\rm free}_{s,x}(n)\cap\kF^c \setminus \kA_{2/3}(n)})
\\&\le  \P( \exists t\geq sn\quad \exists w_n\in \Lambda_{n^{\alpha_d}}(nx); \,{\rm B}_{t}\setminus {\rm B}_{t-1}=\{w_n\}, w_n\leftrightarrow n\mathbf{e}_1, n\mathbf e_1\notin\rmB_{s_n}).
}
Combining this with \eqref{eq:applydrilling} and \eqref{eq:lemtoinf}, The result follows.
\end{proof}

\begin{prop}\label{prop:LLD}Let {$\xi\in(0,\xi_0)$}. it holds that
\[\liminf_{n\rightarrow\infty}\frac{1}{n}\log \P(\mu(\mathbf{e}_1)(1+\xi)n<\cD^{\G_p}(0,n\mathbf{e}_1)<\infty)\ge - J(\xi).\]
\end{prop}

\begin{proof}%[Proof of Proposition \ref{prop:LLD}]
Let $s> 0$ and $x\in\R^d$ be such that
\begin{equation}\label{eq:condref}
s+\mu(x-\mathbf{e}_1)>  (1+\xi)\mu(\mathbf{e}_1).%,\quad  I(s,x)\le \inf_{y\in\R^d}I(1/2,y)=:J(\xi_0).
\end{equation}
%Note that by definition of $\xi_0$, any $(s,x)$ with $I(s,x)>\inf_{y\in\R^d}I(1/2,y)$ satisfies $I(s,x)\geq J(\xi_0)\ge J(\xi)$.
Set 
\[\kA'_{s,x}(n):=\{\exists s_n\ge sn\quad x_n\in \Lambda_{n^{\alpha_d}}(nx): \,{\rm B}_{s_n}\setminus {\rm B}_{s_n-1}=\{x_n\}, {n\mathbf{e}_1\notin {\rm B}_{s_n}}\}.\]
Thanks to the lower tail large deviation {\eqref{lower ldp}}, there exists $c>0$ such that for $n$ large enough,
\begin{equation}\label{eq:lld}
   \P(\cD^{\G_p}(\Lambda_{n^{\alpha_d}}(nx),n\mathbf{e}_1)\leq \mu(\mathbf{e}_1)(1+\xi)n-sn)<e^{-cn}.
\end{equation}
On the event  $\kA'_{s,x}(n)$, let $s_n$ be the smallest integer at least $sn$ such that $|\rmB_{s_n}\setminus\rmB_{s_n-1}|=1$ and $\rmB_{s_n}\setminus\rmB_{s_n-1}\subset \Lambda_{n^{\alpha_d}}(nx)$. 
Let $\sC_n$ be the set of admissible $\rmB_{s_n}$ where $s_n$ is the smallest integer at least $sn$ such that $|\rmB_{s_n}\setminus\rmB_{s_n-1}|=1$,  $n\mathbf{e}_1\notin {\rm B}_{s_n}$ and $\rmB_{s_n}\setminus\rmB_{s_n-1}\subset \Lambda_{n^{\alpha_d}}(nx)$. For $C\in\sC_n$, denote by $E(C)$ the edges that determine $\{\rmB_{s_n}=C\}$:
\[E(C):=\{\langle x,y\rangle\in \E^d: x\in  C\}\setminus \{\langle w_n,y\rangle\in \E^d:y\notin C\},\]
where 
$ \{w_n\}:= \rmB_{s_n}\setminus\rmB_{s_n-1}$.
We have
\begin{equation*}
    \begin{split}
    &\P(\kA'_{s,x}(n),\cD^{\G_p}(0,n\mathbf{e}_1)\leq\mu(\mathbf{e}_1)(1+\xi)n)\\
     %   &\P(\exists s_n\ge sn, \exists x_n\in \Lambda_{n^{\alpha_d}}(nx); \rmB_{s_n}\setminus \rmB_{s_n-1}=\{x_n\},\cD^{\G_p}(0,n\mathbf{e}_1)<\mu(\mathbf{e}_1)(1+\xi)n)\\
        &\leq \P(\exists s_n\ge sn, \exists w_n\in \Lambda_{n^{\alpha_d}}(nx);~\rmB_{s_n}\setminus \rmB_{s_n-1}=\{w_n\},\cD^{\G_p\setminus E(\rmB_{s_n})}(w_n,n\mathbf{e}_1)\leq \mu(\mathbf{e}_1)(1+\xi)n-s_n)\\
         &\le\P(\exists s_n\ge sn, \exists w_n\in \Lambda_{n^{\alpha_d}}(nx);~\rmB_{s_n}\setminus \rmB_{s_n-1}=\{w_n\},\cD^{\G_p\setminus E(\rmB_{s_n})}(\Lambda_{n^{\alpha_d}}(nx),n\mathbf{e}_1)\leq  \mu(\mathbf{e}_1)(1+\xi)n-sn).
        % &= \P(\kA_{s,x}(n))\P(\cD^{\G_p\setminus E(\rmB_{s_n})}(\Lambda_{n^{\alpha_d}}(nx),n\mathbf{e}_1)<\mu(\mathbf{e}_1)(1+\xi)n-sn|\kA_{s,x}(n))
        % &\le\P(\kA_{s,x}(n))\P(\cD^{\G_p}(\Lambda_{n^{\alpha_d}}(nx),n\mathbf{e}_1)<\mu(\mathbf{e}_1)(1+\xi)n-sn)
    \end{split}
\end{equation*}
{This is further bounded from above by}
\begin{equation*}
    \begin{split}
%        \P&(\exists s_n\ge sn, x_n\in \Lambda_{n^{\alpha_d}}(nx); \rmB_{s_n}\setminus \rmB_{s_n-1}=\{x_n\},\cD^{\G_p\setminus E(\rmB_{s_n})}(\Lambda_{n^{\alpha_d}}(nx),n\mathbf{e}_1)<\mu(\mathbf{e}_1)(1+\xi)n-sn)\\&=\sum_{C\in\sC_n}\P(\rmB_{s_n}=C,\,\cD^{\G_p\setminus E(C)}(\Lambda_{n^{\alpha_d}}(nx),n\mathbf{e}_1)<\mu(\mathbf{e}_1)(1+\xi)n-sn)\\
        &\sum_{C\in\sC_n}\P(\rmB_{s_n}=C)\,\P(\cD^{\G_p\setminus E(C)}(\Lambda_{n^{\alpha_d}}(nx),n\mathbf{e}_1)\leq \mu(\mathbf{e}_1)(1+\xi)n-sn)\\
        &\le \P(\kA'_{s,x}(n))\P(\cD^{\G_p}(\Lambda_{n^{\alpha_d}}(nx),n\mathbf{e}_1)<\mu(\mathbf{e}_1)(1+\xi)n-sn) {\leq \P(\kA'_{s,x}(n))e^{-cn}\le \P(\kA_{s,x}(n))e^{-cn}.}
    \end{split}
\end{equation*}
where in the last inequality we have used \eqref{eq:lld}. Therefore, we have
\begin{equation*}
    \begin{split}
       \P&(\exists s_n\ge sn, w_n\in \Lambda_{n^{\alpha_d}}(nx), \rmB_{s_n}\setminus \rmB_{s_n-1}=\{w_n\},{n\mathbf{e}_1\notin {\rm B}_{s_n}}, w_n\leftrightarrow n\mathbf{e}_1)\\&\le\P(\kA'_{s,x}(n), \cD^{\G_p}(0,n\mathbf{e}_1)<\infty)
       \\&\leq \P(\kA'_{s,x}(n),\mu(\mathbf{e}_1)(1+\xi)n<\cD^{\G_p}(0,n\mathbf{e}_1)<\infty)+ \P(\kA_{s,x}(n),\cD^{\G_p}(0,n\mathbf{e}_1)\le \mu(\mathbf{e}_1)(1+\xi)n)\\
       &\le \P(\mu(\mathbf{e}_1)(1+\xi)n<\cD^{\G_p}(0,n\mathbf{e}_1)<\infty)+\P(\kA_{s,x}(n))e^{-cn}.
    \end{split}
\end{equation*}
Besides, thanks to Lemma \ref{lem:toinfinity}, we have
\begin{equation*}
    e^{ {o(n)}}\P(\kA_{s,x}(n))\le   {\P}(\mu(\mathbf{e}_1)(1+\xi)n<\cD^{\G_p}(0,n\mathbf{e}_1)<\infty)+ \P(\kA_{s,x}(n))e^{-cn}.
\end{equation*}
Using Theorem \ref{prop:existencelimit}, we get by taking the liminf in the previous inequality
\[\liminf_{n\rightarrow\infty}\frac{1}{n}\log \P(\mu(\mathbf{e}_1)(1+\xi)n<\cD^{\G_p}(0,n\mathbf{e}_1)<\infty)\ge -I(s,x)\,.\]
Taking the supremum over $(s,x)$ with \eqref{eq:condref} and the continuity of $I$ (Theorem~\ref{prop:existencelimit}), the claim follows.
\end{proof}

\subsection {Upper bound}\label{sec:4}
 We first prove that on the upper tail large deviation event, there is a space-time cut-point with high probability. 
\begin{prop}[Creation of  cut-points on the large deviation event]\label{prop:creationcutpoint} For any $\xi>0$, there exists $s_0>0$ such that for any $\ep>0$,
\begin{equation*}
\begin{split}
   \limsup_{n\rightarrow\infty}\frac{1}{n}\log\P&(\mu(\mathbf{e}_1)(1+\xi)n<\cD^{\G_p}(0,n\mathbf{e}_1)<\infty) \\&\le \limsup_{n\rightarrow\infty}\frac{1}{n}\log\P\left(\bigcup_{\substack{s,t\in(\Z/n)\cap [0,s_0],\,x,y\in(\Z^ d/n)\cap [-s_0,s_0]^d:\\s+t+(1+\ep)\mu(x-y-\mathbf{e}_1){+\ep}\ge (1+\xi)\mu(\mathbf{e}_1)}}\kA_{s,x}(n)\cap(n\mathbf{e}_1+\kA_{t,y}(n))\right).
   \end{split}
\end{equation*}

\end{prop}
We need the following lemmas. We postpone their proofs to appendix.
\begin{lem}\label{lem:proj}
For any finite $S\subset \Z^d$, there exists $i\in[d]$ such that $$|P_i(S)|\ge \frac{1}{2}|S|^{\frac{2}3}.$$
\end{lem}
For $x,y\in\R^ d$, we denote by $[x,y]$ the segment joining $x$ and $y$. 
\begin{lem}\label{lem:dislines}
Let $d\ge 3$, $K\ge\ell\ge  1$ and  $m\ge 1$.
Consider $S_ 1\subset \rH_1(0)\cap \Z^d$ and $S_ 2\subset \rH_1(\ell\mathbf{e}_1)\cap \Z^d$ such that $|S_1|=|S_2|=m$ and
\aln{\label{hypothesis star}
\max_{x\in S_1,\,y\in S_2} \|x-y\|_{\infty}{\leq} K.
}
Then, there exists  a bijection $\sigma$ from $S_1$ to $S_2$ such that 
\[\forall x\neq x'\in S_1\qquad \mathrm{d}_2([x,\sigma(x)], [x',\sigma(x')]) \ge \frac{1}{\sqrt{2}}\frac{\ell}{K},\]
{where ${\rm d}_2$ is the Euclidean distance.} In particular, the segments do not intersect.
\end{lem}

\begin{proof}[Proof of Proposition \ref{prop:creationcutpoint}] Let $\xi,\ep>0$. {Thanks to \eqref{eq:antalpisztora2}, there exists $K>0$ depending on $\xi$ such that for $n$ large enough,
\[\P(\mu(\mathbf{e}_1)(1+\xi)n<\cD^{\G_p}(0,n\mathbf{e}_1)<\infty)\le 2\P(\mu(\mathbf{e}_1)(1+\xi)n<\cD^{\G_p}(0,n\mathbf{e}_1)<Kn).\]
From now on, we will work on the event $\{\mu(\mathbf{e}_1)(1+\xi)n<\cD^{\G_p}(0,n\mathbf{e}_1)<Kn\}$.} 
 Set 
\aln{\label{def of s and t}
s^*_n:=\inf\{i\ge 1: |\rmB_i|\ge n^{7/4}\}\qquad\text{and}\qquad t^*_n:=\inf\{i\ge 1: |\rmB_i(n\mathbf{e}_1)|\ge n^{7/4}\},
}
with the convention $\inf \emptyset=+\infty$.  {Let $s_0:=3K+1$.} We consider two cases. 

 \noindent{\bf Case 1.} {Suppose that} $ s^*_n+t^*_n\ge\cD^ {\kG_p}(0,n\mathbf{e}_1)$.
Then, $\rmB_{s^*_n}\cap\rmB_{t^*_n}(n\mathbf{e}_1)\ne\emptyset$ and there exists $x_n\in\Z^ d$ such that 
\[\cD^ {\kG_p}(0,x_n)\le s^*_n,\quad \cD^ {\kG_p}(x_n,n\mathbf{e}_1)\le t^*_n,\quad \text{and}\quad \cD^ {\kG_p}(0,n\mathbf{e}_1)=\cD^ {\kG_p}(0,x_n)+ \cD^ {\kG_p}(x_n,n\mathbf{e}_1).\]
Denote by $s_n=\cD^ {\kG_p}(0,x_n)$ and $t_n=\cD^ {\kG_p}(x_n,n\mathbf{e}_1)$.  We claim that if
\aln{\label{goal shuta}
\rmB_{s_n}\setminus \rmB_{s_n-1}=\{x_n\}\text{  and }\rmB_{t_n}(x_n)\setminus (\rmB_{t_n-1}(x_n)\cup\rmB_{s_n})=\{n\mathbf{e}_1\},
}
then $\kA_{\mu(\mathbf{e}_1)(1+\xi),\mathbf{e}_1}(n)$ occurs. Indeed let $y\in \rmB_{s_n+t_n}\setminus  \rmB_{s_n+t_n-1}$ and $\gamma$ be a geodesic from $0$ to $y$. By hypothesis, $\gamma$ passes through $x_n$ at time $s_n$ and then passes through $n\mathbf{e}_1$ at time $s_n+t_n{=\cD^ {\kG_p}(0,n\mathbf{e}_1)>} \mu(\mathbf{e}_1)(1+\xi)n$, which implies $y=n\mathbf{e}_1$ and the claim follows.

Let us resample the configuration to create a cut-point at $n\mathbf{e}_1$ {via Lemma \ref{lem:resample}}.
We take a geodesic $\gamma$ from $0$ to $n\mathbf{e}_1$ with a deterministic rule breaking ties. 
Since $|\rmB_{s_n-1}|\le n^{7/4}$ and $|\rmB_{t_n}(x_n)\cap \rmB_{t_n-1}(n\mathbf{e}_1)|\le n^{7/4}$,  {if $s_n\ge 2n^{7/8}$ and $t_n\ge 2n^{7/8}$, then there exist $r_0\in  [s_n-1]\backslash [s_n-2n^{7/8}]$ and  $r_1\in [t_n]\backslash [t_n-2n^{7/8}]$} such that $$|\rmB_{r_0}\backslash \rmB_{r_0-1}|\leq n^{7/8}\text{ and   }|(\rmB_{r_1}(x_n)\backslash \rmB_{r_1-1}(x_n))\cap\rmB_{t_n-1}(n\mathbf{e}_1)|\leq n^{7/8}.$$
 If $s_n< 2n^{7/8}$, then we can simply set $r_0:=0$; if $t_n<2n^{7/8}$, we set $r_1:=0$.
%and $r_2\in [t]\backslash [t-n^{7/8}]$ such that $|\rmB_{r_2}(n\mathbf{e}_1)\backslash \rmB_{r_2-1}(n\mathbf{e}_1)|\leq n^{7/8}$. 
Set 
\begin{equation*}
    \begin{split}
        E_0:=& {\left\{\langle v,w\rangle\in \E^d\setminus \gamma:~v\in \gamma\cap \Big[(\rmB_{s_n}\backslash \rmB_{r_0})\cup ((\rmB_{t_n}(x_n)\backslash \rmB_{r_1}(x_n))\cap\rmB_{t_n-1}(n\mathbf{e}_1))\Big]\right\}}\\
        &\cup \{e\in \E^d\setminus \gamma:~e\text{ connects $\rmB_{r_0}\backslash \rmB_{r_0-1}$ and $\rmB_{r_0-1}$}\}\\
        &\cup \{e\in \E^d\setminus \gamma:~e\text{ connects $(\rmB_{r_1}(x_n)\backslash \rmB_{r_1-1}(x_n))\cap \rmB_{t_n-1}(n\mathbf{e}_1))$ and $\rmB_{{r_1-1}}(x_n)$}\}.
        % &{\SN \cup \{e\in \E^d\setminus \gamma:~e\text{ connects $(\rmB_{r_2}(n\mathbf{e}_1)\backslash \rmB_{r_2-1}(n\mathbf{e}_1))$ and $\rmB_{r_2}(n\mathbf{e}_1)$}}\}
    \end{split}
\end{equation*}
Consider a resampled configuration $(\tau^{\rm r}_e)_{e\in \E^d}$ as in Lemma \ref{lem:resample} with $E_0$ defined above and  $E_1=\emptyset$. We now prove that if all the edges of $\tau^{\rm r}$ in $E_0$ are closed, then \eqref{goal shuta} holds for $\tau^{\rm r}$, and thus $\kA_{\mu(\mathbf{e}_1)(1+\xi),\mathbf{e}_1}(n)$ occurs for $\tau^{\rm r}$ as well. The first part of \eqref{goal shuta} follows easily. For the second part, by the choice of $E_0$, ${\rmB^{\rm r}_{t_n}(x_n)}\subset \rmB_{s_n}\cup \rmB_{r_1-1}(x_n)\cup \gamma$. Therefore, the only vertex in $\rmB_{t_n}(x_n)\setminus \rmB_{s_n}$ at distance $t_n$ from $x_n$ is $n\mathbf{e}_1$, and thus \eqref{goal shuta} follows. Since $|E_0|\leq (2d)^{2d} n^{7/8}$ and $E_0\subset [-n^2,n^2]^d$ for $n$ large enough, by Lemma \ref{lem:resample},  %a resampled configuration $\tau^{\rm r}$,
\begin{equation}\label{eq:creationcut-point1}
\begin{split}
    &e^{-{(8d)^{2d}}(\log n)\,n^{7/8}}\P(\mu(\mathbf{e}_1)(1+\xi)n<\cD^{\G_p}(0,n\mathbf{e}_1)<\infty,s^*_n+t^*_n\ge\cD^ {\kG_p}(0,n\mathbf{e}_1))\\
   %& \le\P\left(\begin{array}{c}\mu(\mathbf{e}_1)(1+\xi)n<\cD^{\G_p^{\rm r}}(0,n\mathbf{e}_1)<\infty,\,\rmB_{s_n}^{\rm r}\setminus \rmB_{s-1}^{\rm r}=\{x\},\\ \rmB_{t}^{\rm r}(n\mathbf{e}_1)\setminus \rmB_{t-1}^{\rm r}(n\mathbf{e}_1)=\{x\},\,\rmB_{t}^{\rm r}(x_n)\setminus (\rmB_{t-1}^{\rm r}(x_n)\cup\rmB_s^{\rm r})=\{n\mathbf{e}_1\}\end{array}\right) \\
   %&=\P\left(\begin{array}{c}\mu(\mathbf{e}_1)(1+\xi)n<\cD^{\G_p}(0,n\mathbf{e}_1)<\infty,\,\rmB_{s_n}\setminus \rmB_{s-1}=\{x\},\\ \rmB_{t}(n\mathbf{e}_1)\setminus \rmB_{t-1}(n\mathbf{e}_1)=\{x\},\,\rmB_{t}(x_n)\setminus (\rmB_{t-1}(x_n)\cup\rmB_s)=\{n\mathbf{e}_1\}\end{array}\right) \\
   &\le \P\left(\kA_{\mu(\mathbf{e}_1)(1+\xi),\mathbf{e}_1}(n)\right) \le \P\left(
  \bigcup_{\substack{s,t\in(\Z/n)\cap [0,s_0],\,x_n,y_n\in(\Z^ d/n)\cap [-s_0,s_0]^d,\\s+t+(1+\ep)\mu(x-y-\mathbf{e}_1)+\ep\ge (1+\xi)\mu(\mathbf{e}_1)}}\kA_{s,x}(n)\cap(n\mathbf{e}_1+\kA_{t,y}(n))\right),
   \end{split}
\end{equation}
 where {$\kA_{0,0}=\{\rmB_0=\{0\}\}$ always occurs}.

\noindent{\bf Case 2.} Suppose that $s^*_n+t^*_n<\cD^ {\kG_p}(0,n\mathbf{e}_1)$.
In particular, we have $s_n^*+t_n^*\leq K n$ and $\rmB_{s^*_n}\cap\rmB_{t^*_n}(n\mathbf{e}_1)=\emptyset$. Let $\delta:=1/(18d)$, $N=\log ^ 2 n$ and $\ep>0$ be chosen later. We consider the macroscopic lattice of sidelength $N$ for the parameter $\ep$. 
 Set
\[\cE_n:=\left\{\#\{\mathbf{i}\in \Lambda_{n^2}\cap\Z^ d: \mathbf{i}\text{ is $\ep$-bad}\}>n \right\}.\]
 By Lemma~\ref{lem:nobadbox}, we have $\P(\cE_n)\leq e^{-n \log{n}}$. Hence, we suppose $\cE^c_n$ from now on. Let ${\rm C}_1$ be the set of boxes intersected by $\rmB_{s^*_n}$ and ${\rm C}_2$ the set of boxes intersected by $\rmB_{t^*_n}(n\mathbf{e}_1)$.
Note that $|{\rm C}_1|\ge n^ {7/4}/(2^dN^ d)$.
Thanks to Lemma \ref{lem:proj}, there exists ${\rm i}_1{\in [d]}$ such that  for $n$ large enough,
\[|P_{{\rm i}_1}({\rm C}_1)|\ge \frac{n^{7/6}}{2^{d+1}\log ^{2d} n}\geq 2n^{10/9}.\]
Denote by $C'_1$ the set of macroscopic sites that are connected in the macroscopic lattice by a macroscopic, good path of length at most $n^{1-\delta}$ to a site in ${\rm C}_1$.
The number of disjoint macroscopic lines $\rL_{{\rm i}_1}(w)$ with $w\in {\rm C}_1$ is at least $ 2n^{10/9}$.  We say that a line $\rL_{{\rm i}_1}(w)$ is good if all sites in $\rL_{{\rm i}_1}^n(w)$ are good. On the event $\cE^c_n$,  there are at least $2 n^{10/9}-n\ge n^{10/9}$ good lines of these disjoint macroscopic lines.
It follows that
$|C'_1|\ge n^{19/9-\delta}.$  We can define similarly $C'_2$.
 Since $s_n^*<\kD^{\kG_p}(0,n\mathbf e_1 )\le Kn$, we have $C'_1,C'_2\subset \cup_{k=-Kn}^{2Kn}\rH_1(k\mathbf{e}_1).$

Besides, since $(\rH_1(k\mathbf{e}_1), -Kn \le k \le 2Kn)$ are disjoint, by pigeon-hole principle, there exists $w_1\in \{k\mathbf{e}_1: -Kn\le k\le 2Kn\}$ such that
\[|\rH_{1}(w_1)\cap C'_1|\ge \frac{1}{4K}n^ {10/9-\delta}.\]
Similarly,  there exists $w_2\in \{k\mathbf{e}_1: -Kn\le k\le 2Kn\}$ such that
\[|\rH_{1}(w_2)\cap C'_2|\ge \frac{1}{4K} n^ {10/9-\delta }.\]
% It follows from the fact that there are many good lines $\rL_1(z)$ intersecting the  good lines $\rL_{i_1}(w)$ with $w\in {\rm C}_1$.
Let us take $S_i\subset \rH_{1}(w_i)\cap C'_i$ for $i=1,2$ such that $|S_1|=|S_2|=:m\geq  \frac{1}{4K} n^ {10/9-\delta }.$ We write $S_1=:\{x^1,\cdots,x^m\}$ and $S_2=:\{y^1,\cdots,y^m\}$.

We take an integer $|k|\le n ^{1-\delta}$ such that
\[|(w_1-w_2)\cdot \mathbf{e}_1-k|\ge n^{1-\delta}.\]
Consider now the set $S'_2:=S_2+k\mathbf{e}_1$. 
 {By using Lemma \ref{lem:dislines}, we obtain a bijection $\sigma:S_1\to S_2'$ and straight lines $(L'_i)_{i=1}^m$ joining points $x_i$ and $\sigma(x_i)$, where ${\rm d_2}(L'_i,L'_j)\geq n^{\delta}/4K$. 
 We denote by $L_i$ the concatenation of $L'_i$ with the straight line joining $\sigma(x_i)$ and $\sigma(x_i)-k\mathbf{e}_1\in S_2$.}
This implies that for each macroscopic site $x\in\Z^d$, we have
\[\#\{i:\mathrm{d}_\infty(x,L_i)\le 1\}\le O(n^{\delta(d-1)})\le n^{\delta d} .\]
Note that $1+\delta d< \frac{10}{9}-\delta$ with $\delta=1/(18d)$. On the event $\cE^c_n$, the number of $L_i$ crossing at least one bad box is at most $n^{1+\delta d}$, hence there exists at least one path $L_i$ between some $\mathbf{x}_1\in S_1$ and $\mathbf{x}_2\in S_2$ such that all the macroscopic sites $\mathbf{x}$ with $\mathrm{d}_\infty(L_i,\mathbf{x})\le 1$ are good. In particular, there exists a macroscopic path of good sites joining $\mathbf{x}_1\in S_1$ and $\mathbf{x}_2\in S_2$.
Let $x_n,y_n,z_n$ be microscopic points in the largest open cluster of  {the boxes corresponding to $\mathbf{x}_1,\mathbf{x}_2,\mathbf{x}_2+k\mathbf{e}_1$, respectively.} By definition of $C'_1$ and $C'_2$, there exist good, macroscopic paths of length at most $n^ {1-\delta}$ to boxes in ${\rm C}_1$ and ${\rm C}_2$, respectively. We conclude  using Lemma \ref{lem:connexiongoodbox} that
\[\cD^ {\kG_p}(0,x_n)<s^*_n+{O(N^d n^{1-\delta})},\quad
\cD^ {\kG_p}(y_n,n\mathbf{e}_1)<t^*_n+O(N^d n^{1-\delta}).\]
Moreover, since the path $L'_i$ joining $\mathbf{x}_1$ and $\mathbf{x}_2+k\mathbf{e}_1$ is a straight line, and all the macroscopic sites $\mathbf{x}$ with $\mathrm{d}_\infty(L_i,\mathbf{x})\le 1$ are good,  using a similar argument as in the proof of Proposition \ref{prop:connexpoints}, we have
\al {\cD^ {\kG_p}(x_n,y_n)&\le \cD^ {\kG_p}(x_n,z_n)+\cD^ {\kG_p}(z_n,y_n)\\
&\le (1+2d\ep)\mu(x_n-z_n)+O(N^d n^{1-\delta}) \leq  (1+2d\ep)\mu(x_n-y_n)+O(N^d n^{1-\delta}).
}

Next, we create two cut-points at $x_n$ and $y_n$.
Let $\gamma_{v,w}$ be a geodesic from $v$ to $w$ with a deterministic rule breaking ties. Set  $s_n:=\kD(0,x_n)\leq s^*_n+O(N^dn^{1-\delta})$ and $t_n:=\kD(n\mathbf{e}_1,y_n)\leq t^*_n+O(N^d n^{1-\delta})$. We denote by $\gamma'$ the concatenation of $\gamma_{0,x_n}$, $\gamma_{x_n,y_n}$, and $\gamma_{y_n,n\mathbf{e}_1}$, i.e. $\gamma':=\gamma_{0,x_n}\oplus \gamma_{x_n,y_n}\oplus\gamma_{y_n,n\mathbf{e}_1}$. First, since we assumed $\kD^ {\kG_p}(0,n\mathbf{e}_1)\ge \mu(\mathbf{e}_1)(1+\xi)n$,  we have
\[s_n+t_n+ (1+2d\ep)\mu(x-y){+2d\ep n} \ge \kD^ {\kG_p}(0,x_n)+ \kD^ {\kG_p}(x_n,y_n)+\kD^ {\kG_p}(y_n,n\mathbf{e}_1)\ge \mu(\mathbf{e}_1)(1+\xi)n.\]
Since $|\rmB_{\min(s^*_n-1,s_n)}|\le n^{7/4}$ and $|\rmB_{\min(t^*_n-1,t_n)}(n\mathbf{e}_1)|\le n ^{7/4}$, there exist $r_0\in [\min(s^*_n-1,s_n)]\backslash [\min(s^*_n-1,s_n)-n^{7/8}]$ 
and $r_1\in[\min(t^*_n-1,t_n)]\backslash [\min(t^*_n-1,t_n)-n^{7/8}]$ such that 
$$ {|\rmB_{r_0}\backslash \rmB_{r_0-1}|\leq n^{7/8},\quad|\rmB_{r_1}(n\mathbf{e}_1)\backslash \rmB_{r_1-1}(n\mathbf{e}_1)|\leq n^{7/8},}$$
where we can simply set $r_0:=0$ if $\min(s^*_n-1,s_n)<n^{7/8}$; $r_1:=0$ if $\min(t^*_n-1,t_n)<n^{7/8}$. Set 
\begin{equation*}
    \begin{split}
        E_0:=&\{\langle v,w\rangle \in \E^d\setminus \gamma':~v\in (\gamma_{0,x_n}\setminus  \rmB_{r_0})\cup(\gamma_{y_n,n\mathbf{e}_1}\setminus\rmB_{r_1}(n\mathbf{e}_1))\}\\&\cup \{e\in \E^d\setminus \gamma':~e\text{ connects $\rmB_{r_0}\backslash \rmB_{r_0-1}$ and $\rmB_{r_0-1}$}\}\\&\cup \{e\in \E^d\setminus \gamma':~e\text{ connects $\rmB_{r_1}(n\mathbf{e}_1)\backslash \rmB_{r_1-1}(n\mathbf{e}_1)$ and $\rmB_{r_1}(n\mathbf{e}_1)$}\}.
    \end{split}
\end{equation*}
Note that 
$$\max\{|\gamma_{0,x_n}\setminus \rmB_{r_0}|,|\gamma_{y_n,n\mathbf{e}_1}\setminus \rmB_{r_1}(n\mathbf{e}_1)|\}= O(N^d n^{1-\delta}).$$
It follows that for $n$ large enough, $|E_0|= O(N^d n^{1-\delta}).$
 {By similar reasoning as in Case 1, we can create the two space-time cut-points at $(s_n,x_n)$ and $(t_n,y_n)$ for the resampled configuration $\tau^{\rm r}$ as in  Lemma \ref{lem:resample} with $E_0$ defined above and $E_1=\emptyset$, since $\gamma'$ is still open and the chemical distances cannot become larger after resampling.} Recall that $s_0=3K+1$. Therefore, {transforming $x=x_n/n,y=y_n-\mathbf{e}_1$, $s=s_n/n$ and $t=t_n/n$}, by Lemma \ref{lem:resample}, we have
\begin{equation}\label{eq:creationcut-point2}
\begin{split}%\sncomment{why $2(\log n)\,n^{1-\delta}${\BD changed}}
   \P&(\mu(\mathbf{e}_1)(1+\xi)n<\cD^{\G_p}(0,n\mathbf{e}_1)<\infty,s^*_n+t^*_n<\cD^ {\kG_p}(0,n\mathbf{e}_1)) \\
   &\le \exp(O(N^d n^{1-\delta}))\P\left(\bigcup_{\substack{s,t\in(\Z/n)\cap [0,s_0],\,x_n,y_n\in(\Z^ d/n)\cap [-s_0,s_0]^d\\s+t+(1+2d\ep)\mu(x-y-\mathbf{e}_1)+2d\ep n\ge (1+\xi)\mu(\mathbf{e}_1)}}\kA_{s,x}(n)\cap(n\mathbf{e}_1+\kA_{t,y}(n))\right).
   \end{split}
\end{equation}
Together with \eqref{eq:creationcut-point1}, the result follows by taking the $\limsup$ and changing $2d\ep$ by $\ep$.
\end{proof}
Using the proposition, we prove the upper bound.
\begin{prop}\label{prop:ULD}Let $\xi\in(0,\xi_0)$. We have
\[\limsup_{n\rightarrow\infty}\frac{1}{n}\log \P(\mu(\mathbf{e}_1)(1+\xi)n<\cD^{\G_p}(0,n\mathbf{e}_1)<\infty)\le - J(\xi).\]
\end{prop}
\begin{proof}%[Proof of Proposition \ref{prop:ULD}]
Let $\xi\in(0,\xi_0)$ and 
 $\ep>0$. {By continuity of $J$ (Proposition~\ref{prop:Jcont}), $J(\xi_0)=\inf_{x\in\R^d, s\ge 1/2}I\left(s,x\right)$. Hence, by definition of $\xi_0$, $J(\xi)<J(\xi_0)$ for any $\xi<\xi_0$.} Thanks to Proposition \ref{prop:LLD}, we have
\ben{\label{claim:contradiction}
\limsup_{n\rightarrow\infty}\frac{1}{n}\log\P(\mu(\mathbf{e}_1)(1+\xi)n<\cD^{\G_p}(0,n\mathbf{e}_1)<\infty)\geq -J(\xi)>-J(\xi_0).
} 

Thanks to Proposition \ref{prop:creationcutpoint}, we have
 \begin{equation*}
\begin{split}
  & \limsup_{n\rightarrow\infty}\frac{1}{n}\log\P(\mu(\mathbf{e}_1)(1+\xi)n<\cD^{\G_p}(0,n\mathbf{e}_1)<\infty) \\&\le \limsup_{n\rightarrow\infty}\frac{1}{n}\log\P\left(\bigcup_{\substack{s,t\in(\Z/n)\cap [0,s_0],\,x,y\in(\Z^ d/n)\cap [-s_0,s_0]^d\\s+t+(1+\ep)\mu(x-y-\mathbf{e}_1)+\ep\ge (1+\xi)\mu(\mathbf{e}_1)}}\kA_{s,x}(n)\cap(n\mathbf{e}_1+\kA_{t,y}(n))\right).
   \end{split}
\end{equation*}
\iffalse
Let $s_0>0$ be a sufficiently large constant depending on $\xi_0$ to be choose later.
Note that
\[\P\left(\bigcup_{\substack{s,t\ge0,\,x,y\in\Z^ d:\\\|x\|_1\ge s_0,\|y\|_1\le t}}\kA_{s,x}(n)\cap(n\mathbf{e}_1+\kA_{t,y}(n))\right)\le \P(\exists s_n\ge s_0 n\quad |\rmB_{s_n}\setminus \rmB_{s_n-1}|=1,\,\Diam(\rmB_{s_n})\ge s_0 n)\,.\]
 {By Proposition \ref{prop:sizeboundary}, since $\rmB\subset \Gamma$ implies $\Diam(\rmB)\leq |\Gamma|$,} there exists $s_0>0$ such that
\[\P(\exists s_n\ge s_0 n\quad |\rmB_{s_n}\setminus \rmB_{s_n-1}|=1,\,\Diam(\rmB_{s_n})\ge s_0 n)\le \exp(-2J(\xi_0)n).\]
%Let $s_0$ be such that the previous inequality holds.
It follows that
 \begin{equation*}
\begin{split}
   &\limsup_{n\rightarrow\infty}\frac{1}{n}\log\P(\mu(\mathbf{e}_1)(1+\xi)n<\cD^{\G_p}(0,n\mathbf{e}_1)<\infty) \\
   &\le {\limsup_{n\rightarrow\infty}\frac{1}{n}\log\P\left(\bigcup_{\substack{s,t\in(\Z/n)\cap [0,s_0],\,x,y\in(\Z^ d/n)\cap [-s_0,s_0]^d\\s+t+(1+\ep)\mu(x-y-\mathbf{e}_1)+\ep\ge (1+\xi)\mu(\mathbf{e}_1)}}\kA_{s,x}(n)\cap(n\mathbf{e}_1+\kA_{t,y}(n))\right).}
   \end{split}
\end{equation*}
\fi
%Hence, we can restrain\sout{ed} to $s,t\in[0,s_0]$ and $x,y$ such that $\|x\|_1\le s_0$ and $\|y\|_1\le s_0$ in order to apply Lemma~\ref{lem:aux}  later. 
Let $n_0$ be as in the statement of Lemma \ref{lem:aux} (depending on $s_0,d,p$). Let $n\ge n_0$.
By pigeon-hole principle,  there exist $x^*,y^*\in (\mathbb Z^d/n)\cap [-s_0,s_0]^d$ and $s^*,t^*\in (\mathbb Z/n)\cap [0,s_0]$ (that may depend on $\ep,n$) such that
\[s^*+t^*+(1+\ep)\mu(x^*-y^* {-\mathbf{e}_1})+\ep\ge (1+\xi)\mu(\mathbf{e}_1),\] and
\aln{\label{choice upper bound}
&\frac{1}{{(4s_0 n)^{2(d+1)}}}\P\left(\bigcup_{\substack{s,t\in(\Z/n)\cap [0,s_0],\,x,y\in(\Z^ d/n)\cap [-s_0,s_0]^d\\s+t+(1+\ep)\mu(x-y-\mathbf{e}_1)+\ep\ge (1+\xi)\mu(\mathbf{e}_1)}}\kA_{s,x}(n)\cap(n\mathbf{e}_1+\kA_{t,y}(n))\right)\notag\\
&\leq \P(\bar{\kA}_{s^*,x^*}(n)\cap(n\mathbf{e}_1+\bar{\kA}_{t^*,y^*}(n)),
}
where $\bar{\kA}_{s^*,x^*}(n):=\kA_{s^*,x^*}(n)\setminus \kA_{s^*+(1/n),x^*}(n)$.

Let us first assume that  these two events  $\bar{\kA}_{s^*,x^*}(n)$ and $(n\mathbf{e}_1+\bar{\kA}_{t^*,y^*}(n))$ do not occur disjointly, that is, there exists an edge used to achieve both two events. Note that on the event $\bar{\kA}_{s,x}(n)$, the cut-point is exactly located at time $s n$ for $s\in \Z/n$. Since we assumed $\rmB_{s^* n}(0)\cap \rmB_{t^* n}(n\mathbf{e}_1)\neq \emptyset$, this implies that $s^*+t^*\ge 1$. Without loss of generality, let us assume $s^*\ge 1/2$.
 By the uniform convergence on a compact set in Theorem~\ref{prop:existencelimit}, together with \eqref{choice upper bound}, we have
 \begin{equation}\label{eq:tocontradict}
    \begin{split}
        \limsup_{n\rightarrow \infty}&\frac{1}{n}\log \P\left(\bigcup_{\substack{s,t\in(\Z/n)\cap [0,s_0],\,x,y\in(\Z^ d/n)\cap [-s_0,s_0]^d\\s+t+(1+\ep)\mu(x-y-\mathbf{e}_1)+\ep\ge (1+\xi)\mu(\mathbf{e}_1)}}\kA_{s,x}(n)\cap(n\mathbf{e}_1+\kA_{t,y}(n))\right)\\
        &\le -\inf_{x\in \R^d, s\ge 1/2}I(s,x)=-J(\xi_0).
    \end{split}
\end{equation}
{which contradicts \eqref{claim:contradiction}. Therefore, the occurrences $\bar{\kA}_{s^*,x^*}(n)$ and $(n\mathbf{e}_1+\bar{\kA}_{t^*,y^*}(n))$ are disjoint.}

%Let $\delta>0$. Let $(\mathrm{x},\mathrm {s})$ be the closest point from $(x^*,s^*)$ in $\delta\Z^d\times \delta \Z$ satisfying $\mathrm{s}-s^*\ge \|\mathrm{x}-x^*\|_1$ with  a deterministic rule breaking ties. We define  $\mathrm{y},\mathrm{t}$ similarly.
%By similar reasoning as above, we have
%\al{
%\P(\kA_{s^*,x^*}(n))\le {e^{c(\delta n+n^{\alpha_d})}}\P(\kA_{\mathrm{s},\mathrm{x}}(n)) \text{and }
%\P(\kA_{t^*,y^*}(n))\le \P(\kA_{\mathrm{t},\mathrm{y}}(n))e^{c\delta n} e ^{cn^{\alpha_d}}.
%}
 Since it is a disjoint occurrence, by BK inequality, we have
\[\P(\bar{\kA}_{s^*,x^*}(n)\cap(n\mathbf{e}_1+\bar{\kA}_{t^*,y^*}(n)))\le \P(\kA_{s^*,x^*}(n))\P(n\mathbf{e}_1+\kA_{t^*,y^*}(n))=\P(\kA_{s^*,x^*}(n))\P(\kA_{t^*,y^*}(n)).\]
It follows that for $n\in\N$ large enough, the right-hand side of \eqref{choice upper bound} is bounded from above by 
\begin{equation*}
    \begin{split}
    %    \P&\left(\bigcup_{\substack{s,t\ge0,\,x,y\in\R^d:\\\|x\|_1\le s,\|y\|_1\le t,\\s+t+(1+\ep)\mu(x-y-\mathbf{e}_1)\ge (1+\xi)\mu(\mathbf{e}_1)}}\kA_{s,x}(n)\cap(n\mathbf{e}_1+\kA_{t,y}(n))\right)e^ {2c(\delta n+n^{\alpha_d})}\\
       % & \sum_{\substack{s,t\in(\Z/n)\cap [0,s_0],\,x,y\in(\Z^ d/n)\cap [-s_0,s_0]^d\\s+t+(1+\ep)\mu(y-\mathbf{e}_1)+\ep\ge (1+\xi)\mu(\mathbf{e}_1)-d\delta}}\P(\kA_{s,x}(n))\P(\kA_{t,y}(n))\\
        \max_{\substack{s,t\in [0,s_0],\,x,y\in [-s_0,s_0]^d\\s+t+(1+\ep)\mu(x-y-\mathbf{e}_1)+\ep\ge (1+\xi)\mu(\mathbf{e}_1)}}\P(\kA_{s,x}(n))\P(\kA_{t,y}(n)).
    \end{split}
\end{equation*}
 By the uniform convergence on a compact set in Theorem~\ref{prop:existencelimit}, %{transforming $y$ to $y-\mathbf{e}_1$,}
\begin{equation*}
    \begin{split}
        \limsup_{n\rightarrow \infty}&\frac{1}{n}\log \P\left(\bigcup_{\substack{s,t\in(\Z/n)\cap [0,s_0],\,x,y\in(\Z^ d/n)\cap [-s_0,s_0]^d\\s+t+(1+\ep)\mu(x-y-\mathbf{e}_1)+ {\ep}\ge (1+\xi)\mu(\mathbf{e}_1)}}\kA_{s,x}(n)\cap(n\mathbf{e}_1+\kA_{t,y}(n))\right)\\
        %&\le -2c\delta -\inf_{\substack{x,y\in\delta \Z^d,s,t\in\delta \sZ:\\ s+t+(1+\ep)\mu(y-\mathbf{e}_1)+\ep\ge (1+\xi)\mu(\mathbf{e}_1)-d\delta}}(I(s,x)+I(t,y))\\
        &\le -\inf_{\substack{s,t\in [0,s_0],\,x,y\in [-s_0,s_0]^d\\s+t+(1+\ep)\mu(x-y-\mathbf{e}_1)+ {\ep}\ge (1+\xi)\mu(\mathbf{e}_1)}}(I(s,x)+I(t,y)).
    \end{split}
\end{equation*}
%By continuity of $I$ (Theorem~\ref{prop:existencelimit}), by letting $\delta$ and then $\ep$ \sout{goes} {go} to $0$, we have 
Putting things together {with letting $\ep\to 0$}, we have
\begin{equation*}
    \begin{split}
       & \limsup_{n\rightarrow \infty}\frac{1}{n}\log \P\left(\bigcup_{\substack{s,t\in \Z/n,\,x,y\in \Z^ d/n\\s+t+\mu(x-y-\mathbf{e}_1)\ge (1+\xi)\mu(\mathbf{e}_1)}}\kA_{s,x}(n)\cap(n\mathbf{e}_1+\kA_{t,y}(n))\right)\\
        &\leq -\inf_{\substack{x,y\in\R^d,s,t\ge 0:\\s+t+\mu(x-y-\mathbf{e}_1)\geq (1+\xi)\mu(\mathbf{e}_1)}} (I(s,x)+I(t,y)).
    \end{split}
\end{equation*}
Moreover, by \eqref{eq:propI} and $I(t,y)=I(t,-y)$,
\al{\inf_{\substack{x,y\in\R^d,s,t\ge 0:\\s+t+\mu(x-y-\mathbf{e}_1)\geq (1+\xi)\mu(\mathbf{e}_1)}}I(s,x)+I(t,y)&\geq \inf_{\substack{x,y\in\R^d,s,t\ge 0:\\s+t+\mu(x-y-\mathbf{e}_1)\geq (1+\xi)\mu(\mathbf{e}_1)}} I(s+t,x-y)\\
&= \inf_{\substack{x\in\R^d,s\ge 0:\\s+\mu(x-\mathbf{e}_1)\ge (1+\xi)\mu(\mathbf{e}_1)}}I(s,x)=J(\xi).
}
This yields the result.
\end{proof}

\appendix
\section{Combinatorial Lemmas}

Let us denote by $\mathfrak{S}_{m}$  the set of all the  permutations of $[m]$. Recall that $[x,y]$ stands for the segment between $x$ and $y$.
\begin{proof}[Proof of Lemma \ref{lem:dislines}] Let us prove the result in the case $\ell=K$. The result for general $\ell$ follows by dilating the space by a factor $\ell/K$ in the $\mathbf e_1$ direction.

Let $S_1=\{x^1,\dots,x^{m}\}$ and $S_2=\{y^1,\dots,y^{m}\}$ be such that $\max_{x\in S_1,\,y\in S_2}\|x-y\|_\infty\leq K$. We will find a permutation $\sigma\in \mathfrak{S}_{m}$ such that $${\rm d}_2([x^i,y^{\sigma(i)}],[x^j,y^{\sigma(j)}])\geq \frac{1}{\sqrt{2}},\qquad\forall i\neq j\in [m].$$
For $k,l\in [d]$, denote by $\pi_{k,l}$ the projection on the plan spanned by $\mathbf{e}_k$ and $\mathbf{e}_l$.
Let $\sigma\in\mathfrak{S}_{m}$ be such that $\sigma$ minimizes
\[\sum_{i=1}^ {m}\sum_{k=2}^d\|\pi_{1,k}(x^i-y^{\sigma(i)})\|_2.\]
 Denote by $L_i$ the segment joining $x^i$ and $y^{\sigma(i)}$. And denote by $v^i$ the unit vector associated to the direction of the segment $L_i$.
We claim $L_i\cap L_j= \emptyset$ for all $i\neq j$. To this end, we first suppose that $L_i\cap L_j\ne \emptyset$ for some $i\neq j\in [m]$, and we shall derive a contradiction. We note that for all $k\ne 1$, we have $\pi_{1,k}(L_i)\cap \pi_{1,k}(L_j)\ne \emptyset$, so we can take $z\in \pi_{1,k}(L_i)\cap \pi_{1,k}(L_j)$. { By the triangular inequality,
\begin{equation}\label{eq:triangineq}
\begin{split}
   &  \|\pi_{1,k}(x^i- y^{\sigma(j)})\|_2+\|\pi_{1,k}(x^j- y^{\sigma(i)})\|_2\\
  &\leq  \|\pi_{1,k}(x^i)- z\|_2+\|z-\pi_{1,k}(y^{\sigma(j)})\|_2+\|\pi_{1,k}(x^j)- z\|_2+\|z-\pi_{1,k}(y^{\sigma(i)})\|_2\\&= \|\pi_{1,k}(x^i- y^{\sigma(i)})\|_2+\|\pi_{1,k}(x^j- y^{\sigma(j)})\|_2.
    \end{split}
\end{equation}
Moreover, since $v^i,v^j$ are not colinear due to $L_i\cap L_j\ne \emptyset$, there exists $k\in \{2,\cdots,d\}$ such that $\pi_{1,k}(v^i)$ is not colinear to $\pi_{1,k}(v^j)$, where $\pi_{1,k}(L_i)$ and  $\pi_{1,k}(L_j)$ intersect only at one point. Then, the inequality in \eqref{eq:triangineq}  becomes strict, i.e.  $$\|\pi_{1,k}(x^i- y^{\sigma(j)})\|_2+\|\pi_{1,k}(x^j- y^{\sigma(i)})\|_2< \|\pi_{1,k}(x^i- y^{\sigma(i)})\|_2+\|\pi_{1,k}(x^j- y^{\sigma(j)})\|_2.$$
}
Therefore, we arrive at 
\[\sum_{k=2}^d\|\pi_{1,k}(x^i-y^{\sigma(j)})\|_2+\sum_{k=2}^d\|\pi_{1,k}(x^j-y^{\sigma(i)})\|_2<\sum_{k=2}^d\|\pi_{1,k}(x^i-y^{\sigma(i)})\|_2+\sum_{k=2}^d\|\pi_{1,k}(x^j-y^{\sigma(j)})\|_2,\]
which contradicts that $\sigma$ is a minimizer. Therefore, we have  $L_i\cap L_j= \emptyset$ for all $i\neq j$.

Let us next assume that for all $k\ne 1$ the projection of the segments $\pi_{1,k}(L_i)$ and $\pi_{1,k}(L_j)$ intersect. By similar reasoning as above, $\pi_{1,k}(v^i)$ and $\pi_{1,k}(v^j)$ are colinear for all $k\ne 1$, since otherwise it contradicts the minimality of $\sigma$. Hence $v^i$ and $v^j$ are colinear. %Though, since the lines $L_i$ and $L_j$ do not intersect, for $k$ such that $(x^i-x_j^1)\cdot \mathbf{e}_k\ne 0$,
{However, since  $x^i\neq x^j$ and $x^i_1=x^j_1$, there exists  $k\in\{2,\cdots,d\}$ such that the $k$-th coordinates of $x^i$ and $x^j$ do not coincide, i.e. $x^i_k\neq x^j_k$. Thus, $\pi_{1,k}(L_i)$ and $\pi_{1,k}(L_j)$ do not intersect, which is a contradiction. Therefore, there exists $k\ne 1$ such that $\pi_{1,k}(L_i)$ and $\pi_{1,k}(L_j)$ do not intersect.

 Let us now compute the distance between $L_i$ and $L_j$.
 Since the minimal distance of non-intersecting segments in $\R^2$ is attained at one of the endpoints, without loss of generality, we can assume  that
\[\mathrm{d}_2(\pi_{1,k}(L_i),\pi_{1,k}(L_j))= \mathrm{d}_2(\pi_{1,k}(x^i),\pi_{1,k}(L_j)).\]
Denote by $\mathbf{u}$ a unit vector in $\R^2$ orthogonal to $\pi_{1,k}(L_j).$ % =\frac{\pi_{1,k}(x^j-y^{\sigma(j)})}{\|\pi_{1,k}(x^j-y^{\sigma(j)})\|_2}$. B
By  $\max_{x\in S_1,\,y\in S_2}\|x-y\|_\infty\leq K$, we have 
 $|\mathbf{u}\cdot \mathbf{e}_1|\le\frac{1}{\sqrt{2}}\le |\mathbf{u}\cdot \mathbf{e}_k|$.  
Besides, since $\pi_{1,k}(L_i)$ and $\pi_{1,k}(L_j)$ do not intersect, $x^i,x^j\in \Z^d$, and $x_1^i=x_1^j$, $|x^i_k-x^j_k|\geq 1$. Thus, we have
\al{
\mathrm{d}_2({ \pi_{1,k}(x^i)},\pi_{1,k}(L_j))&= |\pi_{1,k}(x^i-x^j)\cdot \mathbf{u}|\\
&\ge |\pi_{1,k}(x^i-x^j)\cdot \mathbf{e}_k| |\mathbf{u}\cdot \mathbf{e}_k|\\
&=|x^i_k-x^j_k||\mathbf{u}\cdot \mathbf{e}_k|\ge\frac{1}{\sqrt{2}}.
}}
Therefore, we conclude
\[\mathrm{d}_2(L_i,L_j)\ge \mathrm{d}_2(\pi_{1,k}(L_i),\pi_{1,k}(L_j)) \ge \frac{1}{\sqrt{2}}\,.\]
\end{proof}

\begin{proof} [Proof of Lemma~\ref{lem:disjpathscont}] By taking a subset, without loss of generality, we assume $m:=|S_1|=|S_2|$.  
\noindent{\bf Case 1.}
Let us first study the case where $i=j$. We write 
 \[S_1:=\{x^1,\dots,x^{m}\}\qquad \text{and}\qquad S_2:=\{y^1,\dots,y^{m}\}.\]
 We can assume without loss of generality that $i=1$.  Consider $\sigma\in \mathfrak S_{m}$ as in Lemma \ref{lem:dislines} applied to the sets $S_1$ and $S_2$. In particular, for $i\neq j$, we have
\[\mathrm{d}_2([x^i,y^{\sigma(i)}], [x^j,y^{\sigma(j)}]) \ge \frac{\ell}{K\sqrt{2}}.\]
It is easy to check that there exists a $\Z^d$-path $\fp_i $ from $x^i$ and $y^{ \sigma(i)}$ of length at most $2dK$ included in   \[\{x\in\Z^d:[x^i,y^{ \sigma(i)}]\cap (x+[-1,1]^ d)\ne \emptyset\} .\]
We claim that 
\[\forall x\in\Z^ d\quad\#\{i\in[m]:x\in \fp_i\}\le (2d)^{2d}{\left(\frac K \ell\right)^{d-1}}.\]
To this end, fix $x\in\Z^ d$. We have
\[\#\{i\in[m]:x\in \fp_i\}\le \#\{i\in[m]:[x^i,y^{ \sigma(i)}]\cap (x+ [-1,1]^ d)\ne \emptyset\}.\]
Since the minimal distance between two lines is at least $\frac{\ell}{K\sqrt{2}}$, the number of lines intersecting the cube  $x+[-1,1]^ d$ is at most $(2d)^{2d}\left(\frac K \ell\right)^{d-1}$.

\noindent{\bf Case 2.} Let us now assume that $i\ne j$. Let $K\ge 1$ such that $S_1\cup S_2\subset [-K,K]^d$. Let $m':=\lceil m/2\rceil$ where we assumed $m:=|S_1|=|S_2|$. By reflection symmetry, without loss of generality, we can assume that 
\[\#\{x\in S_1: x\cdot \mathbf{e}_j\ge 0\}\ge \frac{|S_1|}{2},\quad\text{ 
and }\quad
\#\{x\in S_2: x\cdot \mathbf{e}_i\ge 0\}\ge \frac{|S_2|}{2}.\] 
Hence, we can take  $S_1^+\subset \{x\in S_1: x\cdot \mathbf{e}_i\ge 0\}$ and $S_2^+\subset \{x\in S_2: x\cdot \mathbf{e}_i\ge 0\}$ such that 
\[S_1^+=\{x^1,\dots,x^{m'}\},\qquad \text{and}\qquad S_2^+=\{z^1,\cdots,z^{m'}\}.\]
Set $y^i$ to be the intersection between $\rH_i(K\mathbf{e}_i)$ and the line passing through $y^i\in S_2^+$ and of direction $\mathbf{e}_i+\mathbf{e}_j$. 
Denote  $\widetilde{S}_2:=\{y^1,\cdots,y^{m'}\}$.
One can check that 
\[\max_{x\in S_1^+, y\in \widetilde{S}_2}\|x-y\|_\infty \le { 4}K.\]
 We can apply Lemma \ref{lem:dislines}, to find a matching $\sigma\in\mathfrak{S}_{m'}$ such that the corresponding segments are at distance  at least $1/({4}\sqrt{2})$ from each other.
We find $\Z^d$-paths $(\fp_i^1)_{i=1}^{m'}$ joining $z^i\in S_2^+$ to $y^i$ such that $\fp_i^1\subset\{x^i+t(\mathbf{e}_i+\mathbf{e}_j):~t\in\R\}$, which implies that each vertex is crossed by at most $4^d$ paths.
By a similar argument as in Case 1, we find  $\Z^d$-paths $(\fp_i^2)_{i=1}^{m'}$ joining $x^i$ to $y^{\sigma(i)}$ such that each vertex is crossed by at least $(2d)^{2d}$ paths.
We can obtain a  path going from $x^i$ to $y^{\sigma(i)}$ by considering  concatenation of $\fp_i^2$ and $\fp_{\sigma(i)}^1$. This concludes the proof.

\end{proof}

\begin{proof}[Proof of Lemma \ref{lem:proj}]
Our goal is to prove that there exists $i\in [d]$ such that $P_i(S)\geq |S|^{2/3}/2$.

Let $n=|S|$. If $|P_1(S)|\ge n^{2/3}/2$, then the claim holds with $i=1$. Hence,  we assume $|P_1(S)|< n^{2/3}/2$. Define $S_1:=P_1(S)$. Since
\[\sum_{z\in S_1}|P_1^ {-1}(z)\cap S|\mathbf{1}_{|P_1^ {-1}(z)\cap S|\le n^ {1/3}}\le |S_1|n^ {1/3}< \frac{n}{2},\]
we have
\[\sum_{z\in S_1}|P_1^ {-1}(z)\cap S|\mathbf{1}_{|P_1^ {-1}(z)\cap S|\ge n^ {1/3}}\ge \frac{n}{2}.\]
Set
\[S_1':=\{z\in S_1: |P_1^ {-1}(z)\cap S|\ge n^ {1/3}\},\quad S_2:=\{P_2(z): z\in S_1'\}.\]
Let us first assume that  $|S_2|\ge n^{1/3}.$ Then, there exist $z^1,\dots,z^m\in S_1'$ with $m\ge n^{1/3}$ such that $(P_2(z^i))_{i\in[m]}$ are all distinct. Hence, $(P_2(x),~x\in \cup_{i=1}^m P_1^ {-1}(z^i)\cap S)$ are also all distinct. It follows that
\[|P_2(S)|\ge \left|\bigcup_{i=1}^m P_1^ {-1}(z^i)\cap S\right|=\sum_{i=1}^ m|P_1^ {-1}(z^i)\cap S|\ge n^{2/3}.\]
Otherwise, if $|S_2|<n^{1/3}$, then 
\begin{equation*}
    \begin{split}
        \sum_{z'\in S_2}\#\{w\in S: \,P_2\circ P_1(w)=z'\}&\geq   \sum_{z'\in S_2} \sum_{z\in S_1'}\#\{w\in S: \,z=P_1(w),\,P_2(z)=z'\}\\&= \sum_{z\in S_1}|P_1^ {-1}(z)\cap S|\mathbf{1}_{|P_1^ {-1}(z)\cap S|\ge n^ {1/3}}\ge \frac{n}{2}.
    \end{split}
\end{equation*}
Hence, by pigeon-hole principle, there exists $z\in S_2$ such that
\[\#\{w\in S: \,P_2\circ P_1(w)=z\}\ge \frac{n}{2|S_2|}> \frac{n^{2/3}}{2},\]
which implies $|P_3(S)|\ge  \frac{n^{2/3}}{2}.$ 
This concludes the proof.
\end{proof}

\begin{proof}[Proof of Lemma \ref{lem:4.2}] 
Recall $m(d,S)=\left(\frac{|S|}{2^{d-1}\Diam(S)}\right)^{\frac{1}{d-1}}.$ {We will find, by induction on dimension, $i\neq j\in[d]$,  $S'\subset S$ with $|S'|\ge m(d,S)$ such that $z_i\neq z'_i$ and $z_j\neq z'_j$ for any $z\neq z'\in S'$.}% We will prove the result .

Let us start by proving the result for $d=2$. Let $S\subset \Z^2$.
Given $u=(u_1,u_2)\in S$, set \[\rL'(u):=\rL_1(u)\cup \rL_2(u).\]
Now, we construct $u^k$ as follows:  Set $u^1=u$ for some arbitrary $u\in S$. Suppose that $u^1,\cdots,u^\ell$ have been defined. If there exists $u\in S \backslash \bigcup_{l=1}^\ell \rL'(u^l)$, then we set $u^{\ell+1}=u$; otherwise, we stop this procedure and set $N=\ell$. Since the set $S$ is finite, this procedure will eventually stop. Since $|S\cap \rL'(u^k)|\leq 2\Diam{(S)},$
\[2\Diam{(S)} N\ge |S|,\]
so \[N\ge \frac{|S|}{2\Diam(S)}=m(2,S).\]
The family $S'=(u^k)_{k=1}^N$ satisfies the requirements. The proof is completed.

Let us now assume that the claim holds for $d-1\ge 2$.
Let $S\subset \Z^d$.
 Given $z\in S$, set \[\rL(z):=\{w\in\Z^d: w_{d-1}=z_{d-1}\}\cup \{w\in\Z^d: w_d=z_d\}.\] We define $v^k$ recursively as follows: Set $v^1=z$ for some arbitrary $z\in S$. Suppose that $v^1,\cdots,v^\ell$ have been defined. If there exists $z\in S \backslash \bigcup_{l=1}^\ell \rL(v^l)$, then we set $v^{\ell+1}=z$; otherwise, we stop this procedure and set $N=\ell$. Since the set $S$ is finite, this procedure will eventually stop.
If $N\geq m(d,S)$, then the proof is completed with $i=d-1,j=d$ and $S'=(v_k)_{k=1}^N$. Otherwise, if $N<m(d,S)$, then by pigeon-hole principle, there exists $k\in [N]$ such that $$|\{x\in S:~x\in \rL(v^k)\}|\geq \frac{|S|}{m(d,S)}.$$
By pigeon-hole principle, there exists $r\in\{d-1,d\}$ such that
$\#\{x\in S: x_r=v^k_r\}\geq \frac{|S|}{2m(d,S)}.$ Without loss of generality, we suppose $r=d$, i.e. \[\#\{x\in S: x_d=v^k_d\}\geq \frac{|S|}{2m(d,S)}.\] 
It follows that 
the set $\bar{S}:=\{(x_1,\dots,x_{d-1}):x\in S\}$ is of size at least $|S|/(2m(d,S)).$
Applying the induction hypothesis to the set $\bar{S}\subset \Z^{d-1}$, we find $i\neq j\in [d-1]$ and $\bar{S}'\subset \bar{S}$ such that for all $y\ne y'\in \bar{S}'$, $y_i\ne y'_i$ and $y_j\ne y'_j$, and   
\al{|\bar{S}'|&\ge \left(\frac{|\bar{S}|}{2^{d-2}\Diam(\bar{S})}\right)^{1/(d-2)}\\
&\ge \left(\frac{|S|}{2^{d-1}\Diam(S)\,m(d,S)}\right)^{1/(d-2)}\\
&\ge \left(\frac{|S|}{2^{d-1}\Diam(S)}\right)^{\frac{1}{d-2}\left(1-\frac{1}{d-1}\right)}=m(d,S).
}
We take $u^y\in S$ for $y\in \bar{S}'$  such that $(u^y_1,\dots,u^{y}_{d-1})=y$. The result follows with the family $S':=(u^{y})_{y\in \bar{S}'}$ and $i,j$ chosen above. This completes the induction.

\end{proof}

\bibliographystyle{plain}
\bibliography{biblio}
\end{document}